%% file: main.tex
\documentclass[journal,comsoc]{IEEEtran}
\usepackage[T1]{fontenc}% optional T1 font encoding

\ifCLASSINFOpdf
\else
\fi
\usepackage{amsmath}
\usepackage[cmintegrals]{newtxmath}

\input{macros_ieee.tex}
\usepackage{proof-at-the-end}
\usepackage{xparse}
\usepackage{multicol}
\pgfkeys{/prAtEnd/global custom defaults/.style={
     text link={~\hyperref[proof:prAtEnd\pratendcountercurrent]},
     text proof={of \string\autoref{thm:prAtEnd\pratendcountercurrent}}
	}
}

\newtheorem{assumption}{Assumption}
\newtheorem{lemma}{Lemma}
\newtheorem{remark}{Remark}
\newtheorem{definition}{Definition}
\newtheorem{fact}{Fact}

\newcommand{\textfrac}[2]{\textstyle\frac{#1}{#2}}

\begin{document}
%
% paper title
% Titles are generally capitalized except for words such as a, an, and, as,
% at, but, by, for, in, nor, of, on, or, the, to and up, which are usually
% not capitalized unless they are the first or last word of the title.
% Linebreaks \\ can be used within to get better formatting as desired.
% Do not put math or special symbols in the title.
\title{A Decentralized Primal-Dual Framework for Non-convex Smooth Consensus Optimization}
%
%
% author names and IEEE memberships
% note positions of commas and nonbreaking spaces ( ~ ) LaTeX will not break
% a structure at a ~ so this keeps an author's name from being broken across
% two lines.
% use \thanks{} to gain access to the first footnote area
% a separate \thanks must be used for each paragraph as LaTeX2e's \thanks
% was not built to handle multiple paragraphs
%

\author{Gabriel~Mancino-Ball,~\IEEEmembership{}
        Yangyang~Xu,~\IEEEmembership{}
        and~Jie~Chen~\IEEEmembership{}% <-this % stops a space

\thanks{Gabriel Mancino-Ball and Yangyang Xu are with Department of Mathematical Sciences, Rensselaer Polytechnic Institute, Troy, NY, 12180, USA (e-mail: mancig@rpi.edu; xuy21@rpi.edu).}% <-this % stops a space
\thanks{Jie Chen is with MIT-IBM Watson AI Lab, IBM, Cambridge, MA, 02142, USA (email: chenjie@us.ibm.com).}% <-this % stops a space
\thanks{This work was partly supported by the Rensselaer-IBM AI Research Collaboration, part of the IBM AI Horizons Network. Y. Xu is supported in part by NSF grants DMS-2053493 and DMS-2208394 and the ONR award N00014-22-1-2573. J. Chen is supported in part by DOE Award DE-OE0000910.}%
%\thanks{Manuscript received April 19, 2005; revised August 26, 2015.}
}

% note the % following the last \IEEEmembership and also \thanks - 
% these prevent an unwanted space from occurring between the last author name
% and the end of the author line. i.e., if you had this:
% 
% \author{....lastname \thanks{...} \thanks{...} }
%                     ^------------^------------^----Do not want these spaces!
%
% a space would be appended to the last name and could cause every name on that
% line to be shifted left slightly. This is one of those "LaTeX things". For
% instance, "\textbf{A} \textbf{B}" will typeset as "A B" not "AB". To get
% "AB" then you have to do: "\textbf{A}\textbf{B}"
% \thanks is no different in this regard, so shield the last } of each \thanks
% that ends a line with a % and do not let a space in before the next \thanks.
% Spaces after \IEEEmembership other than the last one are OK (and needed) as
% you are supposed to have spaces between the names. For what it is worth,
% this is a minor point as most people would not even notice if the said evil
% space somehow managed to creep in.

% The paper headers
\markboth{Primal-dual framework for decentralized non-convex optimization}%
{Shell \MakeLowercase{\emph{et al.}}: Bare Demo of IEEEtran.cls for IEEE Communications Society Journals}
% The only time the second header will appear is for the odd numbered pages
% after the title page when using the twoside option.
% 
% *** Note that you probably will NOT want to include the author's ***
% *** name in the headers of peer review papers.                   ***
% You can use \ifCLASSOPTIONpeerreview for conditional compilation here if
% you desire.

% If you want to put a publisher's ID mark on the page you can do it like
% this:
%\IEEEpubid{0000--0000/00\$00.00~\copyright~2015 IEEE}
% Remember, if you use this you must call \IEEEpubidadjcol in the second
% column for its text to clear the IEEEpubid mark.

% use for special paper notices
%\IEEEspecialpapernotice{(Invited Paper)}

% make the title area
\maketitle

% As a general rule, do not put math, special symbols or citations
% in the abstract or keywords.
\begin{abstract}
In this work, we introduce ADAPD,  \textbf{A} \textbf{D}ecentr\textbf{A}lized \textbf{P}rimal-\textbf{D}ual algorithmic framework for solving non-convex and smooth consensus optimization problems over a network of distributed agents. The proposed framework relies on a novel problem formulation that elicits ADMM-type updates, where each agent first inexactly solves a local strongly convex subproblem with any method of its choice and then performs a neighbor communication to update a set of dual variables. We present two variants that allow for a single gradient step for the primal updates or multiple communications for the dual updates, to exploit the tradeoff between the per-iteration cost and the number of iterations. When multiple communications are performed, ADAPD can achieve theoretically optimal communication complexity results for non-convex and smooth consensus problems. Numerical experiments on several applications, including a deep-learning one, demonstrate the superiority of ADAPD over several popularly used decentralized methods.
\end{abstract}

% Note that keywords are not normally used for peerreview papers.

% For peer review papers, you can put extra information on the cover
% page as needed:
% \ifCLASSOPTIONpeerreview
% \begin{center} \bfseries EDICS Category: 3-BBND \end{center}
% \fi
%
% For peerreview papers, this IEEEtran command inserts a page break and
% creates the second title. It will be ignored for other modes.
\IEEEpeerreviewmaketitle

% ------------------------------------------------------------------------ %
% 									Introduction									 %
% ------------------------------------------------------------------------ %

\section{Introduction}\label{sec:intro}

\IEEEPARstart{G}{iven} a set of $N$ agents connected by an undirected network (graph) $\graph=(\vertices,\edges)$, where $\vertices=\{1,\dots,N\}$ denotes the set of agents and $\edges=\{(i,j)\colon\text{agent $i$ is connected to agent $j$}\}$ denotes the set of feasible local communications among agents, consensus optimization methods solve the following problem using only local computation and local communication,
	\begin{equation}\label{dco_problem}
			\small\min_{\vec{x}\in\R^p}f(\vec{x})\defined\textfrac{1}{N}\sum_{i=1}^Nf_i(\vec{x})
	\end{equation}
where each $f_i\colon\R^p\to\R$ is a differentiable, potentially non-convex, cost function known only to agent $i$.

Problem (\ref{dco_problem}) arises naturally in various scientific and engineering applications such as distributed machine learning/federated learning~\cite{mcmahan17,lian17,liang20}, decentralized matrix factorization~\cite{hong17}, network sensing and localization~\cite{chen12,shi15,lorenzo16}, and multi-vehicle coordination~\cite{abichandani11}, to name a few. The decision variable $\vec{x}$ can represent the weights of a neural network~\cite{mcmahan17}, the location of a particular object~\cite{abichandani11}, or the state of a smart grid system~\cite{shi15}, for example. Essentially, any scenario in which data is either too large or naturally distributed fits problem~\eqref{dco_problem}.
% ------------------------------------------------------------------------ %
% 							Problem Formulation		  					   %
% ------------------------------------------------------------------------ %
\subsection{Problem Formulation}
It is well known~\cite{shi15,scaman17} that if $\graph$ is connected, the following problem is equivalent to~\eqref{dco_problem}:
	\begin{equation}\label{dco_problem:reformulation1}
		\small\min_{\x}\enskip F(\x)\enskip\text{subject to }\W\x=\x
	\end{equation}
where $\W$ is a \emph{mixing matrix}~\cite{shi15,yuan16,sun20} that satisfies the conditions in Assumption~\ref{assumption:mixing_matrix} below, and
		\begin{equation}\label{objective}
			\small\x\defined\mymatrix{\vecx_1 & \dots & \vecx_N}^\top\in\R^{N\times p},\enskip F(\x)\defined\textfrac{1}{N}\sum_{i=1}^Nf_i(\vecx_i),
		\end{equation}
is the concatenation of local decision variables and the global objective function, respectively, written in matrix notation. Here, $\vec{x}_i$ is agent $i$'s local copy of the global variable $\vec{x}$, and $\W\in\R^{N\times N}$ represents the connectivity structure of the network $\graph$.
% ------------------------------------------------------------------------ %
% 					Mixing Matrix Assumptions						%
% ------------------------------------------------------------------------ %
\begin{assumption}\label{assumption:mixing_matrix}
	The mixing matrix, $\W\in\R^{N\times N}$, satisfies,
	
	\begin{enumerate}[\itshape(i)]
	
		\item \textbf{(Decentralized property)} $w_{ij}>0$ if $(i,j)\in \edges$, otherwise $w_{ij}=0$,
		
		\item \textbf{(Symmetric property)} $\W=\W^\top,$
	
		\item \textbf{(Null space property)} $\mathrm{null}\left(\identity-\W\right)=\mathrm{span}\{\basis\}$, where $\basis\in\R^N$ is the vector of all ones, and
		
		\item \textbf{(Spectral property)} the eigenvalues of $\W$ lie in the range $(-1,1]$ and can be ordered as
		{\small
			\begin{align*}
				-1<\lambda_N(\W)\le\dots\le\lambda_2(\W)<\lambda_1(\W)=1.
			\end{align*}}
	\end{enumerate}
\end{assumption}

% ------------------------------------------------------------------------ %
% 							Mixing Matrix Options						 %
% ------------------------------------------------------------------------ %
%\subsubsection{On the Choice of Mixing Matrices}

Several common choices for mixing matrices are presented in~\cite{shi15}.

\begin{itemize}
		
	\item[-] \emph{Laplacian-based constant edge weight matrix,}
		\begin{equation}\label{mixing_mat:laplacian}
			\small\W=\identity-\textfrac{\textbf{L}}{\tau}
		\end{equation}
		where $\textbf{L}$ is the Laplacian matrix of $\graph$ and $\tau>\textfrac{1}{2}\lambda_1(\textbf{L}).$ Here, $\lambda_1(\textbf{L})$ is the largest positive eigenvalue of $\textbf{L}.$ If the eigenvalues of $\textbf{L}$ are unknown, by the Gershgorin circle theorem one can use $\tau=\max_{i\in\vertices}\{\abs{\mathcal{N}_i}\}+\epsilon$, for some $\epsilon>0$, where $\mathcal{N}_i\triangleq\{j\colon(i,j)\in\edges\}$ is the set of agents that can communicate with agent $i$.
	
	\item[-] \emph{Metropolis constant edge weight matrix,} for some $\epsilon>0$,
			\begin{equation}\label{mixing_mat:metropolis}
				\small w_{ij}=\begin{cases}\textfrac{1}{\max\{\abs{\mathcal{N}_i},\abs{\mathcal{N}_j}\}+\epsilon},&(i,j)\in\edges,\\0,&(i,j)\not\in\edges\text{ and }i\ne j,\\1-\sum_{k\in\vertices}w_{ik},&i=j.\end{cases}
			\end{equation}

	\item[-] \emph{Symmetric fastest distributed linear averaging matrix}, (FDLA), which is a matrix that achieves the fastest information diffusion through $\graph$ and is obtained by solving a semidefinite program~\cite{xiao04}.
		
\end{itemize}

% This paragraph is moved here for a better flow.
Note that the constraint formulation $\W\x=\x$ in~\eqref{dco_problem:reformulation1} is not the only choice for a consensus problem. Under Assumption~\ref{assumption:mixing_matrix}, an equivalent consensus constraint adopted by others~\cite{wei13,hong17,wang20} is $\vec{x}_i=\vec{x}_j$ for all $(i,j)\in\edges.$ This constraint is an \emph{edge-based} constraint, whereas we consider a \emph{vertex-based} constraint. When a primal-dual approach is designed, if $\graph$ is dense, then a vertex-based constraint introduces fewer dual variables than an edge-based constraint. Further, an optimal $\W$ can be designed, given $\graph$~\cite{xiao04}.

% ------------------------------------------------------------------------ %
% 						Further Reform + Notes							%
% ------------------------------------------------------------------------ %
A vital quantity for our analysis comes from the spectral properties of $\graph$. We define
% ----- SPECTRAL PROPERTIES ----- %
	\begin{equation}\label{spectral_gap}
		\small\rho\triangleq \norm{\W-\textfrac{1}{N}\basis\basis^\top}_2=\max\left\{\abs{\lambda_2\left(\W\right)},\abs{\lambda_N\left(\W\right)}\right\}\in[0,1).
	\end{equation}
The metric in~\eqref{spectral_gap} is one way to measure the connectivity of $\graph$, where $\rho\approx0$ implies good connectivity. %Now, 

Under Assumption~\ref{assumption:mixing_matrix}, particularly $\mathrm{null}\big(\sqrt{\identity-\W}\big)=\mathrm{null}\big(\identity-\W\big)=\mathrm{span}\{\basis\}$, a further equivalent reformulation to~\eqref{dco_problem} is 
% ----- REFORMULATIONS ----- %
	\begin{equation}\label{dco_problem:reformulation2}
		\small\min_{\x}\enskip F(\x)\enskip\text{subject to }\sqrt{\identity-\W}\x=\0,
	\end{equation}
where $\0\in\R^{N\times p}$ is the matrix of all zeros. A benefit of this formulation is that the constraint $\sqrt{\identity-\W}\x=\0$ can now be incorporated into a penalty term, $\textfrac{1}{2\eta}\fronorm{\sqrt{\identity-\W}\x}$, where $\eta>0$ is a penalty parameter. The gradient associated with this term is $\textfrac{1}{\eta}\left(\identity-\W\right)\x$, which can be computed by a single neighbor communication.

One way to solve~\eqref{dco_problem:reformulation2} is to form the augmented Lagrangian with dual variables $\y\defined\mymatrix{\vec{y}_1&\dots&\vec{y}_N}^\top$ and perform a primal-dual type update as in IDEAL~\cite{arjevani20}. The issue with this approach is that the communication and computation phases are inherently coupled, illustrated as follows. The classic primal-dual updates at iteration $k$ used to solve~\eqref{dco_problem:reformulation2} are %given by
	{\small
	\begin{align}
		\x^{k+1}&=\argmin_{\x}F(\x)+\ip{\sqrt{\identity-\w}\y^k,\x}+\textfrac{1}{2\eta}\fronorm{\sqrt{\identity-\w}\x},\nonumber\\
		\sqrt{\identity-\w}\y^{k+1}&=\sqrt{\identity-\w}\y^k+\textfrac{1}{\eta}(\identity-\w)\x^{k+1}.\label{naive_primal_dual}
	\end{align}
    }If a first-order method is used to solve the $\x$ subproblem, then part of the gradient will contain $\w\x$ \emph{at each gradient computation}, thus for every gradient computed, one neighbor communication must be performed.

Another way to solve~\eqref{dco_problem:reformulation2}, as suggested by the Prox-PDA method~\cite{hong17}, is to introduce an additional proximal term of the form $\frac{1}{2\eta}\norm{\x-\x^k}_{\vec{B}^\top\vec{B}}^2$, where $\vec{B}^\top\vec{B}=-\left(\identity-\W\right)+\vec{D}$ with some diagonal matrix $\vec{D}$. This negates the neighbor communication required in the $\x$ subproblem of~\eqref{naive_primal_dual}, but introduces a new parameter $\vec{D}$ that impacts this method's numerical performance.

Interestingly, Prox-PDA with a special choice of $\vec{B}^\top\vec{B}$ recovers the distributed ADMM algorithm~\cite{shi14} for consensus optimization with edge-based constraint; see Appendix~\ref{appendix:prox} for details. Hence, part of this work serves to compare ADMM-type methods derived from using edge-based constraints versus vertex-based constraints for solving problem~\eqref{dco_problem} in a decentralized manner. Our numerical findings in Section~\ref{sec:numerical} indicate that our below derived inexact ADMM gives better performance than distributed ADMM~\cite{shi14}.

To remove the addition of $\vec{D}$ from Prox-PDA, yet still decouple the communication and computation phases of traditional primal-dual methods, we propose adding an extra variable (and constraint), leading to the following formulation:
	\begin{equation}\label{decentralized_problem}
		\small\min_{\x,\x_0}\enskip F(\x)\enskip\text{subject to}\enskip\x=\x_0,\enskip\sqrt{\identity-\W}\x_0=\0.
	\end{equation}
Governed now by two blocks of primal variables, a natural approach to solve~\eqref{decentralized_problem} would be to use an ADMM-type update~\cite{boyd11,wang19}, but as argued in Section~\ref{sec:algo}, the classic ADMM cannot be implemented in a decentralized manner to solve~\eqref{decentralized_problem}. Hence, we are motivated to design a method that: (i) solves~\eqref{decentralized_problem} using only decentralized communication and local gradient computations and (ii) achieves the optimal communication complexity results established in~\cite{sun19}.

We state the technical assumptions on $F$ below.
% ------------------------------------------------------------------------ %
% 				Assumptions on the Objective					  %
% ------------------------------------------------------------------------ %
\begin{assumption}\label{assumption:objective_function}
	The objective function $F$ in (\ref{decentralized_problem}) satisfies:
	\begin{enumerate}[\itshape(i)]
		\item $F$ is $L$-smooth, i.e. there is $0<L<\infty$ such that
			\begin{equation}\label{assumption:smoothness}
				\small\norm{\nabla F(\x)-\nabla F(\y)}_F\le L\norm{\x-\y}_F,\enskip\forall\enskip\x,\y\in\R^{N\times p}.
			\end{equation}
		\item $F$ is lower bounded, i.e. there is $\underline{f}$ such that
			\begin{equation}\label{assumption:lower_bound}
				\small-\infty<\underline{f}\le F(\x),\enskip\forall\enskip\x\in\R^{N\times p}.
			\end{equation}
	\end{enumerate}
\end{assumption}
The gradient of $F$, written in matrix notation, is
	\begin{equation}\label{objective_gradient}
		\small\nabla F(\x)\defined\textfrac{1}{N}\mymatrix{\nabla f_1(\vecx_1)&\dots&\nabla f_N(\vecx_N)}^\top\in\R^{N\times p}.
	\end{equation}

Note that the assumptions (\ref{assumption:smoothness}) and (\ref{assumption:lower_bound}) are standard in non-convex optimization. If each $f_i$ is $L_i$-smooth then $L\ge\max_{i}L_i$ and the lower boundedness assumption is equivalent to the existence of a minimizer of $F$.

Before demonstrating a brief literature review, we state a standard definition~\cite{hong17,sun19,sun20} for stationary points of~\eqref{dco_problem}.
% ----- STATIONARITY VIOLATION ----- %
\begin{definition}[$\boldsymbol{\varepsilon}$-stationary point]\label{def:stationarity}
	A matrix $\x\in\R^{N\times p}$ is called an $\varepsilon$-stationary point of (\ref{dco_problem}) if 
 	\begin{equation}\label{def:stationarit:eqn}
 		\small\norm{\textfrac{1}{N}\sum_{i=1}^N\nabla f_i(\bar{\vec{x}})}_2^2+\fronorm{\x-\bar{\x}}\le\varepsilon
	 \end{equation}
 where $\bar{\vec{x}}\triangleq\textfrac{1}{N}\basis^\top\x$ is the average vector across the $N$ rows of $\x$ and $\bar{\x}\triangleq\textfrac{1}{N}\basis\basis^\top\x$ is a matrix version of this same average.
\end{definition}

% ------------------------------------------------------------------------ %
% 								Related Works									 %
% ------------------------------------------------------------------------ %
\subsection{Related Works}

Distributed computing dates back decades ago to the seminal work~\cite{bertsekas89}. \emph{Centralized} computing paradigms, where $\W=\textfrac{1}{N}\basis\basis^\top$ in~\eqref{dco_problem:reformulation1} have been heavily studied; when each $f_i$ is convex, methods such as ADMM~\cite{boyd11} and FedAVG~\cite{li20} have theoretical convergence guarantees. The focus of this paper is on \emph{decentralized} computing paradigms. Methods such as DGD~\cite{yuan16} and the distributed subgradient method in~\cite{nedic09} have been shown to have sublinear convergence in the convex differentiable and convex non-differentiable settings, respectively. When strong convexity is assumed, the NEAR-DGD~\cite{berahas19} method improved the convergence result of DGD by allowing for multiple communications during each iteration. If $f_i$ has Lipschitz continuous gradient and is strongly convex, ADMM~\cite{shi15} and Acc-DNGD-SC~\cite{qu19} exhibit linear convergence. The EXTRA~\cite{shi15} method also exhibits linear convergence if the global function $f$ is restricted strongly convex\footnote{A convex, differentiable function $h\colon\R^p\to\R$ is restricted strongly convex about a point $\tilde{\vec{x}}$ with parameter $\mu>0$ if $\ip{\nabla h(\vec{x})-\nabla h(\tilde{\vec{x}}),\vec{x}-\tilde{\vec{x}}}\ge\mu\norm{\vec{x}-\tilde{\vec{x}}}_2^2$ for all $\vec{x}\in\R^p.$}. SSDA~\cite{scaman17} and the recent distributed FGM~\cite{urbine21} were designed for convex problems where the gradients of the Fenchel conjugates\footnote{The Fenchel conjugate of a convex function $h\colon\R^p\to\R$ is $h^*(\vec{y})\triangleq\sup_{\vec{x}}\ip{\vec{x},\vec{y}}-h(\vec{x}).$} of the objective functions $f_i$ are computable, with the later providing complexity results when only approximate gradients are computable. IDEAL~\cite{arjevani20} and FlexPD~\cite{mansoori21} are recent primal-dual methods that perform many, or just a few, local neighbor communications per local primal update, respectively.

Of particular interest to us are algorithms dealing explicitly with non-convex local cost functions, e.g. neural networks. When each $f_i$ has Lipschitz continuous gradient, the celebrated DGD~\cite{zeng18} has been shown to converge using diminishing step-sizes with a rate $\bigO{(1-\rho)^{-2}K^{-1}}$, where $\rho$ is defined in~\eqref{spectral_gap} and $K$ is the iteration number. As indicated in the introduction, Prox-PDA~\cite{hong17} is a primal-dual method that is closely related to non-convex ADMM~\cite{hong16} which converges at a rate $\bigO{K^{-1}}$, but has superior numerical performance when compared to DGD. SONATA/NEXT~\cite{lorenzo16,scutari19} is a primal method that exhibits the same convergence rate as Prox-PDA and incorporated a potentially non-smooth but convex regularizer into the objective. While SONATA is applicable to a larger class of problems, it needs to take step-sizes proportional to $N^{-1}$ for convergence; as $N\to\infty$, SONATA's performance can suffer because of this requirement. If the Chebyshev communication protocol~\cite{auzinger11} is used, SONATA additionally can achieve the $\bigO{(1-\rho)^{-0.5}K^{-1}}$ rate, but SONATA must communicate two variables for every algorithm update. Both Prox-PDA and SONATA require agents to solve a local strongly convex subproblem. Our proposed framework can achieve a convergence rate of $\bigO{(1-\rho)^{-0.5}K^{-1}}$ when multiple neighbor communications are performed; note this is optimal for the class of smooth nonconvex problems~\cite{sun19}.

Methods that use stochastic gradients have also been heavily studied. Adapting DGD to  stochastic updates yields D-PSGD~\cite{lian17} which is shown to have a convergence rate of $\bigO{K^{-0.5}}$.
Recent works such as $D^2$~\cite{tang18}, DSGT~\cite{zhang20gradtrack}, and D-GET~\cite{sun20} make use of stochastic gradient updates mixed with a gradient tracking scheme and draw inspiration from their non-stochastic and centralized counterparts~\cite{shi15,nguyen17}.  $D^2$ improves the convergence of D-PSGD, but requires more restrictive assumptions on the eigenvalues of $\W$. The convergence rate of DSGT was shown to be $\bigO{K^{-0.5}+(1-\rho)^{-3}K^{-1}}$ in~\cite{zhang20gradtrack} and later improved to $\tilde{\mathcal{O}}\left(K^{-0.5}+(1-\rho^2)^{-1}K^{-1}\right)$ in~\cite{koloskova21}. D-GET is able to achieve a rate $\bigO{K^{-1}}$ but requires a full gradient computation every few iterations; GT-SARAH~\cite{xin21sarah} achieves the same rate but removes the full gradient computation. The authors in~\cite{yi20} develop a primal-dual method with convergence rate $\bigO{K^{-0.5}},$ where each agent computes one local stochastic gradient per update. The recent SPPDM~\cite{wang21} can also achieve a stochastic $\varepsilon$-stationary point in $\bigO{\varepsilon^{-1}}$ iterations using stochastic gradients and incorporates a potentially non-smooth but convex regularizer into the objective; SPPDM requires a mini-batch of size $\Omega\left(\varepsilon^{-1}\right)$ to achieve this rate. We remark that our framework can exhibit the optimal convergence rate when \emph{deterministic} gradients are used, yet we include relevant decentralized stochastic methods here for sake of completeness.

Additional algorithms to consider are asynchronous algorithms that do not require a synchronous communication step and algorithms that use time-varying mixing matrices and/or mixing matrices that do not satisfy Assumption~\ref{assumption:mixing_matrix}. Some prominent asynchronous algorithms include AD-PSGD~\cite{lian17_async}, the Asynchronous Primal-Dual method in~\cite{wu17}, APPG~\cite{zhang21}, and the asynchronous ADMM~\cite{wei13, hong17admm}. Algorithms that handle different network structures from those considered here have also been considered: Push-Pull~\cite{pu20} handles directed graphs and DIGing~\cite{nedic17} is a gradient tracking algorithm that works for network structures where $\W$ changes at every iteration. While these scenarios are certainly interesting, we focus on synchronous updates and undirected graphs.

% ------------------------------------------------------------------------ %
% 									Contributions									 %
% ------------------------------------------------------------------------ %
\subsection{Summary of Contributions}

Our main contributions are listed below:

\begin{itemize}

	\item[-] We motivate the novel problem formulation of~\eqref{decentralized_problem} for solving the non-convex and smooth decentralized consensus optimization problem. We propose ADAPD,  \textbf{A} \textbf{D}ecentr\textbf{A}lized \textbf{P}rimal-\textbf{D}ual algorithmic framework for solving such problem.  Our framework is based on performing inexact ADMM-type updates by the augmented Lagrangian function of problem~\eqref{decentralized_problem}. Two variants to our framework: ADAPD-OG (ADAPD-\textbf{O}ne \textbf{G}radient) and ADAPD-MC (ADAPD-\textbf{M}ultiple \textbf{C}ommunications) are presented. ADAPD-OG performs a single gradient step instead of inexactly solving a local strongly convex subproblem. ADAPD-MC allows each agent to communicate multiple times with their neighbors for each update. These variants allow for agents to optimize the balance between performing local computation and local communication.
	
	\item[-] We prove that ADAPD and ADAPD-OG converge to an $\varepsilon$-stationary point, see (\ref{def:stationarit:eqn}), in $\bigO{L(1-\rho)^{-2}\varepsilon^{-1}}$ neighbor communications. When the MC variant is used, this complexity is reduced to $\bigO{L(1-\rho)^{-0.5}\varepsilon^{-1}}$, which is optimal for smooth, non-convex consensus optimization problems~\cite{sun19}. For both ADAPD and ADAPD-OG, a key ingredient of our analysis is defining a Lyapunov function that decreases with every iteration.
	 
	 \item[-] We compare ADAPD on several non-convex problems to state-of-the-art methods such as DGD~\cite{zeng18}, Prox-PDA~\cite{hong17}, D-PSGD~\cite{lian17}, DSGT~\cite{zhang20gradtrack}, D-GET~\cite{sun20}, and SPPDM~\cite{wang21}. Four experiments are conducted in total: two using deterministic gradients and two using stochastic gradients. In all cases, ADAPD demonstrates numerical superiority over these popularly used methods.

\end{itemize}

% ------------------------------------------------------------------------ %
% 										Notation										 %
% ------------------------------------------------------------------------ %
\subsection{Notation}

We use bold face letters such as $\x$ and $\vec{x}$ to denote a matrix and a vector, respectively. Let $x_{ij}$ denote the element in the $i^{th}$ row and $j^{th}$ column of  the matrix $\x.$ The Frobenius norm of a matrix is denoted $\norm{\cdot}_F$, while the Euclidean norm of a vector is denoted $\norm{\cdot}_2.$ Define the standard matrix inner product of $\vec{A},\vec{B}\in\R^{N\times p}$ to be $\ip{\vec{A},\vec{B}}\triangleq\sum_{i=1}^N\sum_{j=1}^pa_{ij}b_{ij}.$ For a given symmetric matrix $\vec{U}\in\R^{N\times N},$ we denote $\norm{\vec{A}}_{\vec{U}}^2\triangleq\ip{\vec{A},\vec{U}\vec{A}}.$ If $\vec{U}$ is positive definite, then $\norm{\vec{A}}_{\vec{U}}^2$ defines a norm. For square matrices $\vec{A},\vec{B}\in\R^{N\times N},$ define the matrix inequality $\vec{A}\preccurlyeq\vec{B}$ to hold if and only if $\vec{B}-\vec{A}$ is positive semi-definite.

% ------------------------------------------------------------------------ %
% 				Proposed Algorithmic Framework				%
% ------------------------------------------------------------------------ %
\section{ADAPD Framework}\label{sec:algo}

To solve (\ref{decentralized_problem}), we employ the augmented Lagrangian function with penalty parameter $0<\eta<\textfrac{1}{L}$, which is
{\small
	\begin{equation}\label{augmentedLagrangian}
	\begin{split}
		&\L_{\eta}(\x,\x_0;\y,\z)=F(\x)+\ip{\y,\x-\x_0}+\textfrac{1}{2\eta}\fronorm{\x-\x_0}\\
		&\quad+\ip{\z,\sqrt{\identity-\W}\x_0}+\textfrac{1}{2\eta}\fronorm{\sqrt{\identity-\W}\x_0}
	\end{split}
	\end{equation}
}with dual variables
{\small
	\begin{equation}\label{dualvariables}
		\y\defined\mymatrix{\vec{y}_1&\dots&\vec{y}_N}^\top,\enskip\z\defined\mymatrix{\vec{z}_1&\dots&\vec{z}_N}^\top\in\R^{N\times p}.
	\end{equation}
}The classic ADMM~\cite{boyd11} approach for solving (\ref{decentralized_problem}) performs the following updates using \eqref{augmentedLagrangian}:
{\small
	\begin{align}
			\x^{k+1}&=\argmin_{\x}\L_\eta(\x,\x_0^{k};\y^k,\z^k)\label{x_update_classic}\\
			\x_0^{k+1}&=\argmin_{\x_0}\L_\eta(\x^{k+1},\x_0;\y^k,\z^k)\label{x0_update_classic}\\
			\y^{k+1}&=\y^k+\beta_1\left(\x^{k+1}-\x_0^{k+1}\right)\label{y_update_classic}\\
			\z^{k+1}&=\z^{k}+\beta_2\sqrt{\identity-\W}\x_0^{k+1}\label{z_update_classic}
	\end{align}
}where $\beta_1,\beta_2>0$ are the step-sizes for the dual variables.

Notice that in practice, the exact minimizer of (\ref{x_update_classic}) is difficult to find; thus a natural alternative to (\ref{x_update_classic}) would be to perform an \emph{inexact} update to the local decision variable as in~\cite{eckstein18,kumar19}. This would lead to a computationally efficient way to solve the local subproblem \eqref{x_update_classic} that fully utilizes local computing power without overburdening the agents.

Further, notice that the optimal solution to (\ref{x0_update_classic}) involves solving %the linear system 
\begin{equation}\label{eq:exact-update-X0}
\textfrac{1}{\eta}\big(2\identity-\W\big)\x_0=\textfrac{1}{\eta}\x^{k+1}+\y^k-\sqrt{\identity-\W}\z^k.
\end{equation} 
It should be stated that $\big(2\identity-\W\big)^{-1}$ exists (by Assumption~\ref{assumption:mixing_matrix}(iv)). However, it is not easy to solve in a decentralized setting: since $\W$ is not diagonal, solving (\ref{x0_update_classic}) would involve another iterative method (e.g. Jacobi method), which may require many communications to find the exact minimizer. To remedy this, we note that~\eqref{eq:exact-update-X0} is a linear equation and apply a simple split for the unknown $\x_0$:
\begin{equation}\label{x0_update_pre}
\textfrac{1}{\eta}2\x_0^{k+1}-\textfrac{1}{\eta}\W\x_0^k=\textfrac{1}{\eta}\x^{k+1}+\y^k-\sqrt{\identity-\W}\z^k.
\end{equation}%
Remarkably, such a rough estimate for solving the $\x_0$ subproblem based on the past iterate $\x_0^k$ still guarantees convergence, based on the following intuition. Let $\widehat{\x}_0^{k+1}$ be the solution of \eqref{eq:exact-update-X0}. Then the one gradient step in \eqref{x0_update_pre} replaces the unknown term $\W\widehat{\x}_0^{k+1}$ by $\W\x_0^k$. Our analysis will show that $\x_0^{k+1}-\x_0^k \to \mathbf{0}$. Hence, the one-step gradient descent update will become a close approximation to the exact update, and thus it can still guarantee convergence.

Additionally, the $\z$ update in (\ref{z_update_classic}) cannot be implemented in a decentralized manner, as $\sqrt{\identity-\W}$ in general will not preserve the underlying network topology. However, notice that if $\z^0\in\mathrm{range}(\sqrt{\identity-\W}),$ then $\z^k\in\mathrm{range}(\sqrt{\identity-\W})$, for all $k\ge0$ from \eqref{z_update_classic}. Hence, we can multiply $\sqrt{\identity-\W}$ to the left of all terms in \eqref{z_update_classic} and obtain the equivalent update
\begin{equation}\label{z_update_pre}
\sqrt{\identity-\W}\z^{k+1}=\sqrt{\identity-\W}\z^k+\textfrac{1}{\eta}(\identity-\W)\x_0^{k+1}.
\end{equation}
Doing so allows us to use $\tilde{\z}^k=\sqrt{\identity-\W}\z^k$  to simplify all relevant terms in~\eqref{x0_update_pre} and~\eqref{z_update_pre}.

To summarize, defining $\beta_1=\beta_2=\textfrac{1}{\eta}$, we propose the following modifications to (\ref{x_update_classic})-(\ref{z_update_classic}):
% ------------------------------------------------------------------------ %
% 										Our Updates									 %
% ------------------------------------------------------------------------ %
{\small
	\begin{align}
		\x^{k+1}&\approx\argmin_{\x}\L_\eta(\x,\x_0^{k};\y^k,\z^k)\label{x_update}\\
		\x_0^{k+1}&=\textfrac{1}{2}\left(\W\x_0^k+\x^{k+1}+\eta\big(\y^k-\tilde{\z}^k\big)\right)\label{x0_update}\\
		\y^{k+1}&=\y^k+\textfrac{1}{\eta}\left(\x^{k+1}-\x_0^{k+1}\right)\label{y_update}\\
		\tilde{\z}^{k+1}&=\tilde{\z}^k+\textfrac{1}{\eta}(\identity-\W)\x_0^{k+1}.\label{z_update}
	\end{align}
}
% ----- REFORMULATION OF Z UPDATE ----- %
On the surface, there are two multiplications with $\W$ in (\ref{x_update})-(\ref{z_update}). However, they involve the same variable $\x_0$ differing in only two consecutive iterations. This implies that except for the first iteration, our framework requires only one multiplication by $\W$ per iteration and hence only one communication among agents (for networks where multiple communications are permitted, see Section~\ref{sec:variants}).

% ----- LOCAL OPTIMALITY CONDITIONS ----- %
We make two remarks on the solution of the local subproblem~\eqref{x_update}.
  \begin{remark}
For $\eta<\textfrac{1}{L}$, the $\x$ update performed in~\eqref{x_update} is accomplished by inexactly solving the following strongly convex local subproblem for all agents $i=1,\dots,N$,
	\begin{equation}\label{x_subproblem}
\textstyle		{\small \min_{\vec{x}_i}f_i(\vec{x}_i)+\ip{\vec{y}_i^k,\vec{x}_i-\vec{x}_{0,i}^k}+\frac{1}{2\eta}\norm{\vec{x}_i-\vec{x}_{0,i}^k}_2^2,}
	\end{equation}
where the inexactness is quantified by the following error quantities. We require
{\small
	\begin{equation}\label{x_optimality}
		\begin{split}
			&\textstyle \norm{\vec{r}_i^{k+1}}_2^2\le\textfrac{\epsilon_{k+1}}{N},\text{ with }\\
			\vec{r}_i^{k+1}\triangleq \nabla f_i(\vec{x}_i^{k+1})&+\vec{y}_i^k+\textfrac{1}{\eta}(\vec{x}_i^{k+1}-\vec{x}_{0,i}^k),\enskip\forall k\ge0,
		\end{split}
	\end{equation}
}to hold for the local error at iteration $k$ and for the cumulative error we require,
% ----- ERROR SUMMATION CONDITION ----- %
	\begin{equation}\label{error_summation}
	\textstyle	\small\sum_{k=1}^\infty\epsilon_{k+1}=\bigO{1-\rho}.
	\end{equation}
  \end{remark}
% ----- STOCHASTIC OPTIMALITY CONDITION ----- %
	\begin{remark}
		Similar to the results in~\cite{zhang21fedpd}, we can require,
			\begin{equation}\label{stochastic_x_optimality}
		\textstyle		\small\Exp\left[\norm{\vec{r}_i^{k+1}}_2^2\right]\le\textfrac{\epsilon_{k+1}}{N},\enskip\forall k\ge0,\text{ and}~\eqref{error_summation}
			\end{equation}
	and the theoretical results are not significantly affected. This allows for stochastic solvers to be used by each local agent.
	\end{remark}
From an agent's point of view, (\ref{x_update})-(\ref{z_update}) can be summarized in Alg.~\ref{algo:} below.

% ------------------------------------------------------------------------ %
% 								State the Algorithm							  %
% ------------------------------------------------------------------------ %
\begin{algorithm}
	\DontPrintSemicolon
	
	% Inputs the the algorithm
	\KwIn{\small $\x^0$, $\x_0^0$, $\y^0$, $\tilde{\z}^0=\sqrt{\identity-\W}\z^0$ with ${\z}^0\in\mathrm{range}\big(\sqrt{\identity-\W}\big)$, $K$, $\eta>0,$ a non-increasing sequence $\{\epsilon_k\}_{k=1}^{K}.$}
	
	% Write code here
	\SetInd{0.25em}{0.25em}
	\For{$k=0,\dots,K-1$}{
	
		\For{$i=1,\dots,N$ in parallel}{
			% X subproblem
			Update $\small\vec{x}_i$ until $\small\norm{\vec{r}_i^{k+1}}_2^2\le\textfrac{\epsilon_{k+1}}{N}$ with $\small\vec{r}_i^{k+1}$ in \eqref{x_optimality}\; %$\small\vec{r}_i^{k+1}\triangleq\nabla f_i(\vec{x}_i^{k+1})+\vec{y}^k+\textfrac{1}{\eta}(\vec{x}_i^{k+1}-\vec{x}_{0,i}^k)$\;
			
			% X0 subproblem
			\eIf{$k=0$}
				{
				$\small\vec{x}_{0,i}^{k+1}\gets\textfrac{1}{2}\left(\sum_{j\in\mathcal{N}_i\cup\{i\}}w_{ij}\vec{x}_{0,j}^{k}+\vec{x}_i^{k+1}+\eta(\vec{y}_i^k-\tilde{z}_i^k)\right)$\;
				}
				{
				$\small\vec{x}_{0,i}^{k+1}\gets\textfrac{1}{2}\left(\vec{x}_i^{k+1}+\vec{x}_{0,i}^k+\eta(\vec{y}_i^k-2\tilde{\vec{z}}_i^{k}+\tilde{\vec{z}}_i^{k-1})\right)$ \;
			}
						
			% Y subproblem
			$\small\vec{y}_i^{k+1}\gets\vec{y}_i^k+\textfrac{1}{\eta}\left(\vec{x}_i^{k+1}-\vec{x}_{0,i}^{k+1}\right)$\;
						
			% Z subproblem
			$\small\tilde{\vec{z}}_i^{k+1}\gets\tilde{\vec{z}}_i^{k}+\textfrac{1}{\eta}(1-w_{ii})\vec{x}_{0,i}^{k+1}-\textfrac{1}{\eta}\sum_{j\in\mathcal{N}_i}w_{ij}\vec{x}_{0,j}^{k+1}$ \;
			
		}
	}

	% Give the algorithm a title
	\caption{\textbf{ADAPD} (agent view)}
	
	% Label for reference
	\label{algo:}
\end{algorithm}

Recall that we obtain a unique sequence $\{\z^k\}_{k=1}^K$ in $\mathrm{range}\big(\sqrt{\identity-\W}\big)$ from the generated $\tilde\z$-sequence. Therefore, without causing confusion, we can use the corresponding $\z$-sequence in our analysis. Notice that our framework is sufficiently flexible to allow each agent to use different local subroutines to solve (\ref{x_subproblem}). In networks where the computing power of the agents differs vastly (see, e.g.~\cite{mcmahan17}), this flexible framework allows for agents with higher compute capabilities to fully utilize their compute power whereas agents with lower compute capabilities are not expected to utilize heavy optimization tools to solve their local subproblems. We now describe two variants/modifications to Alg.~\ref{algo:} that can be employed if the computational constraints and/or the communication constraints are relaxed.

% ------------------------------------------------------------------------ %
% 				 		    	State Modifications							  %
% ------------------------------------------------------------------------ %
\subsection{Framework Variants}\label{sec:variants}

% ------------------------------------------------------------------------ %
% 						    	Computation Mod								  %
% ------------------------------------------------------------------------ %

\subsubsection{Computation Variant}

In scenarios where agents may face a lack of computational resources to solve (\ref{x_subproblem}), it may be inefficient to compute $\nabla f_i(\cdot)$ many times. To remedy this, we propose ADAPD-OG (\textbf{O}ne \textbf{G}radient), which requires each agent to only compute a single gradient during every iteration. More precisely, we do the update:
	\begin{equation}\label{x_single_step_update}
		\small\x^{k+1}=\x_0^k-\eta\left(\nabla F(\x^k)+\y^k\right).
	\end{equation}
%which requires only a single local gradient computation. 
Notice that if $\widehat{\vec{x}}_i^{k+1}$ is the exact solution of \eqref{x_subproblem}, then $\widehat{\x}^{k+1}=\x_0^k-\eta\left(\nabla F(\widehat{\x}^{k+1})+\y^k\right)$, which is a backward step because $\nabla F(\widehat{\x}^{k+1})$ is unknown. The update in \eqref{x_single_step_update} is a forward step. Since we can show $\|\x^{k+1} - \x^{k}\|_F\to 0$, the forward step will eventually be a close approximation of the backward step, and thus we can expect convergence. Alg.~\ref{algo:single_step} displays the pseudocode of ADAPD-OG.

% OG
\begin{algorithm}
	\DontPrintSemicolon
	
	% Inputs the the algorithm
	\KwIn{\small$\x^0$, $\x_0^0$, $\y^0$, $\tilde{\z}^0=\sqrt{\identity-\W}\z^0$ with ${\z}^0\in\mathrm{range}\big(\sqrt{\identity-\W}\big)$, $K$, $\eta>0$.}
	
	% Write code here
	\SetInd{0.25em}{0.25em}
	\For{$k=0,\dots,K-1$}{
	
		\For{$i=1,\dots,N$ in parallel}{
		
			% X subproblem
			$\small\vec{x}_i^{k+1}\gets\vec{x}_{0,i}^k-\eta\left(\nabla f_i(\vec{x}_i^k)+\vec{y}_i^k\right)$\;
			
			% X0 subproblem
			Perform lines 4 - 9 in Alg.~\ref{algo:} to update $\vec{x}_{0,i}^{k+1},\vec{y}_i^{k+1},$ and $\vec{z}_i^{k+1}$\;
		}
	}

	% Give the algorithm a title
	\caption{\textbf{ADAPD-OG} (agent view)}
	
	% Label for reference
	\label{algo:single_step}
\end{algorithm}

% ------------------------------------------------------------------------ %
% 					    	Communication Mod							  %
% ------------------------------------------------------------------------ %
\subsubsection{Communication Variant}

For convergence, it may be practical to allow agents more than one communication during each ADAPD update. We denote the following multiple communication modification (either to Alg.~\ref{algo:} or Alg.~\ref{algo:single_step}) with appending an ``-MC'' (\textbf{M}ultiple \textbf{C}ommunications) to the algorithm name.

As stated in the introduction, our analysis depends on the value of $\rho$ which measures how quickly an average value can be computed in a decentralized manner. In a \emph{centralized} computing paradigm, where each agent is allowed to communicate with all other agents either directly or via a central server, the mixing matrix $\W$ can be replaced by the averaging matrix $\textfrac{1}{N}\basis\basis^\top.$ In this instance $\rho=0$, which can lead to the fastest convergence for our algorithm in both theory and practice. However, by Assumption~\ref{assumption:mixing_matrix}(i), the communication pattern of the agents is limited to performing only neighbor communications.

One straightforward modification to improve our method's dependence on $\rho$ is to replace $\W$ by $\W^R$ ($R$ denotes a power, not an iteration number here) for the $\tilde{\z}$ update in (\ref{z_update}) and the computation of $\x_0^1$, where $R\ge1$ is an integer. Notice that $\W^R$ satisfies Assumption~\ref{assumption:mixing_matrix}(ii)-(iv). Thus all our theoretical results  hold for this MC modification. Since $\rho(\W^R)=\|\W^R-\textfrac{1}{N}\basis\basis^\top\|_2= \rho(\W)^R$, this MC modification can lead to a smaller $\rho$ if $R>1$. However, if $\rho(\W)$ is very close to \emph{one}, $R$ needs to be very large in order to push $\rho(\W^R)$ to \emph{zero}. For this case, more efficient methods have been proposed in the literature for distributed averaging~\cite{liu11,ye20dapg,auzinger11}. We employ the Chebyshev accelerated method considered in~\cite{auzinger11}. The pseudocode is shown in Alg.~\ref{algo:cheby}. While the Chebyshev acceleration is called at iteration $k$ of ADAPD-MC or ADAPD-OG-MC, the input $\vec{A}^0$ will be $\x_0^{k+1}$.

% ------------------------------------------------------------------------ %
% 					    					Chebyshev										  %
% ------------------------------------------------------------------------ %
\begin{algorithm}
	\DontPrintSemicolon
{\small	
	% Inputs the the algorithm
	\KwIn{\small$\W,\vec{A}^{0}, \vec{A}^{1}=\W\vec{A}^0$, $R.$}
	% Write code here
	Compute the step-sizes $\mu_{0}=1,\mu_{1}=\frac{1}{\rho}$\;
	\SetInd{0.25em}{0.25em}
	\For{$r=1,\dots,R$}{
		$\small\mu_{r+1}\gets\frac{2}{\rho}\mu_{r}-\mu_{r-1}$\;
		$\small\vec{A}^{r+1}\gets\textfrac{2\mu_r}{\rho\mu_{r+1}}\W\vec{A}^r-\frac{\mu_{r-1}}{\mu_{r+1}}\vec{A}^{r-1}$\;
	}
	% Outputs of the algorithm
	\KwOut{$\small\vec{A}^R$}
}	
	% Give the algorithm a title
	\caption{\textbf{Chebyshev acceleration}}
	% Label for reference
	\label{algo:cheby}
\end{algorithm}

The following lemma shows that the properties in Assumption~\ref{assumption:mixing_matrix}(ii)-(iv) still hold for the operator used in the Chebyshev acceleration and provides an explicit convergence rate for Alg.~\ref{algo:cheby}. A proof is included in Appendix~\ref{appendix:cheby}; see the proof of Theorem 4 in~\cite{scaman17} and Corollary 6.1 in~\cite{auzinger11} for further details.

% ----- FAST MIX ASSUMPTIONS ----- %
\begin{theoremEnd}[category=cheby]{lemma}\label{lemma:chebyshev}
The output of Alg.~\ref{algo:cheby} can be represented as $\vec {A}^R = \mathcal {P}\left (\W ,R\right )\vec {A}^0$, where $\mathcal {P}\left (\W ,R\right )$ is a degree-$R$ polynomial of $\W$ and satisfies Assumptions~\ref{assumption:mixing_matrix}(ii)-(iv). Additionally, we have that $\bar{\vec{A}}^R=\bar{\vec{A}}^0\triangleq\bar{\vec{A}}$ for any $R$ and
	\begin{equation}\label{lemma:chebyshev:bound}
	\norm{\vec{A}^R-\bar{\vec{A}}}_F\le2\left(1-\sqrt{1-\rho}\right)^R\norm{\vec{A}^0-\bar{\vec{A}}}_F.
	\end{equation}
\end{theoremEnd}
\begin{proofEnd}\enskip
Notice that the iterations of Alg.~\ref{algo:cheby} define a polynomial in $\W$; denote this as $\mathcal{P}\left(\W,r\right)$ for any $r\ge1$. Let $\vec{A}^R=\mathcal{P}\left(\W,R\right)\vec{A}^0$ be the output of Alg.~\ref{algo:cheby} after $R\ge 1$ iterations. Proceeding by induction, we first show that $\mathcal{P}\left(\W,R\right)$ satisfies Assumption~\ref{assumption:mixing_matrix} (ii) and (iii). Then we establish~\eqref{lemma:chebyshev:bound} which gives part (iv) of Assumption~\ref{assumption:mixing_matrix}. For the case where $R=1,$ $\mathcal{P}\left(\W,1\right)=\W$ and hence Assumption~\ref{assumption:mixing_matrix} is satisfied. For the inductive case, it is obvious that Assumption~\ref{assumption:mixing_matrix} (ii) holds. For (iii), we compute the following
  	\begin{align*}
  		\mathcal{P}\left(\W,R\right)\basis&=\frac{2\mu_{R-1}}{\rho\mu_R}\W\mathcal{P}\left(\W,R-1\right)\basis-\frac{\mu_{R-2}}{\mu_{R}}\mathcal{P}\left(\W,R-2\right)\basis\\
  		&=\frac{2\mu_{R-1}}{\rho\mu_R}\basis-\frac{\mu_{R-2}}{\mu_{R}}\basis\\
  		&=\basis
  	\end{align*}
  where the second equality holds because of the inductive hypothesis and the last equality uses line 3 in Alg.~\ref{algo:cheby}.\\
  To show part (iv) of Assumption~\ref{assumption:mixing_matrix}, we prove~\eqref{lemma:chebyshev:bound}. Namely, we have
  	\begin{align*}
  	\norm{\vec{A}^R-\bar{\vec{A}}}_F&=\norm{\mathcal{P}\left(\W,R\right)\vec{A}^0-\bar{\vec{A}}}_F\\
  	&=\norm{\left(\mathcal{P}\left(\W,R\right)-\frac{1}{N}\basis\basis^\top\right)(\vec{A}^0-\bar{\vec{A}})}_F\\
  	&\le\norm{\mathcal{P}\left(\W,R\right)-\frac{1}{N}\basis\basis^\top}_2\norm{\vec{A}^0-\bar{\vec{A}}}_F\\
  	&=\norm{\sum_{r=0}^R\gamma_r\left(\W^r-\frac{1}{N}\basis\basis^\top\right)}_2\norm{\vec{A}^0-\bar{\vec{A}}}_F
  	\end{align*}
  where the last equality comes from defining $\mathcal{P}\left(\W,R\right)\triangleq\sum_{r=0}^R\gamma_r\W^r$ and noting that by our previous argument, $\sum_{r=0}^R\gamma_r=1.$ By Corollary 6.1 in~\cite{auzinger11}, we have that
  	\begin{align*}
  		\norm{\sum_{r=0}^R\gamma_r\left(\W^r-\frac{1}{N}\basis\basis^\top\right)}_2&=2\frac{c^R}{1+c^{2R}}\le2c^R
  	\end{align*}
  where $c\triangleq\frac{\sqrt{\kappa}-1}{\sqrt{\kappa}+1}$ with $\kappa\triangleq\frac{1+\rho}{1-\rho}.$ Finally utilizing that $c^R\le\left(1-\sqrt{1-\rho}\right)^R<1$ for any $\rho\in[0,1)$ proves~\eqref{lemma:chebyshev:bound} and by the symmetry of $\mathcal{P}\left(\W,r\right)$, part (iv) of Assumption~\ref{assumption:mixing_matrix} holds.
\end{proofEnd}

We note that employing Alg.~\ref{algo:cheby} is only feasible if either: (i) the communication pattern is so sparse that consensus error is the main bottleneck for convergence, or (ii) communication cost is low relative to the computation cost, meaning that agents can communicate faster than they can compute. In practice, it is suggested that agents find a balance that distributes work evenly between communication and computation.

% ------------------------------------------------------------------------ %
% 								Theoretical Results					           %
% ------------------------------------------------------------------------ %
\section{Theoretical Guarantees}\label{sec:theory}

Our theoretical analysis draws from decentralized analytical methods such as~\cite{zeng18,hong17} and classical non-convex ADMM analyses, as in~\cite{wang19}. We first show the change in the augmented Lagrangian function value after one ADAPD iteration, i.e. (\ref{x_update})-(\ref{z_update}). Then we define a Lyapunov function and use it to show convergence. A crucial quantity for our analysis is
	\begin{equation}\label{v_vector}
		\small\vec{V}_0^k\defined\left(\x_0^{k+1}-\x_0^k\right)-\left(\x_0^k-\x_0^{k-1}\right).
	\end{equation}
We define $\x_0^{-1}\triangleq\x_0^0,$ to ensure that $\vec{V}_0^k$ is defined for all $k\ge0.$ In the convergence analysis of Alg.~\ref{algo:}, we will make consistent use of the following two facts.
% ----- PETER-PAUL ----- %
\begin{fact}[Peter-Paul and Young's Inequality] For any $\vec{A},\vec{B}\in\R^{N\times p}$, for any $\delta>0$ and $i=1,\dots,m,$ we have,
	\begin{equation}\label{fact:peterpaulinequality}
	\textstyle	\small\ip{\vec{A}, \vec{B}}\le \textfrac{\delta}{2}\fronorm{\vec{A}}+\textfrac{1}{2\delta}\fronorm{\vec{B}},
	\end{equation}
	\begin{equation}\label{fact:youngsinequality}
	\textstyle	\small\norm{\sum_{i=1}^m\vec{A}_i}_F^2\le m\sum_{i=1}^m\norm{\vec{A}_i}_F^2.
	\end{equation}
\end{fact}

% ------------------------------------------------------------------------ %
% 								Inexact Algorithm					           %
% ------------------------------------------------------------------------ %
\subsection{Convergence Results of ADAPD}

The first step in the analysis creates an equivalence expression among the dual and primal variables. Proofs are located in Appendix~\ref{appendix:algo}.
% ----- DUAL VARIABLE REFORMULATION ----- %
\begin{theoremEnd}[category=dualbound]{lemma}\label{lemma:dual_variable_relation}
For all $k\ge0$, the dual variables in \eqref{y_update} and \eqref{z_update} can be expressed as
			{\small
			\begin{align}
				\sqrt{\identity-\W}\z^k&=\y^k-\textfrac{1}{\eta}\W(\x_0^k-\x_0^{k-1})\label{lemma:dual_variable_relation:z_bound}\\
				\y^{k}&=\vec{R}^k-\nabla F(\x^k)-\textfrac{1}{\eta}(\x_0^k-\x_0^{k-1})\label{lemma:dual_variable_relation:y_bound}
			\end{align}
			}where $\vec{R}^k\triangleq\mymatrix{\vec{r}^k_1&\dots&\vec{r}^k_N}^\top$ for $\vec{r}_i^k$ defined in~\eqref{x_optimality} for all $i=1,\dots,N.$
\end{theoremEnd}

\begin{proofEnd}\enskip
	Recall $\tilde{\z}^k\triangleq\sqrt{\identity-\W}\z^k$ for all $k\ge0.$ Thus by (\ref{x0_update}), we have
		{\small
		\begin{align*}
			&\x_0^{k}-\x_0^{k-1}\\
			&=-\textfrac{\eta}{2}\left(-\y^{k-1}-\textfrac{1}{\eta}(\x^{k}-\x_0^{k-1})+\tilde{\z}^{k-1}+\textfrac{1}{\eta}(\identity-\W)\x_0^{k-1}\right)\\
			&\stack{(\ref{y_update}),(\ref{z_update})}{=}\enskip-\textfrac{\eta}{2}\left(-\y^k-\textfrac{1}{\eta}(\x_0^k-\x_0^{k-1})+\tilde{\z}^k-\textfrac{1}{\eta}(\identity-\W)[\x_0^k-\x_0^{k-1}]\right)\\
			&=\textfrac{\eta}{2}\y^k+\textfrac{1}{2}(\x_0^k-\x_0^{k-1})-\textfrac{\eta}{2}\tilde{\z}^k+\textfrac{1}{2}(\identity-\W)[\x_0^k-\x_0^{k-1}].
		\end{align*}
		}Combining like terms, rearranging, and multiplying both sides by $\textfrac{2}{\eta}$ gives (\ref{lemma:dual_variable_relation:z_bound}).\\
	To prove (\ref{lemma:dual_variable_relation:y_bound}),  we have from (\ref{y_update}) that $\y^{k-1}=\y^k-\textfrac{1}{\eta}\left(\x^k-\x_0^k\right)$; plugging this into (\ref{x_optimality}) with $k\gets k-1$ yields the desired result.
\end{proofEnd}
Next, we characterize the change of the augmented Lagrangian function value after one ADAPD iteration.

% ------------------------------------------------------------------------ %
% 									AL Decrease					     		      %
% ------------------------------------------------------------------------ %
\begin{theoremEnd}[category=buildLyapunov]{lemma}\label{lemma:al_decrease} 
	Let $\{(\x^{k},\x_0^{k};\y^{k},\z^{k})\}$ be obtained from Alg.~\ref{algo:} or equivalently by updates \eqref{x_update}-\eqref{z_update} such that \eqref{x_optimality} holds.	If $\eta<\textfrac{1}{2L}$, then it holds for all $k\ge0$ that
	{\small
		\begin{equation}\label{lemma:al_decrease:bound}
			\begin{split}
				&\L_\eta(\x^{k+1},\x_0^{k+1};\y^{k+1},\z^{k+1})-\L_\eta(\x^{k},\x_0^{k};\y^{k},\z^{k})\\&\le\textfrac{2L\eta-1}{2\eta}\fronorm{\x^{k+1}-\x^k}+\textfrac{\epsilon_{k+1}}{2L}-\textfrac{1}{2\eta}\fronorm{\x_0^{k+1}-\x_0^k}\\
				&\quad \textstyle +\eta\fronorm{\y^{k+1}-\y^k}+\eta\fronorm{\z^{k+1}-\z^k}.
		\end{split}
		\end{equation}
	}
\end{theoremEnd}

\begin{proofEnd}\enskip 
	The inequality follows from rewriting $\L_\eta(\x^{k+1},\x_0^{k+1};\y^{k+1},\z^{k+1})-\L_\eta(\x^{k},\x_0^{k};\y^{k},\z^{k})$ as the summation of the left-hand sides of (\ref{lemma:x_decrease:bound}), (\ref{lemma:x0_decrease:bound}), (\ref{lemma:dual_decrease:y_bound}), and (\ref{lemma:dual_decrease:z_bound}) and using those four inequalities.
\end{proofEnd}

Notice that the inequality in \eqref{lemma:al_decrease:bound} does not imply the non-increasing monotonicity of the augmented Lagrangian function at the generated iterates. Below, we bound the dual variable change by the primal variable change and the $\vec{V}_0^k$ term. Then we establish another inequality and add it to \eqref{lemma:al_decrease:bound} to build a non-increasing Lyapunov function.

% Further Bounds!!
\begin{theoremEnd}[category=buildLyapunov]{lemma}\label{lemma:dual_var_bound}
	Under the assumptions of Lemma~\ref{lemma:al_decrease}, it holds that for all $k\ge0,$
	{\small
	\begin{align}
		\textstyle &\eta\fronorm{\y^{k+1}-\y^k}\le4L^2\eta\fronorm{\x^{k+1}-\x^k}+\textfrac{4}{\eta}\fronorm{\vec{V}_0^k}+8\eta\epsilon_k,\label{lemma:dual_var_bound:y_bound}\\
	\textstyle	 &\eta(1-\rho)\fronorm{\z^{k+1}-\z^k}\le8L^2\eta\fronorm{\x^{k+1}-\x^k}+\textfrac{10}{\eta}\fronorm{\vec{V}_0^k}+16\eta\epsilon_k,\label{lemma:dual_var_bound:z_bound}
	\end{align}
	}where $\vec{V}_0^k$ is defined in (\ref{v_vector}).

\end{theoremEnd}

\begin{proofEnd}\enskip
	To prove (\ref{lemma:dual_var_bound:y_bound}), by (\ref{lemma:dual_variable_relation:y_bound}), we have
	{\small	\begin{align*}
			&\eta\fronorm{\y^{k+1}-\y^k}\\
			&=\eta\fronorm{\vec{R}^{k+1}-\vec{R}^k-\nabla F(\x^{k+1})+\nabla F(\x^k)-\textfrac{1}{\eta}\vec{V}_0^k}\\
			&\stack{(\ref{fact:youngsinequality})}{\le} \textstyle 4\eta\left(\fronorm{\vec{R}^{k+1}}+\fronorm{\vec{R}^k}+\fronorm{\nabla F(\x^{k+1})-\nabla F(\x^k)}+\textfrac{1}{\eta^2}\fronorm{\vec{V}_0^k}\right)\\
			&\stack{(\ref{x_optimality}),(\ref{assumption:smoothness})}{\le}\quad4L^2\eta\fronorm{\x^{k+1}-\x^k}+\textfrac{4}{\eta}\fronorm{\vec{V}_0^k}+8\eta\epsilon_k
		\end{align*}
}where in the last inequality we have further used $\epsilon_{k+1}\le\epsilon_{k}$ for all $k\ge0$.\\
	To prove (\ref{lemma:dual_var_bound:z_bound}), notice that if $\z^0\in\mathrm{range}\big(\sqrt{\identity-\W}\big),$ then by (\ref{z_update_classic}), $\z^k\in\mathrm{range}\big(\sqrt{\identity-\W}\big)$ for all $k\ge0.$ Thus with $\rho_2\triangleq1-\lambda_2(\w)$, we have
		\begin{equation}\label{eq:bd-range-Z}
			\small\eta\rho_2\fronorm{\z^{k+1}-\z^k}\le\eta\fronorm{\sqrt{\identity-\W}\big(\z^{k+1}-\z^k\big)},
		\end{equation}
	and since $1-\rho\le\rho_2$, it further holds that,
		\begin{equation}\label{eq:bd-range-Z-1}
			\small\eta(1-\rho)\fronorm{\z^{k+1}-\z^k}\le\eta\rho_2\fronorm{\z^{k+1}-\z^k}
		\end{equation}
	In addition,
	    {\small
		\begin{align}
				&\eta\fronorm{\sqrt{\identity-\W}\big(\z^{k+1}-\z^k\big)}\\
				&\stack{(\ref{lemma:dual_variable_relation:z_bound})}{=}\eta\fronorm{\y^{k+1}-\y^k-\textfrac{1}{\eta}\W\vec{V}_0^k}\cr
				&\stack{(\ref{fact:peterpaulinequality})}{\le}2\eta\fronorm{\y^{k+1}-\y^k}+\textfrac{2}{\eta}\fronorm{\W\vec{V}_0^k}\label{eq:bd-range-Z-mid}\\
				&\stack{(\ref{lemma:dual_var_bound:y_bound})}{\le}8L^2\eta\fronorm{\x^{k+1}-\x^k}+\textfrac{8}{\eta}\fronorm{\vec{V}_0^k}+16\eta\epsilon_k+\textfrac{2}{\eta}\fronorm{\W\vec{V}_0^k}\cr
				&\le8L^2\eta\fronorm{\x^{k+1}-\x^k}+\textfrac{10}{\eta}\fronorm{\vec{V}_0^k}+16\eta\epsilon_k\label{eq:bd-range-Z-2}
			\end{align}
		}where the last inequality uses Assumption~\ref{assumption:mixing_matrix}(iv). Using~\eqref{eq:bd-range-Z} with~\eqref{eq:bd-range-Z-1}, we complete the proof.
\end{proofEnd}

% SUPPORTING INEQUALITY LEMMA
\begin{theoremEnd}[category=buildLyapunov]{lemma}\label{lemma:supporing_inequality}
	For all $k\ge0,$ the following relation holds
	{\small
	\begin{equation}\label{lemma:supporting_inequality:bound}
		\begin{split}
		&\textstyle \textfrac{1}{2\eta} \textstyle \left(\fronorm{\sqrt{\identity-\W}\x_0^{k+1}}+\fronorm{\sqrt{\identity-\W}(\x_0^{k+1}-\x_0^k)}\right)\\
		&\quad \textstyle -\textfrac{1}{2\eta}\fronorm{\sqrt{\identity-\W}\x_0^{k}}\\
		&\quad \textstyle +\textfrac{1}{2\eta}\left(\norm{\vec{V}_0^k}_\W^2+\norm{\x_0^{k+1}-\x_0^k}_\W^2-\norm{\x_0^k-\x_0^{k-1}}_\W^2\right)\\
		&\textstyle \le(L-\textfrac{1}{2\eta})\fronorm{\x_0^{k+1}-\x_0^k}+\textfrac{L}{2}\fronorm{\x^{k+1}-\x^k}\\
		&\quad \textstyle +\textfrac{1}{2\eta}\fronorm{\x_0^k-\x_0^{k-1}}-\textfrac{1}{2\eta}\fronorm{\vec{V}_0^k}+\textfrac{2}{L}\epsilon_k
		\end{split}
	\end{equation}
	}where $\vec{V}_0^k$ is defined in (\ref{v_vector}).

\end{theoremEnd}

\begin{proofEnd}\enskip
	By (\ref{lemma:dual_variable_relation:y_bound}), we have
	{\small
	\begin{equation}\label{lemma:supporting_inequality:eqn1}
	\begin{split}
		\small&\ip{\y^{k+1}-\y^k,\x_0^{k+1}-\x_0^{k}}\\
		&=\ip{\vec{R}^{k+1}-\vec{R}^k-\nabla F(\x^{k+1})+\nabla F(\x^k)-\textfrac{1}{\eta}\vec{V}_0^k,\x_0^{k+1}-\x_0^{k}}.
	\end{split}
	\end{equation}
}	We handle the two sides of \eqref{lemma:supporting_inequality:eqn1} separately. First, we have
{\small
		\begin{align}\label{lemma:supporting_inequality:lhs_bound}
			\small&\ip{\y^{k+1}-\y^k,\x_0^{k+1}-\x_0^{k}}\\
			&\stack{(\ref{lemma:dual_variable_relation:z_bound})}{=}\ip{\sqrt{\identity-\W}\z^{k+1}-\sqrt{\identity-\W}\z^k+\textfrac{1}{\eta}\W\vec{V}_0^k,\x_0^{k+1}-\x_0^{k}}\cr
			&\stack{(\ref{z_update})}{=}\ip{\textfrac{1}{\eta}(\identity-\W)\x_0^{k+1}+\textfrac{1}{\eta}\W\vec{V}_0^k,\x_0^{k+1}-\x_0^{k}}\cr
			&=\textfrac{1}{2\eta}\left(\fronorm{\sqrt{\identity-\W}\x_0^{k+1}}+\fronorm{\sqrt{\identity-\W}(\x_0^{k+1}-\x_0^k)}\right)\\
			&\quad-\textfrac{1}{2\eta}\fronorm{\sqrt{\identity-\W}\x_0^{k}}\\
				&\quad+\textfrac{1}{2\eta}\left(\norm{\vec{V}_0^k}_\W^2+\norm{\x_0^{k+1}-\x_0^k}_\W^2-\norm{\x_0^k-\x_0^{k-1}}_\W^2\right)\nonumber
		\end{align}
}where the last equality can be verified straightforwardly. Second, we have
	{\small
		\begin{align*}
			&\ip{\vec{R}^{k+1}-\vec{R}^k-\nabla F(\x^{k+1})+\nabla F(\x^k)-\textfrac{1}{\eta}\vec{V}_0^k,\x_0^{k+1}-\x_0^{k}}\\
			&\stack{(\ref{fact:peterpaulinequality}),(\ref{assumption:smoothness})}{\le}\quad\textfrac{1}{2L}\fronorm{\vec{R}^{k+1}-\vec{R}^k}+\textfrac{L}{2}\fronorm{\x_0^{k+1}-\x_0^k}\\
				&\quad+\textfrac{L}{2}\left(\fronorm{\x^{k+1}-\x^k}+\fronorm{\x_0^{k+1}-\x_0^k}\right)\\
				&\quad-\textfrac{1}{\eta}\ip{\vec{V}_0^k,\x_0^{k+1}-\x_0^{k}}\\
			&\stack{(\ref{fact:youngsinequality})}{\le}\textfrac{2}{L}\epsilon_k+\textfrac{L}{2}\fronorm{\x^{k+1}-\x^k}+L\fronorm{\x_0^{k+1}-\x_0^k}\\
			&-\textfrac{1}{2\eta}\left(\fronorm{\vec{V}_0^k}+\fronorm{\x_0^{k+1}-\x_0^k}-\fronorm{\x_0^k-\x_0^{k-1}}\right)
		\end{align*}
	}where the last inequality comes from $\epsilon_{k+1}\le\epsilon_{k}$ for all $k\ge1.$ Combining like terms results in the right hand side of (\ref{lemma:supporting_inequality:bound}); further using the equality established in (\ref{lemma:supporting_inequality:eqn1}) completes the proof.
\end{proofEnd}

Using Lemmas \ref{lemma:dual_var_bound} and \ref{lemma:supporing_inequality}, we are ready to build a non-increasing Lyapunov function as follows.

% FURTHER BOUND THE AUGMENTED LAGRANGIAN FUNCTION
\begin{theoremEnd}[category=buildLyapunov]{lemma}\label{lemma:al_decrease2}
	Let $\{(\x^{k},\x_0^{k};\y^{k},\z^{k})\}$ be obtained from Alg.~\ref{algo:} or equivalently by updates \eqref{x_update}-\eqref{z_update} such that \eqref{x_optimality} holds.	If $\eta<\textfrac{1}{2L}$, then
	{\small
	\begin{equation}\label{lemma:al_decrease2:bound}
		\begin{split}
			&\L_\eta(\x^{k+1},\x_0^{k+1};\y^{k+1},\z^{k+1})+\textfrac{C}{2\eta}\fronorm{\sqrt{\identity-\W}\x_0^{k+1}}\\
			&\quad+\textfrac{C}{2\eta}\fronorm{\x_0^{k+1}-\x_0^k}\\ 
			&\le\L_\eta(\x^{k},\x_0^{k};\y^{k},\z^{k})+\textfrac{C}{2\eta}\fronorm{\sqrt{\identity-\W}\x_0^{k}}+\textfrac{C}{\eta}\fronorm{\x_0^k-\x_0^{k-1}}\\
			&\quad+\left(\textfrac{(8L^2(1-\rho)+16L^2)\eta^2+(C+2)L(1-\rho)\eta-(1-\rho)}{2(1-\rho)\eta}\right)\fronorm{\x^{k+1}-\x^k}\\
			&\quad+\left(\textfrac{2CL\eta-C-1}{2\eta}\right)\fronorm{\x_0^{k+1}-\x_0^k}\\
			&\quad+\textfrac{(1-\rho)+(32L+16L(1-\rho))\eta+4C(1-\rho)}{2L(1-\rho)}\epsilon_k.
		\end{split}
	\end{equation}
	}for all $k\ge0,$ where $C\ge\textfrac{20+8(1-\rho)}{(1-\rho)^2}$ is a fixed constant.
\end{theoremEnd}

\begin{proofEnd}\enskip 
	By Lemmas \ref{lemma:al_decrease} and \ref{lemma:dual_var_bound} and also using $\epsilon_{k+1}\le\epsilon_k$, we have
		{\small
			\begin{align*}
				&\L_\eta(\x^{k+1},\x_0^{k+1};\y^{k+1},\z^{k+1})-\L_\eta(\x^{k},\x_0^{k};\y^{k},\z^{k})\\
				&\le\textfrac{(8L^2(1-\rho)+16L^2)\eta^2+2L(1-\rho)\eta-(1-\rho)}{2(1-\rho)\eta}\fronorm{\x^{k+1}-\x^k}\\
				&\quad-\textfrac{1}{2\eta}\fronorm{\x_0^{k+1}-\x_0^k}+\textfrac{10+4(1-\rho)}{(1-\rho)\eta}\fronorm{\vec{V}_0^k}+\textfrac{32L\eta+16L(1-\rho)\eta+(1-\rho)}{2L(1-\rho)}\epsilon_k.
			\end{align*}
		}Now multiplying $C>0$ to both sides of (\ref{lemma:supporting_inequality:bound}) and adding to the above inequality, we have
		{\small
			\begin{equation}\label{lemma:al_decrease2:eqn1}
			\begin{split}
				&\L_\eta(\x^{k+1},\x_0^{k+1};\y^{k+1},\z^{k+1})+\textfrac{C}{2\eta}\norm{\x_0^{k+1}-\x_0^k}_\W^2+\textfrac{C}{2\eta}\norm{\vec{V}_0^k}_\W^2\\
				&\quad+\textfrac{C}{2\eta}\fronorm{\sqrt{\identity-\W}\x_0^{k+1}}+\textfrac{C}{2\eta}\fronorm{\sqrt{\identity-\W}(\x_0^{k+1}-\x_0^k)}\\
				&\le\L_\eta(\x^{k},\x_0^{k};\y^{k},\z^{k})+\textfrac{C}{2\eta}\fronorm{\sqrt{\identity-\W}\x_0^{k}}\\
				&\quad+\textfrac{C}{2\eta}\norm{\x_0^{k}-\x_0^{k-1}}_\W^2+\textfrac{C}{2\eta}\fronorm{\x_0^k-\x_0^{k-1}}\\
					&\quad+\left(\textfrac{(8L^2(1-\rho)+16L^2)\eta^2+(C+2)L(1-\rho)\eta-(1-\rho)}{2(1-\rho)\eta}\right)\fronorm{\x^{k+1}-\x^k}\\
					&\quad+\textfrac{2CL\eta-C-1}{2\eta}\fronorm{\x_0^{k+1}-\x_0^k}+\textfrac{20+8(1-\rho)-C(1-\rho)}{2(1-\rho)\eta}\fronorm{\vec{V}_0^k}\\
					&\quad+\textfrac{(1-\rho)+(32L+16L(1-\rho))\eta+4C(1-\rho)}{2L(1-\rho)}\epsilon_k.
			\end{split}
			\end{equation}
		}Since the minimum eigenvalue of $\identity+\W$ is $\rho_N\triangleq1+\lambda_N(\w)>0$, it holds $\textfrac{20+8(1-\rho)}{(1-\rho)}\identity \preccurlyeq C\left(\identity+\W\right)$ when $C\ge\textfrac{20+8(1-\rho)}{(1-\rho)\rho_N}$. Hence, we have $0\le\textfrac{C(1-\rho)-20-8(1-\rho)}{2(1-\rho)\eta}\fronorm{\vec{V}_0^k}+\textfrac{C}{2\eta}\norm{\vec{V}_0^k}_\W^2$ by noticing
	{\small
			\begin{equation}\label{lemma:c_bound-eq}
			\begin{split}
				&\textfrac{C(1-\rho)-20-8(1-\rho)}{2(1-\rho)\eta}\fronorm{\vec{V}_0^k}+\textfrac{C}{2\eta}\norm{\vec{V}_0^k}_\W^2\\
				&=\norm{\vec{V}_0^k}_{\textfrac{C(1-\rho)-20-8(1-\rho)}{2(1-\rho)\eta}\identity+\textfrac{C}{2\eta}\W}^2 \ge0.
			\end{split}
			\end{equation}
	}Noticing that $\textfrac{1}{(1-\rho)}\ge\textfrac{1}{\rho_N}$, we have $C\ge\textfrac{20+8(1-\rho)}{(1-\rho)^2}$ also satisfies the above requirement. In addition, it holds
	{\small
		\begin{align*}
			&\textfrac{C}{2\eta}\fronorm{\sqrt{\identity-\W}(\x_0^{k+1}-\x_0^k)}+\textfrac{C}{2\eta}\norm{\x_0^{k+1}-\x_0^k}^2_{\W}\\
			&=\textfrac{C	}{2\eta}\norm{\x_0^{k+1}-\x_0^k}^2_{\identity-\W}+\textfrac{C}{2\eta}\norm{\x_0^{k+1}-\x_0^k}^2_{\W}=\textfrac{C}{2\eta}\fronorm{\x_0^{k+1}-\x_0^k}.
		\end{align*}
	}Furthermore, noting $\textfrac{C}{2\eta}\norm{\x_0^{k}-\x_0^{k-1}}^2_{\W}\le\textfrac{C}{2\eta}\fronorm{\x_0^{k}-\x_0^{k-1}}$, we obtain the desired result from \eqref{lemma:al_decrease2:eqn1}.
\end{proofEnd}

For the rest of the analysis, we fix $C\triangleq\textfrac{28}{(1-\rho)^2}$ as used in Lemma \ref{lemma:al_decrease2}, define the Lyapunov function:
	\begin{equation}\label{lyapunov}
		\small \Phi^k\defined\L_\eta(\x^{k},\x_0^{k};\y^{k},\z^{k})+\textfrac{C}{2\eta}\fronorm{\sqrt{\identity-\W}\x_0^{k}}+\textfrac{C}{\eta}\fronorm{\x_0^k-\x_0^{k-1}}.
	\end{equation}
We show the lower boundedness of this Lyapunov function in the following proposition and use this to obtain the convergence of Alg.~\ref{algo:}.

% LOWER BOUND OF LYAPUNOV FUNCTION
\begin{theoremEnd}[category=inexactproofs]{proposition}\label{proposition:lyapunov_lower_bound}
	Under Assumptions \ref{assumption:mixing_matrix} and \ref{assumption:objective_function}, let $\{(\x^{k},\x_0^{k};\y^{k},\z^{k})\}$ be obtained from Alg.~\ref{algo:} or equivalently by updates \eqref{x_update}-\eqref{z_update} such that \eqref{x_optimality} and \eqref{error_summation} hold. Choose $C$ and $\eta$ such that
	{\small
		\begin{equation}\label{eq:cond-C-eta}
		\begin{split}
			C=\textfrac{28}{(1-\rho)^2}\text{ and }\eta<\textfrac{1}{2CL}.
		\end{split}
		\end{equation}
	}Then the Lyapunov function (\ref{lyapunov}) is uniformly lower bounded. More specifically, for all $k\ge0,$
		\begin{equation}\label{proposition:lyapunov_lower_bound:bound}
			\small \Phi^k \ge \underline{\phi}:= \underline{f}-\textfrac{(1-\rho)+(32L+16L(1-\rho))\eta+4C(1-\rho)}{2L(1-\rho)}\sum_{k=0}^\infty\epsilon_{k}-1 > -\infty,
		\end{equation}
	where we take $\epsilon_0=\epsilon_1$ and $\underline{f}$ is defined in Assumption~\ref{assumption:objective_function}.	
\end{theoremEnd}

\begin{proofEnd}\enskip
	First, it is obvious that $\L_\eta(\x^{k},\x_0^{k};\y^{k},\z^{k})\le\Phi^{k}$ for all $k\ge0$ by the definition of $\Phi^k$ in \eqref{lyapunov}.
	Second, notice
	{\small
		\begin{align*}
			&\L_\eta(\x^{k+1},\x_0^{k+1};\y^{k+1},\z^{k+1})\\
			&=F(\x^{k+1})+\ip{\y^{k+1},\x^{k+1}-\x_0^{k+1}}+\textfrac{1}{2\eta}\fronorm{\x^{k+1}-\x_0^{k+1}}\\
			&\quad+\ip{\z^{k+1},\sqrt{\identity-\W}\x_0^{k+1}}+\textfrac{1}{2\eta}\fronorm{\sqrt{\identity-\W}\x_0^{k+1}}\\
			&\stack{(\ref{y_update}),(\ref{z_update})}{=}\quad F(\x^{k+1})+\ip{\y^{k+1},\eta(\y^{k+1}-\y^{k})}+\textfrac{1}{2\eta}\fronorm{\x^{k+1}-\x_0^{k+1}}\\
				&\quad+\ip{\z^{k+1},\eta(\z^{k+1}-\z^{k})}+\textfrac{1}{2\eta}\fronorm{\sqrt{\identity-\W}\x_0^{k+1}}\\
			&=F(\x^{k+1})+\textfrac{\eta}{2}\left(\fronorm{\y^{k+1}}+\fronorm{\y^{k+1}-\y^{k}}-\fronorm{\y^{k}}\right)\\
				&\quad+\textfrac{1}{2\eta}\fronorm{\x^{k+1}-\x_0^{k+1}}+\textfrac{1}{2\eta}\fronorm{\sqrt{\identity-\W}\x_0^{k+1}}\\
				&\quad+\textfrac{\eta}{2}\left(\fronorm{\z^{k+1}}+\fronorm{\z^{k+1}-\z^{k}}-\fronorm{\z^{k}}\right)
		\end{align*}
	}Thus, by the definition of $\underline{f}$ in (\ref{assumption:lower_bound}), we have that for any integer number $K\ge1$,
	{\small
		\begin{align}\label{eq:lower-bd-phik-sum}
			&\sum_{k=0}^{K-1}\left(\Phi^{k+1}-\underline{f}\right)\cr
			&\ge\sum_{k=0}^{K-1}\left(\L_\eta(\x^{k+1},\x_0^{k+1};\y^{k+1},\z^{k+1})-\underline{f}\right)\cr
			&=\sum_{k=0}^{K-1}\left(F(\x^{k+1})-\underline{f}+\textfrac{\eta}{2}\fronorm{\y^{k+1}-\y^{k}}+\textfrac{1}{2\eta}\fronorm{\x^{k+1}-\x_0^{k+1}}\right)\cr
				&+\sum_{k=0}^{K-1}\left(\textfrac{\eta}{2}\fronorm{\z^{k+1}-\z^{k}}+\textfrac{1}{2\eta}\fronorm{\sqrt{\identity-\W}\x_0^{k+1}}\right)\cr
				&\quad+\textfrac{\eta}{2}\fronorm{\y^{K}}-\textfrac{\eta}{2}\fronorm{\y^0}+\textfrac{\eta}{2}\fronorm{\z^{K}}-\textfrac{\eta}{2}\fronorm{\z^0}\cr
			&\ge-\textfrac{\eta}{2}\fronorm{\y^0}-\textfrac{\eta}{2}\fronorm{\z^0}\triangleq-M.
		\end{align}
	}Thirdly, by (\ref{lemma:al_decrease2:bound}) and the definition of $\Phi^k$ in \eqref{lyapunov}, 
	we have 
	{\small
		\begin{equation}\label{proposition:lyapunov_lower_bound:eqn1-0}
		\begin{split}
			&\Phi^{k+1}+\textfrac{(1-\rho)-(C+2)L(1-\rho)\eta-(8L^2(1-\rho)+16L^2)\eta^2}{2(1-\rho)\eta}\fronorm{\x^{k+1}-\x^k}\\
				&+\left(\textfrac{1}{2\eta}-CL\right)\fronorm{\x_0^{k+1}-\x_0^k}\le\Phi^k+\textfrac{(1-\rho)+(32L+16L(1-\rho))\eta+4C(1-\rho)}{2L(1-\rho)}\epsilon_k.
		\end{split}
		\end{equation}
	}Hence, it holds from the choice of $C$ and $\eta$ that
		\begin{equation}\label{proposition:lyapunov_lower_bound:eqn1}
			\small\Phi^{k+1}\le\Phi^k+\textfrac{(1-\rho)+(32L+16L(1-\rho))\eta+4C(1-\rho)}{2L(1-\rho)}\epsilon_k.
		\end{equation}
	Now assume that there exists $k_0\ge0$ such that $\Phi^{k_0}-\underline{f}<-\textfrac{(1-\rho)+(32L+16L(1-\rho))\eta+4C(1-\rho)}{2L(1-\rho)}\sum_{k=0}^\infty\epsilon_{k}-1$. Then summing up (\ref{proposition:lyapunov_lower_bound:eqn1}) gives $\Phi^k-\underline{f}\le\Phi^{k_0}-\underline{f}+\textfrac{(1-\rho)+(32L+16L(1-\rho))\eta+4C(1-\rho)}{2L(1-\rho)}\sum_{k=k_0}^\infty\epsilon_k <-1$ for all $k\ge k_0$. Hence, $\sum_{k=k_0+1}^\infty\left(\Phi^k-\underline{f}\right) = -\infty$, which contradicts \eqref{eq:lower-bd-phik-sum}. Therefore, we conclude that $\Phi^{k}-\underline{f}\ge-\textfrac{(1-\rho)+(32L+16L(1-\rho))\eta+4C(1-\rho)}{2L(1-\rho)}\sum_{k=0}^\infty\epsilon_{k}-1,$ for all $k\ge0$ and complete the proof.
\end{proofEnd}

We are now in position to prove the convergence rate results of ADAPD.

% CONVERGENCE RESULTS
\begin{theoremEnd}[category=inexactproofs]{thm}\label{theorem:convergence1}
	Under the same conditions assumed in Proposition~\ref{proposition:lyapunov_lower_bound}, it holds that
	{\small
		\begin{equation}\label{theorem:convergence1:bound}
			\begin{split}
				&\textfrac{C_1}{K}\sum_{k=0}^{K-1}\left(\fronorm{\x^{k+1}-\x^k}+\fronorm{\x_0^{k+1}-\x_0^k}\right)\\
				&\le\textfrac{\Delta_{\Phi}}{K}+\textfrac{(32L+16L(1-\rho))\eta+(4C+1)(1-\rho)}{2L(1-\rho) K}\sum_{k=0}^{K-1}\epsilon_k,
			\end{split}
		\end{equation}
	}where $\Delta_{\Phi}\defined\Phi^0-\underline{\phi}$ and
		\begin{equation}\label{theorem:convergence1:constant}
		    \small C_1\defined\textfrac{1-2CL\eta}{2\eta}.
		\end{equation}
\end{theoremEnd}

\begin{proofEnd}\enskip
	Summing up \eqref{proposition:lyapunov_lower_bound:eqn1-0} from $k=0$ to $K-1$ and dividing by $K$ results in
	{\small
		\begin{align}
			&\left(\textfrac{(1-\rho)-(C+2)L(1-\rho)\eta-(8L^2(1-\rho)+16L^2)\eta^2}{2(1-\rho)\eta}\right)\textfrac{1}{K}\sum_{k=0}^{K-1}\fronorm{\x^{k+1}-\x^k}\nonumber\\
			&+\textfrac{1-2CL\eta}{2\eta}\textfrac{1}{K}\sum_{k=0}^{K-1}\fronorm{\x_0^{k+1}-\x_0^k}\nonumber\\
			&\le \textfrac{\Phi^0-\Phi^{K}}{K}+\textfrac{(1-\rho)+(32L+16L(1-\rho))\eta+4C(1-\rho)}{2L(1-\rho)}\cdot\textfrac{1}{K}\sum_{k=0}^{K-1}\epsilon_k\nonumber\\
			&\stack{(\ref{proposition:lyapunov_lower_bound:bound})}{\le}\textfrac{\Phi^0-\underline{\phi}}{K}+\textfrac{(1-\rho)+(32L+16L(1-\rho))\eta+4C(1-\rho)}{2L(1-\rho)}\cdot\textfrac{1}{K}\sum_{k=0}^{K-1}\epsilon_k.\label{eq:sum-diff-X-X0-bd}
		\end{align}
	}By the choice of $C$ and $\eta$, it holds that $\textfrac{1-2CL\eta}{2\eta}\le\textfrac{(1-\rho)-(C+2)L(1-\rho)\eta-(8L^2(1-\rho)+16L^2)\eta^2}{2(1-\rho)\eta}$ so $C_1$ as defined in (\ref{theorem:convergence1:constant}) is positive, and thus the inequality in~\eqref{eq:sum-diff-X-X0-bd} implies the desired result.
\end{proofEnd}

% ----- TOTAL CONVERGENCE ----- %
\begin{theoremEnd}[category=inexactproofs]{thm}[Convergence of ADAPD]\label{theorem:total_convergence} 
	Under the same conditions assumed in Proposition~\ref{proposition:lyapunov_lower_bound}, it holds
	{\small
		\begin{align}
			&\textfrac{1}{K}\sum_{k=0}^{K-1}\left(\norm{\nabla f(\bar{\vec{x}}^{k+1})}_2^2+\fronorm{\x^{k+1}-\bar{\x}^{k+1}}\right)\nonumber\\
			&{\le}\textfrac{\left((2L^2+1)C_2+C_4\right)\Delta_{\Phi}}{KC_1}+\textfrac{(192L^2+96)\eta^2}{KC_1(1-\rho)^2}\sum_{k=0}^{K-1}\epsilon_k\nonumber\\
			&\quad+\textfrac{\left((2L^2+1)C_2C_3+C_3C_4+4C_1\right)}{KC_1}\sum_{k=0}^{K-1}\epsilon_k\label{theorem:total_convergence:eqn}
		\end{align}
	}where $C_1$ is defined in (\ref{theorem:convergence1:constant}), 	$C_2\triangleq\textfrac{208}{(1-\rho)^2}$, $C_3\triangleq\textfrac{(32L+16L(1-\rho))\eta+(4C+1)(1-\rho)}{2L(1-\rho)}$, $C_4\triangleq\textfrac{16}{\eta^2}$, $\bar{\vec{x}}^{k}\triangleq\textfrac{1}{N}\sum\limits_{i=1}^N\vec{x}_i^k$, and $\bar{\x}^{k}\triangleq\textfrac{1}{N}\basis\basis^\top\x^{k}$.
\end{theoremEnd}

\begin{proofEnd}\enskip 
	First, we have for all $k\ge0$ that
	{\small
			\begin{align*}
			&\fronorm{\x^{k+1}-\bar{\x}^{k+1}}\\
			&=\fronorm{\x^{k+1}-\W\x^{k+1}+\W\x^{k+1}-\bar{\x}^{k+1}}\\
			&=\fronorm{\x^{k+1}-\W\x^{k+1}}+\fronorm{\left(\W-\textfrac{1}{N}\basis\basis^\top\right)(\x^{k+1}-\bar{\x}^{k+1})}\\
				&\quad+2\ip{\x^{k+1}-\W\x^{k+1},\left(\W-\textfrac{1}{N}\basis\basis^\top\right)(\x^{k+1}-\bar{\x}^{k+1})}\\
			&\stack{(\ref{fact:peterpaulinequality})}{\le}\fronorm{\x^{k+1}-\W\x^{k+1}}+\fronorm{\left(\W-\textfrac{1}{N}\basis\basis^\top\right)(\x^{k+1}-\bar{\x}^{k+1})}\\
				&\quad+\textfrac{1}{\delta}\fronorm{\x^{k+1}-\W\x^{k+1}}+\delta\fronorm{\left(\W-\textfrac{1}{N}\basis\basis^\top\right)(\x^{k+1}-\bar{\x}^{k+1})}\\
			&\le \rho^2\left(1+\delta\right)\fronorm{\x^{k+1}-\bar{\x}^{k+1}}+\left(1+\textfrac{1}{\delta}\right)\fronorm{\x^{k+1}-\W\x^{k+1}},
			\end{align*}
		}where $\delta>0$, and we have used $\rho=\norm{\W-\textfrac{1}{N}\basis\basis^\top}_2$ from (\ref{spectral_gap}). Choosing $\delta=\textfrac{1-\rho}{\rho}>0$ and simplifying the result,
		we obtain   
			\begin{equation}\label{theorem:consensus:eqn1}
				\small\fronorm{\x^{k+1}-\bar{\x}^{k+1}}\le\textfrac{1}{(1-\rho)^2}\fronorm{\x^{k+1}-\W\x^{k+1}}.
			\end{equation}
		Now, by (\ref{z_update}) and the proof of Lemma~\ref{lemma:dual_var_bound}, we have
		{\small
			\begin{equation}\label{theorem:convergence1:eqn2}
				\begin{split}
					&\textfrac{1}{\eta}\fronorm{\big(\identity-\W\big)\x_0^{k+1}}=\eta\fronorm{\sqrt{\identity-\W}(\z^{k+1}-\z^k)}\\
					&\le8L^2\eta\fronorm{\x^{k+1}-\x^k}+\textfrac{10}{\eta}\fronorm{\vec{V}_0^k}+16\eta\epsilon_k
				\end{split}
			\end{equation}
		}and by (\ref{y_update}),
		{\small
			\begin{equation}\label{theorem:convergence1:eqn3}
				\begin{split}
				&\fronorm{\x^{k+1}-\x_0^{k+1}}=\eta^2\fronorm{\y^{k+1}-\y^k}\\
				&\stack{(\ref{lemma:dual_var_bound:y_bound})}{\le}4L^2\eta^2\fronorm{\x^{k+1}-\x^k}+4\fronorm{\vec{V}_0^k}+8\eta^2\epsilon_k.
				\end{split}
			\end{equation}
		}Thus, 
			{\small
			\begin{align}\label{theorem:consensus:bound}
				\small&\textfrac{1}{K}\sum_{k=0}^{K-1}\fronorm{\x^{k+1}-\bar{\x}^{k+1}}\cr
				&\stack{(\ref{theorem:consensus:eqn1})}{\le}\enskip\textfrac{1}{(1-\rho)^2K}\sum_{k=0}^{K-1}\fronorm{\big(\identity-\W\big)\x^{k+1}}\cr
				&\le\textfrac{2}{(1-\rho)^2K}\sum_{k=0}^{K-1}\fronorm{\big(\identity-\W\big)(\x^{k+1}-\x_0^{k+1})}\cr
					&\quad+\textfrac{2}{(1-\rho)^2K}\sum_{k=0}^{K-1}\fronorm{\big(\identity-\W\big)\x_0^{k+1}}\cr
				&\le\textfrac{2}{(1-\rho)^2K}\left(4\sum_{k=0}^{K-1}\fronorm{\x^{k+1}-\x_0^{k+1}}+\sum_{k=0}^{K-1}\fronorm{\big(\identity-\W\big)\x_0^{k+1}}\right)\cr
				&\stack{(\ref{theorem:convergence1:eqn2}),(\ref{theorem:convergence1:eqn3})}{\le}\quad\textfrac{48L^2\eta^2}{(1-\rho)^2K}\sum_{k=0}^{K-1}\fronorm{\x^{k+1}-\x^k}+\textfrac{52}{(1-\rho)^2K}\sum_{k=0}^{K-1}\fronorm{\vec{V}_0^k}\cr
					&\quad+\textfrac{96\eta^2}{(1-\rho)^2K}\sum_{k=0}^{K-1}\epsilon_k\cr
				&\stack{(\ref{fact:youngsinequality})}{\le}\textfrac{48L^2\eta^2}{(1-\rho)^2K}\sum_{k=0}^{K-1}\fronorm{\x^{k+1}-\x^k}\cr
						&\quad+\textfrac{208}{(1-\rho)^2K}\sum_{k=0}^{K-1}\fronorm{\x_0^{k+1}-\x_0^k}+\textfrac{96\eta^2}{(1-\rho)^2K}\sum_{k=0}^{K-1}\epsilon_k\cr
				&\stack{(\ref{theorem:convergence1:bound})}{\le}\textfrac{1}{C_1K}\left(C_2\Delta_{\Phi}+\textfrac{96\eta^2+C_2C_3(1-\rho)^2}{(1-\rho)^2}\sum_{k=0}^{K-1}\epsilon_k\right),
			\end{align}
		}where we have used the fact that $\norm{\identity-\W}_2\le2$ and the choice of $\eta$ in~\eqref{prop:ss:constants_bound} to have $\max\left\{\textfrac{48L^2\eta^2}{(1-\rho)^2},\textfrac{208}{(1-\rho)^2}\right\}=\textfrac{208}{(1-\rho)^2}\triangleq C_2$ and defined $C_3\triangleq\textfrac{(32L+16L(1-\rho))\eta+(4C+1)(1-\rho)}{2L(1-\rho)}$.
		Furthermore, we use (\ref{lemma:dual_variable_relation:z_bound}) and (\ref{lemma:dual_variable_relation:y_bound}) to have
				\begin{equation}\label{theorem:kkt:eqn1}
					\small\nabla F(\x^{k+1})+\sqrt{\identity-\W}\z^{k+1}=\vec{R}^{k+1}-\textfrac{1}{\eta}(\identity+\W)[\x_0^{k+1}-\x_0^k].
				\end{equation}
		Now, by Assumption~\ref{assumption:mixing_matrix}(iii), we have $\basis^\top\sqrt{\identity-\W}=\0.$ Hence
		{\small
				\begin{align}\label{theorem:kkt:bound}
					&\textfrac{1}{K}\sum_{k=0}^{K-1}\fronorm{\nabla f(\bar{\vec{x}}^{k+1})}\cr
					&=\textfrac{1}{K}\sum_{k=0}^{K-1}\fronorm{\textfrac{1}{N}\basis\basis^\top\left(\nabla F(\bar{\x}^{k+1})+\sqrt{\identity-\W}\z^{k+1}\right)}\cr
					&\le\norm{\textfrac{1}{N}\basis\basis^\top}_2^2\textfrac{1}{K}\sum_{k=0}^{K-1}\fronorm{F(\bar{\x}^{k+1})+\sqrt{\identity-\W}\z^{k+1}}\cr
					&\le\textfrac{2}{K}\sum_{k=0}^{K-1}\fronorm{F(\x^{k+1})+\sqrt{\identity-\W}\z^{k+1}}\cr
						&\quad+\textfrac{2}{K}\sum_{k=0}^{K-1}\fronorm{\nabla F(\bar{\x}^{k+1})-\nabla F(\x^{k+1})}\cr
					&\stack{(\ref{theorem:kkt:eqn1}),(\ref{assumption:smoothness})}{\le}\quad\textfrac{2}{K}\sum_{k=0}^{K-1}\fronorm{\vec{R}^{k+1}-\textfrac{1}{\eta}(\identity+\W)[\x_0^{k+1}-\x_0^k]}\cr
						&\quad+\textfrac{2L^2}{K}\sum_{k=0}^{K-1}\fronorm{\x^{k+1}-\bar{\x}^{k+1}}\cr
					&\stack{(\ref{fact:youngsinequality}),(\ref{x_optimality})}{\le}\quad\textfrac{16}{K\eta^2}\sum_{k=0}^{K-1}\fronorm{\x_0^{k+1}-\x_0^k}+\textfrac{4}{K}\sum_{k=0}^{K-1}\epsilon_{k}\cr
						&\quad+\textfrac{2L^2}{K}\sum_{k=0}^{K-1}\fronorm{\x^{k+1}-\bar{\x}^{k+1}}\cr
					&\stack{(\ref{theorem:convergence1:bound}),(\ref{theorem:consensus:bound})}{\le}\quad\textfrac{(2C_2L^2+C_4)\Delta_{\Phi}}{KC_1}\cr
						&\quad+\textfrac{192L^2\eta^2+\left(2C_2C_3L^2+C_3C_4+4C_1\right)(1-\rho)^2}{KC_1(1-\rho)^2}\sum_{k=0}^{K-1}\epsilon_k
				\end{align}
		}where we have used $\norm{\identity+\W}_2\le2$ in the fourth inequality and defined $C_4\triangleq\textfrac{16}{\eta^2}.$
		Finally, we have that
		{\small
			\begin{align}\label{eq:final-bd-stat}
				&\min_{1\le k'\le K}\left(\norm{\nabla f(\bar{\vec{x}}^{k'})}_2^2+\fronorm{\x^{k'}-\bar{\x}^{k'}}\right)\nonumber\\
				&\le\textfrac{1}{K}\sum_{k=0}^{K-1}\left(\norm{\nabla f(\bar{\vec{x}}^{k+1})}_2^2+\fronorm{\x^{k+1}-\bar{\x}^{k+1}}\right)\cr
				&\stack{(\ref{theorem:consensus:bound}),(\ref{theorem:kkt:bound})}{\le}\quad\textfrac{\left((2L^2+1)C_2+C_4\right)\Delta_{\Phi}}{KC_1}+\textfrac{(192L^2+96)\eta^2}{KC_1(1-\rho)^2}\sum_{k=0}^{K-1}\epsilon_k\cr
					&+\textfrac{\left((2L^2+1)C_2C_3+C_3C_4+4C_1\right)}{KC_1}\sum_{k=0}^{K-1}\epsilon_k.
		\end{align}
		}We complete the proof.
\end{proofEnd}

\begin{remark}\label{rm:local_complexity}
Let $k_0=\underset{1\le k\le K}{\argmin}\left(\norm{\nabla f(\bar{\vec{x}}^{k})}_2^2+\fronorm{\x^{k'}-\bar{\x}^{k}}\right)$. Then $\norm{\nabla f(\bar{\vec{x}}^{k_0})}_2^2+\fronorm{\x^{k_0}-\bar{\x}^{k_0}}=\bigO{\textfrac{1}{K}}$. Hence, in order to produce an $\varepsilon$-stationary point as defined in Definition~\ref{def:stationarity}, we need $K=\bigO{\textfrac{1}{\varepsilon}}$ iterations. Furthermore, notice that all the problems in \eqref{x_subproblem} are smooth and strongly convex. The steepest gradient method has linear convergence to solve such problems. Hence, to produce $\vec{x}^{k+1}_i$ as a $\textfrac{\varepsilon_{k+1}}{N}$-accurate solution of the problem in \eqref{x_subproblem}, it needs $\bigO{\log\textfrac{N}{\varepsilon_{k+1}}}$ gradient evaluations for each $i=1,\ldots,N$. Choose $\varepsilon_{k+1}=\textfrac{\epsilon_0}{(k+1)^\gamma}$ for all $k\ge0$ and for some $\gamma>1$ where $\epsilon_0=\bigO{1-\rho}$. Then $\{\varepsilon_{k+1}\}$ is summable, and the total gradient evaluations to produce an $\varepsilon$-stationary point of (\ref{dco_problem}) would be $\sum_{k=0}^{K-1}\bigO{\log N(k+1)^\gamma} = \bigO{\textfrac{1}{\varepsilon}\log\textfrac{N}{\varepsilon^\gamma}}$.
\end{remark}

% ------------------------------------------------------------------------ %
% 									OG Convergence							   %
% ------------------------------------------------------------------------ %

\subsection{Convergence Results of ADAPD-OG}

The convergence rate results of the ADAPD-OG  follow the same logic as the results for ADAPD; proofs are differed to Appendix~\ref{appendix:single_step}. Notice that (\ref{lemma:dual_variable_relation:y_bound}) is no longer a valid relation when Alg.~\ref{algo:single_step} is used. Instead, we have the following from (\ref{y_update}) and (\ref{x_single_step_update}):  
	\begin{equation}\label{ss:dual_variable_relation:y_bound}
		\y^k=-\nabla F(\x^{k-1})-\textfrac{1}{\eta}\left(\x_0^k-\x_0^{k-1}\right), \forall\, k\ge0.
	\end{equation}
As in the analysis for ADAPD, we define $\x_0^{-1}\triangleq\x_0^0$ and further define $\x^{-1}\triangleq\x^0.$ We have the following result.

\begin{theoremEnd}[category=singlestep]{thm}[Convergence of ADAPD-OG]\label{theorem:ss:total_convergence}
	Under Assumptions \ref{assumption:mixing_matrix} and \ref{assumption:objective_function}, let $\{(\x^k,\x_0^k;\y^k,\z^k)\}$ be obtained from Alg.~\ref{algo:single_step} or equivalently by \eqref{x_single_step_update} and \eqref{x0_update}-\eqref{z_update}. Choose $\hat{C}$ and $\eta$ such that
	{\small
			\begin{equation}\label{prop:ss:constants_bound}
			\hat{C}\triangleq\textfrac{16}{(1-\rho)^2}\text{ and }
			\eta<\textfrac{1}{2\hat{C}L}.
			\end{equation}
		}Then, it holds
	{\small
		\begin{align*}
			\textfrac{1}{K}\sum_{k=0}^{K-1}\left(\norm{\nabla f(\bar{\vec{x}}^{k+1})}_2^2+\fronorm{\x^{k+1}-\bar{\x}^{k+1}}\right)&\le\textfrac{\left((2L^2+1)\hat{C}_2+\hat{C}_3\right)\Delta_{\hat{\Phi}}}{\hat{C}_1K}
		\end{align*}
	}where $\small\hat{C}_1\triangleq\textfrac{L}{(1-\rho)^2}\le\textfrac{(1-\rho)-(\hat{C}+1)L(1-\rho)\eta-((1-\rho)+1)4L^2\eta^2}{2(1-\rho)\eta}$, $\hat{C}_2\triangleq\textfrac{112}{(1-\rho)^2}$, $\hat{C}_3\triangleq\textfrac{8}{\eta^2}$, $\Delta_{\hat{\Phi}}\triangleq\hat{\Phi}^0-\underline{f}+1$, $\bar{\vec{x}}^{k}\triangleq\textfrac{1}{N}\sum\limits_{i=1}^N\vec{x}_i^k$, and $\bar{\x}^{k}\triangleq\textfrac{1}{N}\basis\basis^\top\x^{k}$.
\end{theoremEnd}

\begin{proofEnd}\enskip 
		By (\ref{z_update}), we have
			\begin{equation}\label{theorem:ss:total_convergence:eqn1}
			\begin{split}
				\textfrac{1}{\eta}\fronorm{\big(\identity-\W\big)\x_0^{k+1}}&=\eta\fronorm{\sqrt{\identity-\W}(\z^{k+1}-\z^k)}\\
				&\le4L^2\eta\fronorm{\x^{k}-\x^{k-1}}+\textfrac{6}{\eta}\fronorm{\vec{V}_0^k}
			\end{split}
			\end{equation}
		and by (\ref{y_update}),
			\begin{equation}\label{lemma:ss:total_convergence:eqn2}
			\begin{split}
				\fronorm{\x^{k+1}-\x_0^{k+1}}&=\eta^2\fronorm{\y^{k+1}-\y^k}\\
				&\stack{(\ref{lemma:ss:dual_var_bound:z_bound})}{\le}2L^2\eta^2\fronorm{\x^{k}-\x^{k-1}}+2\fronorm{\vec{V}_0^k}.
			\end{split}
			\end{equation}
		Thus, 
			\begin{align}\label{theorem:ss:consensus_bound}
				&\textfrac{1}{K}\sum_{k=0}^{K-1}\fronorm{\x^{k+1}-\bar{\x}^{k+1}}\nonumber\\
				&\stack{(\ref{theorem:consensus:eqn1})}{\le}\enskip\textfrac{1}{(1-\rho)^2K}\sum_{k=0}^{K-1}\fronorm{\big(\identity-\W\big)\x^{k+1}}\nonumber\\
				&\le\textfrac{2}{(1-\rho)^2K}\sum_{k=0}^{K-1}\fronorm{\big(\identity-\W\big)(\x^{k+1}-\x_0^{k+1})}\\
					&\quad+\textfrac{2}{(1-\rho)^2K}\sum_{k=0}^{K-1}\fronorm{\big(\identity-\W\big)\x_0^{k+1}}\nonumber\\
				&\le\textfrac{2}{(1-\rho)^2K}\left(4\sum_{k=0}^{K-1}\fronorm{\x^{k+1}-\x_0^{k+1}}+\sum_{k=0}^{K-1}\fronorm{\big(\identity-\W\big)\x_0^{k+1}}\right)\nonumber\\
				&\le\textfrac{24L^2\eta^2}{(1-\rho)^2K}\sum_{k=0}^{K-1}\fronorm{\x^{k+1}-\x^k}+\textfrac{28}{(1-\rho)^2K}\sum_{k=0}^{K-1}\fronorm{\vec{V}_0^k}\nonumber\\
				&\stack{(\ref{fact:youngsinequality})}{\le}\textfrac{48L^2\eta^2}{(1-\rho)^2K}\sum_{k=0}^{K-1}\fronorm{\x^{k+1}-\x^k}\nonumber\\
					&\quad+\textfrac{112}{(1-\rho)^2K}\sum_{k=0}^{K-1}\fronorm{\x_0^{k+1}-\x_0^k}\nonumber\\
				&\stack{(\ref{theorem:ss:convergence:eqn1})}{\le}\textfrac{\hat{C}_2\Delta_{\hat{\Phi}}}{\hat{C}_1K}
			\end{align}
		where we have used the fact that $\norm{\identity-\W}_2\le2$ and defined $\hat{C}_2\triangleq\max\left\{\textfrac{48L^2\eta^2}{(1-\rho)^2},\textfrac{112}{(1-\rho)^2}\right\}$.
		Furthermore, we use (\ref{lemma:dual_variable_relation:z_bound}) and (\ref{ss:dual_variable_relation:y_bound}) to have
				\begin{equation}\label{theorem:ss:kkt:eqn1}
					\nabla F(\x^{k})+\sqrt{\identity-\W}\z^{k+1}=-\textfrac{1}{\eta}(\identity+\W)[\x_0^{k+1}-\x_0^k].
				\end{equation}
		Now, using Assumption~\ref{assumption:mixing_matrix}(ii), we have $\basis^\top\sqrt{\identity-\W}=\0$. Hence,
				\begin{align}\label{theorem:ss:kkt:bound}
					&\textfrac{1}{K}\sum_{k=0}^{K-1}\fronorm{\nabla f(\bar{\vec{x}}^{k+1})}\nonumber\\
					&=\textfrac{1}{K}\sum_{k=0}^{K-1}\fronorm{\textfrac{1}{N}\basis\basis^\top\left(\nabla F(\bar{\x}^{k+1})+\sqrt{\identity-\W}\z^{k+2}\right)}\nonumber\\
					&\le\norm{\textfrac{1}{N}\basis\basis^\top}_2^2\textfrac{1}{K}\sum_{k=0}^{K-1}\fronorm{F(\bar{\x}^{k+1})+\sqrt{\identity-\W}\z^{k+2}}\nonumber\\
					&\le\textfrac{2}{K}\sum_{k=0}^{K-1}\fronorm{F(\x^{k+1})+\sqrt{\identity-\W}\z^{k+2}}\nonumber\\
						&\quad+\textfrac{2}{K}\sum_{k=0}^{K-1}\fronorm{\nabla F(\bar{\x}^{k+1})-\nabla F(\x^{k+1})}\nonumber\\
					&\stack{(\ref{theorem:ss:kkt:eqn1}),(\ref{assumption:smoothness})}{\le}\quad\textfrac{2}{K}\sum_{k=0}^{K-1}\fronorm{-\textfrac{1}{\eta}(\identity+\W)[\x_0^{k+2}-\x_0^{k+1}]}\nonumber\\
						&\quad+\textfrac{2L^2}{K}\sum_{k=0}^{K-1}\fronorm{\x^{k+1}-\bar{\x}^{k+1}}\nonumber\\
					&\stack{(\ref{theorem:ss:convergence:eqn1}),(\ref{theorem:ss:consensus_bound})}{\le}\quad\textfrac{\left(2\hat{C}_2L^2+\hat{C}_3\right)\Delta_{\hat{\Phi}}}{\hat{C}_1K},
				\end{align}
			where we have used $\norm{\identity+\W}_2\le2$ in last inequality and defined $\hat{C}_3\triangleq\textfrac{8}{\eta^2}$. 
	Finally, we have that
		\begin{align*}
			&\min_{1\le k'\le K}\left(\norm{\nabla f(\bar{\vec{x}}^{k'})}_2^2+\fronorm{\x^{k'}-\bar{\x}^{k'}}\right)\\
			&\le\textfrac{1}{K}\sum_{k=0}^{k-1}\left(\norm{\nabla f(\bar{\vec{x}}^{k+1})}_2^2+\fronorm{\x^{k+1}-\bar{\x}^{k+1}}\right)\\
			&\stack{(\ref{theorem:ss:consensus_bound}),(\ref{theorem:ss:kkt:bound})}{\le}\quad\textfrac{\left((2L^2+1)\hat{C}_2+\hat{C}_3\right)\Delta_{\hat{\Phi}}}{\hat{C}_1K}.
		\end{align*}
		We complete the proof.
\end{proofEnd}

\begin{remark}\label{rm:big_0}
Theorems~\ref{theorem:total_convergence} and ~\ref{theorem:ss:total_convergence} give the convergence results in terms of the constants $C_1,C_2,C_3,$ and $C_4$ for Alg.~\ref{algo:} (or $\hat{C}_1,\hat{C}_2,\hat{C}_3$, and $\hat{C}_4$ for Alg.~\ref{algo:single_step}) which depend on $C$ ($\hat{C}$) and $\eta$, and in turn depend on $L$ and $\rho.$ To make this dependency clearer, we use the $\bigO{\cdot}$ notation to give dependency only in terms of $L,\rho,$ and the algorithm iteration number $K$. For Alg.~\ref{algo:}, using~\eqref{error_summation}, and for Alg.~\ref{algo:single_step}, we have
	\begin{equation}\label{rm:big_0:bound}
	\textfrac{1}{K}\sum_{k=0}^{K-1}\left(\norm{\nabla f(\bar{\vec{x}}^{k+1})}_2^2+\fronorm{\x^{k+1}-\bar{\x}^{k+1}}\right)=\bigO{\textfrac{L}{(1-\rho)^2K}}.
	\end{equation}
\end{remark}

% ------------------------------------------------------------------------ %
% 							Complexity Results						   %
% ------------------------------------------------------------------------ %
\subsection{Complexity Analysis}\label{sec:complexity}

We now give a complexity analysis for Alg.'s~\ref{algo:} and~\ref{algo:single_step} regarding the number of primal gradient computations and neighbor communications each method must perform to find an $\varepsilon$-stationary point (see Definition~\ref{def:stationarity}); we refer to these quantities as the \emph{computation} and \emph{communication} complexities, respectively. This leads to the following corollaries, whose proofs are in Appendix~\ref{appendix:complexity}.
% Complexty of ADAPD
\begin{theoremEnd}[category=complexity]{corollary}[Complexity results of ADAPD]\label{corollary:complexity_adapd}
	Under the same conditions assumed in Theorem~\ref{theorem:total_convergence}, if steepest gradient descent is used to solve the subproblem~\eqref{x_subproblem}, such that conditions~\eqref{x_optimality} and~\eqref{error_summation} hold, then Alg.~\ref{algo:} can produce an $\varepsilon$-stationary point in respectively
		\begin{equation}\label{corollary:complesity_adapd:eqn}
			\tilde{\mathcal{O}}\left(\textfrac{L}{(1-\rho)^2\varepsilon}\right)\text{ and }\bigO{\textfrac{L}{(1-\rho)^2\varepsilon}}
		\end{equation}
	gradient computations\footnote{The $\tilde{\mathcal{O}}(\cdot)$ hides a log dependency on $\varepsilon$ here.} and neighbor communications.
\end{theoremEnd}
\begin{proofEnd}\enskip
	By Remark~\ref{rm:big_0}, since one communication round is performed during each iteration of Alg.~\ref{algo:}, the communication complexity in~\eqref{corollary:complesity_adapd:eqn} follows from setting~\eqref{rm:big_0:bound} less than or equal to $\varepsilon$ and solving for $K$. Remark~\ref{rm:local_complexity} demonstrates the additional logarithmic dependence on the number of gradient computations.
\end{proofEnd}

% Complexity of ADAPD-OG
\begin{theoremEnd}[category=complexity]{corollary}[Complexity results of ADAPD-OG]\label{corollary:complexity_adapd_og}
	Under the same conditions assumed in Theorem~\ref{theorem:ss:total_convergence}, Alg.~\ref{algo:single_step} can produce an $\varepsilon$-stationary point in
		\begin{equation}\label{corollary:complesity_adapd_og:eqn}
			\bigO{\textfrac{L}{(1-\rho)^2\varepsilon}}
		\end{equation}
	gradient computations and neighbor communications.
\end{theoremEnd}
\begin{proofEnd}\enskip
	Since only one communication and one gradient computation are performed during Alg.~\ref{algo:single_step},~\eqref{corollary:complesity_adapd_og:eqn} follows from setting~\eqref{rm:big_0:bound} less than or equal to $\varepsilon$ and solving for $K$.
\end{proofEnd}

Corollaries~\ref{corollary:complexity_adapd} and~\ref{corollary:complexity_adapd_og} show that both ADAPD and ADAPD-OG depend upon the quantity $(1-\rho)^{-2}$ in terms of the number of communications required to achieve $\varepsilon$-stationarity. To improve this to the optimal communication complexity in terms of the dependence on $\rho$ (see, e.g.~\cite{scaman17}), we have the following theorem.

% Complexity of MC variant
\begin{theoremEnd}[category=complexity]{thm}[Complexity results of ADAPD-MC]\label{theorem:complexity_mc}
	Under the same conditions assumed in Theorem~\ref{theorem:total_convergence}, let $R=\lceil\textfrac{2}{\sqrt{1-\rho}}\rceil$ iterations of the Chebyshev acceleration Alg.~\ref{algo:cheby} be performed during the line 9 update of Alg.~\ref{algo:}. Then Alg.~\ref{algo:} can produce an $\varepsilon$-stationary point in $\tilde{\mathcal{O}}\left(\textfrac{L}{\varepsilon}\right)$ and $\bigO{\textfrac{L}{\sqrt{1-\rho}\varepsilon}}$ gradient computations and neighbor communications, respectively.
\end{theoremEnd}
\begin{proofEnd}\enskip
	By Lemma~\ref{lemma:chebyshev}, the dependence on the spectrum of the graph after $R$ iterations of Alg.~\ref{algo:cheby} becomes $2\left(1-\sqrt{1-\rho}\right)^R$; define this quantity to be $\rho_R\triangleq 2\left(1-\sqrt{1-\rho}\right)^R$ such that~\eqref{rm:big_0:bound} becomes
				\begin{equation}\label{theorem:complexity_mc:eqn1}
					\textfrac{1}{K}\sum_{k=0}^{K-1}\left(\norm{\nabla f(\bar{\vec{x}}^{k+1})}_2^2+\fronorm{\x^{k+1}-\bar{\x}^{k+1}}\right)=\bigO{\textfrac{L}{(1-\rho_R)^2K}}.
				\end{equation}
		With $R=\lceil\textfrac{2}{\sqrt{1-\rho}}\rceil$, we find a $u>0$ such that,
			\begin{align*}
				\textfrac{1}{\left(1-2(1-\sqrt{1-\rho})^{\lceil\frac{2}{\sqrt{1-\rho}}\rceil}\right)^2}\le u.
			\end{align*}
		First, we rearrange to have
			\begin{align*}
				\left(1-\sqrt{1-\rho}\right)^{\lceil\frac{2}{\sqrt{1-\rho}}\rceil}\le\textfrac{\sqrt{u}-1}{2\sqrt{u}}.
			\end{align*}
		Now, let $x=\sqrt{1-\rho}\in(0,1]$, then $(1-x)\in[0,1)$ and $\textfrac{2}{x}\le\lceil\textfrac{2}{x}\rceil$ so that
			\begin{align*}
				\left(1-x\right)^{\lceil\frac{2}{x}\rceil}\le\left(1-x\right)^{\frac{2}{x}}.
			\end{align*}
		Next, we maximize this quantity with respect to $x\in(0,1].$ Define $g(x)\triangleq\left(1-x\right)^{\frac{2}{x}}$ and compute $\textfrac{d}{dx}g(x)$ to have
			\begin{align*}
				\textfrac{d}{dx}g(x)=-\left(1-x\right)^{\frac{2}{x}}\left(\textfrac{2}{x(1-x)}+\frac{2\ln(1-x)}{x^2}\right) < 0, \forall\, x \in (0,1).
			\end{align*}
		Hence, $g(x)$ is decreasing on $(0,1)$. Since $g(0+) = \textfrac{1}{e^2}$, we have $g(x)<\textfrac{1}{e^2}$ for $x\in(0,1].$ Now we compute,
			\begin{align*}
				\textfrac{1}{e^2}\le\textfrac{\sqrt{u}-1}{2\sqrt{u}},
			\end{align*}
		which holds for all $u\ge 2.$ Thus, it holds that $(1-\rho_R)^{-2}\le2.$ Hence we have the number of gradient computations is independent of $\rho_R$ and the number of neighbor communications must be multiplied by $R=\bigO{\textfrac{1}{\sqrt{1-\rho}}}.$
\end{proofEnd}

% Complexity results of the OG-MC
\begin{theoremEnd}[category=complexity]{thm}[Complexity results of ADAPD-OG-MC]\label{theorem:complexity_og_mc}
	Under the same conditions assumed in Theorem~\ref{theorem:ss:total_convergence}, let $R=\lceil\textfrac{2}{\sqrt{1-\rho}}\rceil$ iterations of the Chebyshev acceleration Alg.~\ref{algo:cheby} be performed during the line 9 update of Alg.~\ref{algo:single_step}. Then Alg.~\ref{algo:single_step} can produce an $\varepsilon$-stationary point in $\mathcal{O}\left(\textfrac{L}{\varepsilon}\right)$ and $\bigO{\textfrac{L}{\sqrt{1-\rho}\varepsilon}}$ gradient computations and neighbor communications, respectively.
\end{theoremEnd}
\begin{proofEnd}\enskip
	The proof follows the same logic as the proof of Theorem~\ref{theorem:complexity_mc}.
\end{proofEnd}
% ------------------------------------------------------------------------ %
% 							Numerical Experiments					   %
% ------------------------------------------------------------------------ %
\section{Numerical Experiments}\label{sec:numerical}

We test our proposed methods on several non-convex problems: (i) a binary classification problem using logistic regression with a non-convex regularizer, (ii) a multi-target cooperative localization problem, and (iii) two image classification problems using convolutional neural networks. The experiments serve to verify both the flexibility of our methods, as well as their numerical superiority over other decentralized optimization methods. Implementations of our methods are made available at~\url{https://github.com/RPI-OPT/ADAPD}.

For experiments (i) and (ii), we compare our methods to DGD with a diminishing step-size~\cite{zeng18} and the single gradient version of Prox-PDA, called Prox-GPDA~\cite{hong17}. We also ran experiments with Prox-PDA but found no advantage over using Prox-PDA versus Prox-GPDA; since Prox-GPDA only requires one gradient computation per update, we use this as a baseline. For Alg.~\ref{algo:}, we use $\epsilon_{k}=\frac{\hat{\epsilon}}{(k+1)^d}$ in (\ref{x_optimality}) where $\hat{\epsilon}$ and $d$ are tuned from a fixed set of values and solve each agent's  local problem (\ref{x_subproblem}) by the FISTA~\cite{beck09} method. For experiment (iii), we compare to D-PSGD~\cite{lian17}, DSGT~\cite{zhang20gradtrack}, D-GET~\cite{sun20}, a single stochastic gradient implementation of Prox-PDA~\cite{hong17}, and SPPDM~\cite{wang21}. For all experiments, we fix a set of penalty parameters/step-sizes and optimize each algorithm over this set, choosing whichever penalty/step-size performs the best. For all methods besides Prox-(G)PDA and SPPDM, we use the same mixing matrix, which will be described in each subsection below. For Prox-(G)PDA, we take $\W$ to be the formulation as given in~\cite{hong17} (see equation (23) in~\cite{hong17} and the discussion that follows) and for SPPDM, we use the graph Laplacian as stated in their problem formulation.

% ------------------------------------------------------------------------ %
% 							Logistic Regression							  	    %
% ------------------------------------------------------------------------ %
\subsection{Non-convex Regularized Logistic Regression}

We consider the non-convex decentralized binary classification problem~\cite{hong17,zhang21fedpd}. Utilizing a logistic regression formulation, the local agent cost functions are given by,
	\begin{equation}\label{numerical:logistic_regression}
		\small f_i(\vec{x}_i)=\textfrac{1}{m_i}\sum_{j=1}^{m_i}\log\left(1+\exp(-b_j\ip{\vec{x}_i,\vec{a}_j})\right)+\sum_{d=1}^{D}\textfrac{\alpha(\vec{x}_i[d])^2}{1+(\vec{x}_i[d])^2}
	\end{equation}
where $\vec{x}_i[d]$ denotes the $d^{th}$ component of the vector $\vec{x}_i.$ Given a set of data $\{(\vec{a}_j,b_j)\}_{j=1}^{m_i}$ for all $i=1,\dots,N$, where $b_j\in\{-1,+1\}$ denotes a particular class label, (\ref{numerical:logistic_regression}) can be used to perform binary classification and the non-convex regularizer, $\sum_{d=1}^{D}\textfrac{\alpha(\vec{x}_i[d])^2}{1+(\vec{x}_i[d])^2}$ helps to promote sparsity on the solutions. We use the a9a dataset~\cite{chang11,dua17} which consists of $32{,}561$ training data points and $16{,}281$ testing data points. Each data point $\vec{a}_j\in\R^{123}$ contains numerical features about adults from the 1994 Census database and $b_j$ indicates whether or not the adults earn more or less than \$$50{,}000$ per year. We fix $N=50$ for this experiment and simulate agent connectivity in two ways: (i) using a ring-structured graph and (ii) using a random Erd\"os R\'enyi graph, with connection probability equal to 0.3 (i.e. each agent is connected to roughly 15 other agents). 

For the ring-structured graph, we choose $\W$ to be
	\begin{align*}
		w_{ij}&=\begin{cases}
			\textfrac{1}{2},&i=j,\\
			\textfrac{1}{4}, &(i,j)\in\edges\text{ and }i\ne j,\\
			0,&\text{otherwise},
		\end{cases}
\end{align*}
and for the random Erd\"os R\'enyi graph, we use the Laplacian-based constant edge weight matrix from (\ref{mixing_mat:laplacian}). We vary $\alpha\in\{0.01,1.0\}$ to study the effect that the non-convex term has on each agent's local subproblem. For all runs, we fix the communication budget to 500 neighbor communications. Additionally, we compare ADAPD-MC and ADAPD-OG-MC to the other methods. We perform 5 iterations of the Chebyshev acceleration in Alg.~\ref{algo:cheby} during every outer iteration of Alg.~\ref{algo:} for ADAPD-MC and 2 iterations for ADAPD-OG-MC. This means we only compute 250 gradients for ADAPD-OG-MC, to keep with the 500 communication budget. We report the $\varepsilon$-stationarity violation (\ref{def:stationarit:eqn}) for the following four scenarios: (i) the random Erd\"os R\'enyi graph with $\alpha=1.0$, (ii) the random Erd\"os R\'enyi graph with $\alpha=0.01$, (iii) the ring graph with $\alpha=1.0,$ and (iv) the ring graph with $\alpha=0.01.$

% ----- LOGISTIC REGRESSION RESULTS ----- %
\begin{figure*}[h]
		
		\centering
		\setlength\tabcolsep{1pt}
		\begin{tabular}{cccc}
			
							\includegraphics[width=0.22\linewidth]{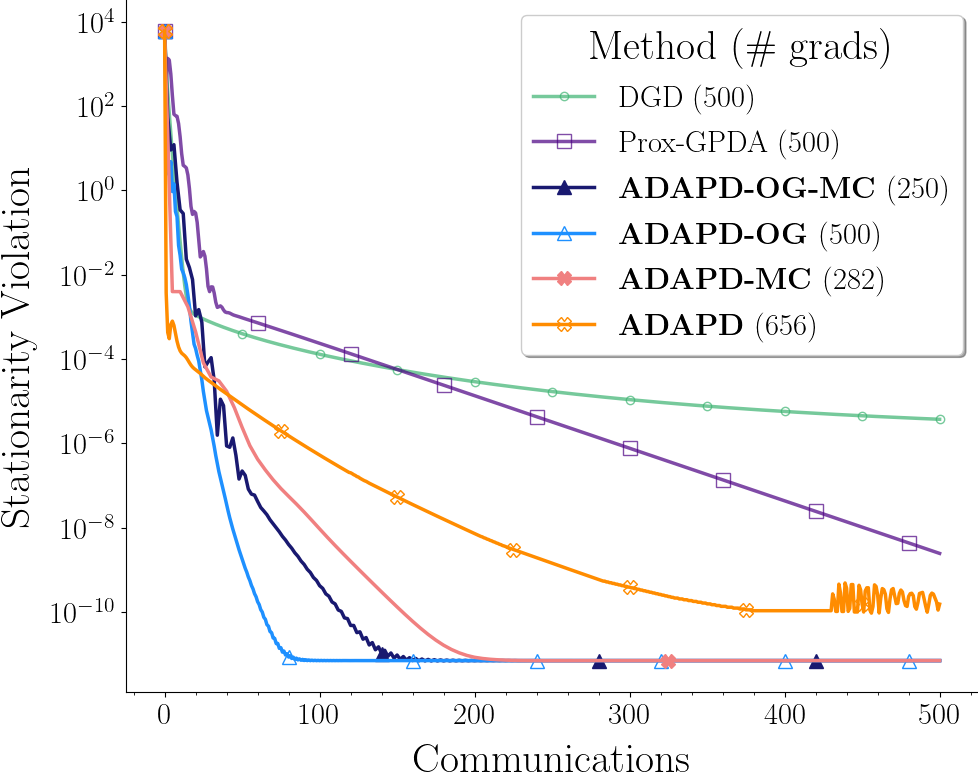} &
		
							\includegraphics[width=0.22\linewidth]{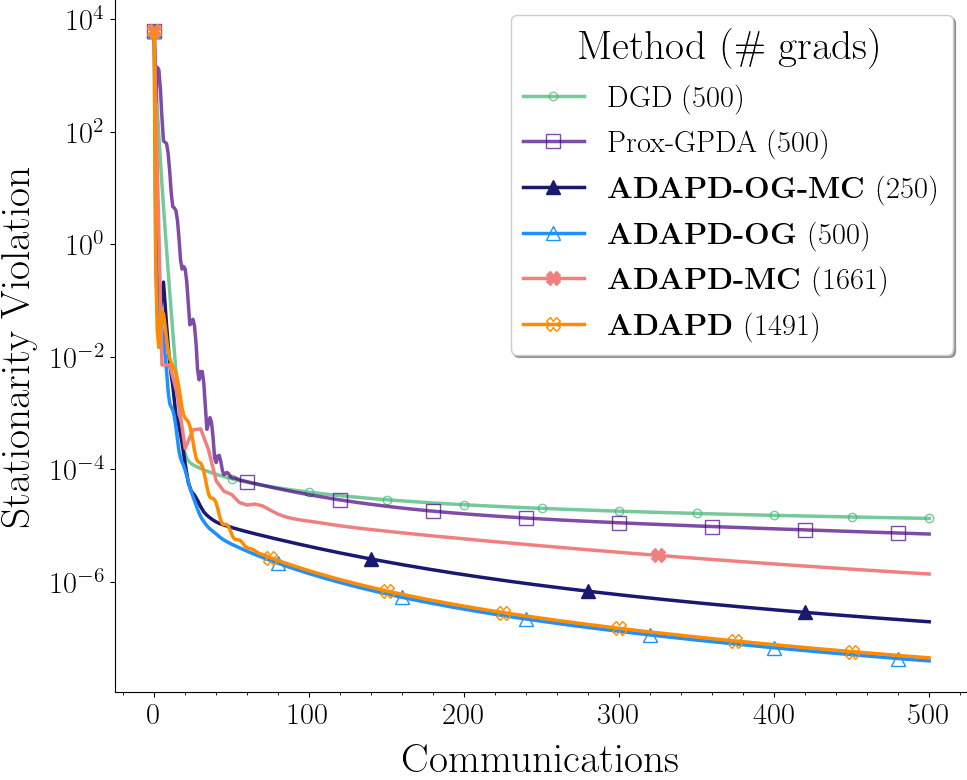} &
		
							\includegraphics[width=0.22\linewidth]{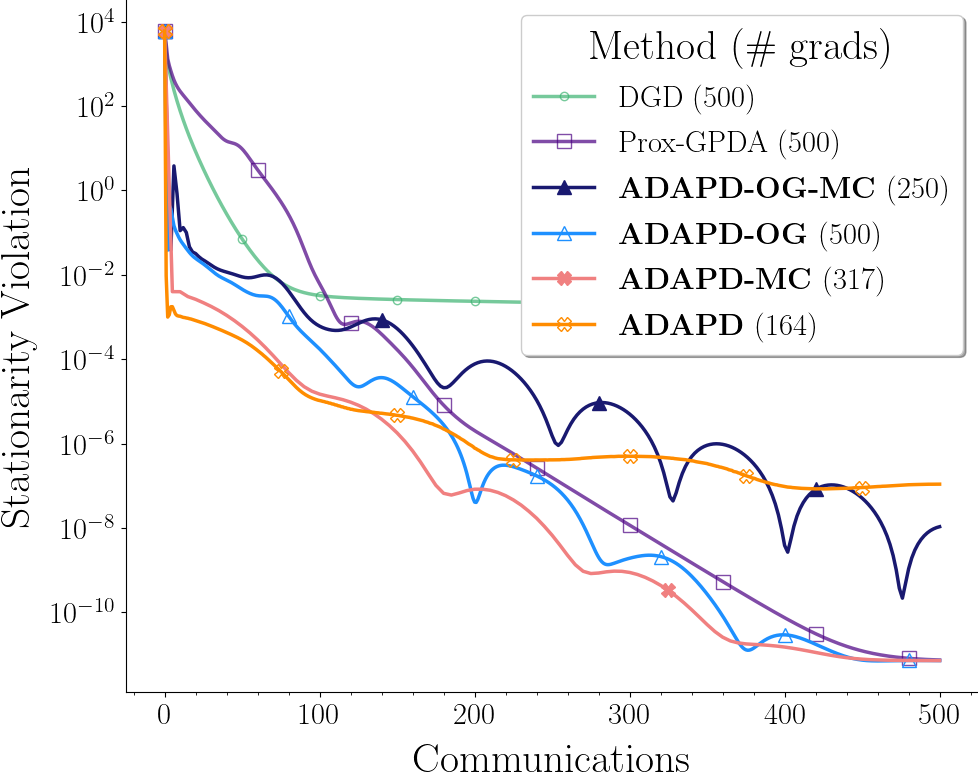} &
		
							\includegraphics[width=0.22\linewidth]{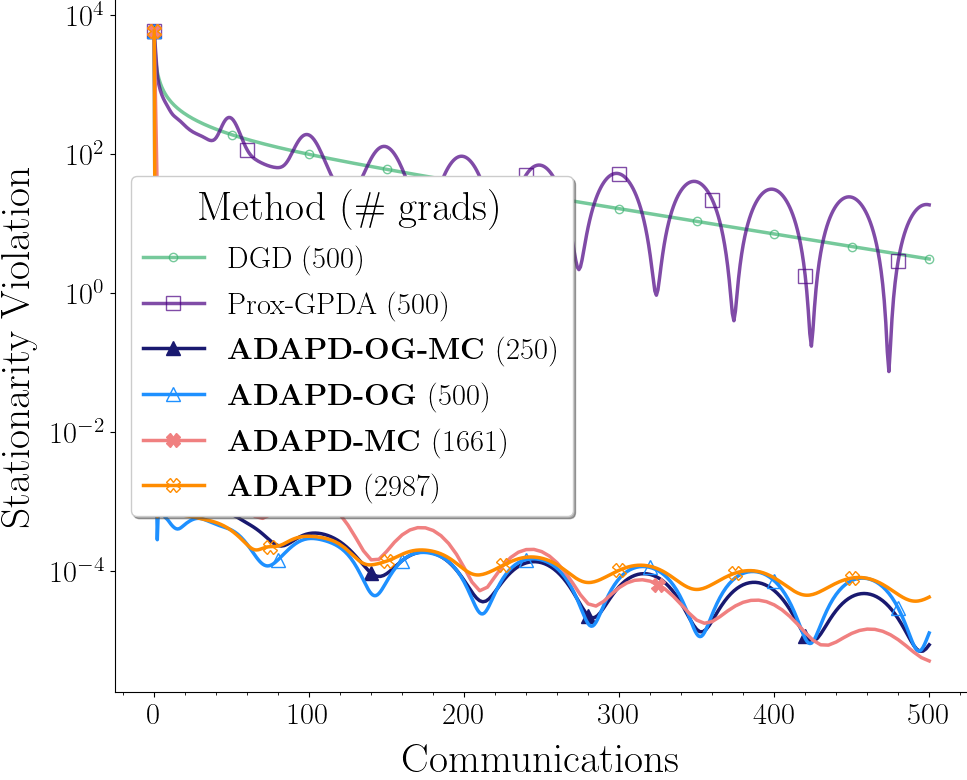} 
		
			\end{tabular}
				
			\caption{Stationarity violation for the non-convex logistic regression problem (in order from left to right): random Erd\"os R\'enyi graph with $\alpha=1.0$, random Erd\"os R\'enyi graph with $\alpha=0.01$, ring-structured graph with $\alpha=1.0,$ and ring-structured graph with $\alpha=0.01.$}
				
			\label{fig:log_reg_results}
		
\end{figure*}

From Figure~\ref{fig:log_reg_results}, it is evident that when the communication pattern is sparse (i.e. the two rightmost plots), performing multiple communications and multiple local updates can reduce the stationarity violation faster over performing just one neighbor communication or just one local update. When the communication pattern is not too sparse (i.e. the two leftmost plots), ADAPD-OG performs significantly better and requires fewer gradients than the other methods compared here.  In all cases, ADAPD and it's variants outperform DGD and Prox-GPDA.

% ------------------------------------------------------------------------ %
% 						Cooperative Localization						%
% ------------------------------------------------------------------------ %
\subsection{Multi-Target Cooperative Localization}

Multi-target cooperative localization is a target locating problem~\cite{chen12}: given only a noisy distance metric, can $N$ agents locate $N_T$ common targets? Let $\{\omega_{i}\}_{i=1}^N$ be a set of locations of the agents, i.e. $\omega_i\in\R^2$ for all $i=1,\dots,N.$ Then the local objective function for each agent is given by
	\begin{equation}\label{numerical:cooperative_localization}
		f_{i}(\vec{x}_i)=\textfrac{1}{4}\sum_{t=1}^{N_T}\left(\xi_{i,t}-\norm{\vec{x}_i[t]-\omega_i}_2^2\right)^2
	\end{equation}
where $\xi_{i,t}$ is a random variable that represents a noisy distance metric, and $\vec{x}_i=\mymatrix{\vec{x}_i[1]^\top&\dots&\vec{x}_i[N_T]^\top}^\top\in\R^{2N_T}$ is a stacking of the vectors $\{\vec{x}_i[t]\}_{t=1}^{N_T}$. Note that (\ref{numerical:cooperative_localization}) is indeed non-convex, but it is not globally $L$-smooth for any $L\ge0.$ However, we still find this problem is valuable to test our methods. Denote the true targets as $\vec{x}^*[t]$ for all $t=1,\dots,N_T;$ these are used to generate $\xi_{i,t}$ for all $i$ and $t$ by computing $\xi_{i,t}=\norm{\vec{x}^*[t]-\omega_{i}}_2^2+\epsilon_{i,t}$ where $\epsilon_{i,t}$ is drawn from a normal distribution with mean 0 and variance $\sigma^2>0.$ For all of our experiments we set $\sigma^2=0.01.$ We simulate agent connectivity by randomly generating $N=50$ agents in $[-1,1]\times[-1,1]$ grid and creating an edge between agents if the Euclidean distance between them is less than or equal to 0.3. Each coordinate in the targets $\{\vec{x}^*[t]\}_{t=1}^{N_T}$ is drawn independently from a normal distribution with mean 0 and variance 0.1. Figure~\ref{fig:coop_loc_results} shows the connectivity of the agents, as well as an example of target locations.

% ----- COOPERATIVE LOCALIZATION RESULTS ----- %
\begin{figure*}[h]
		
		\centering
		\setlength\tabcolsep{1pt}
		\begin{tabular}{cccc}
			
							\includegraphics[width=0.22\linewidth]{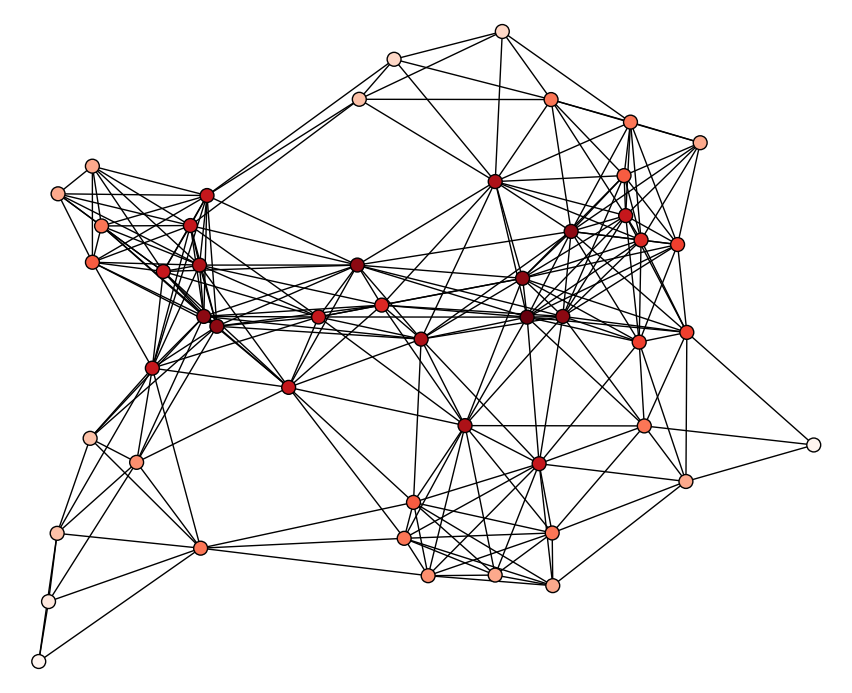} &
		
							\includegraphics[width=0.22\linewidth]{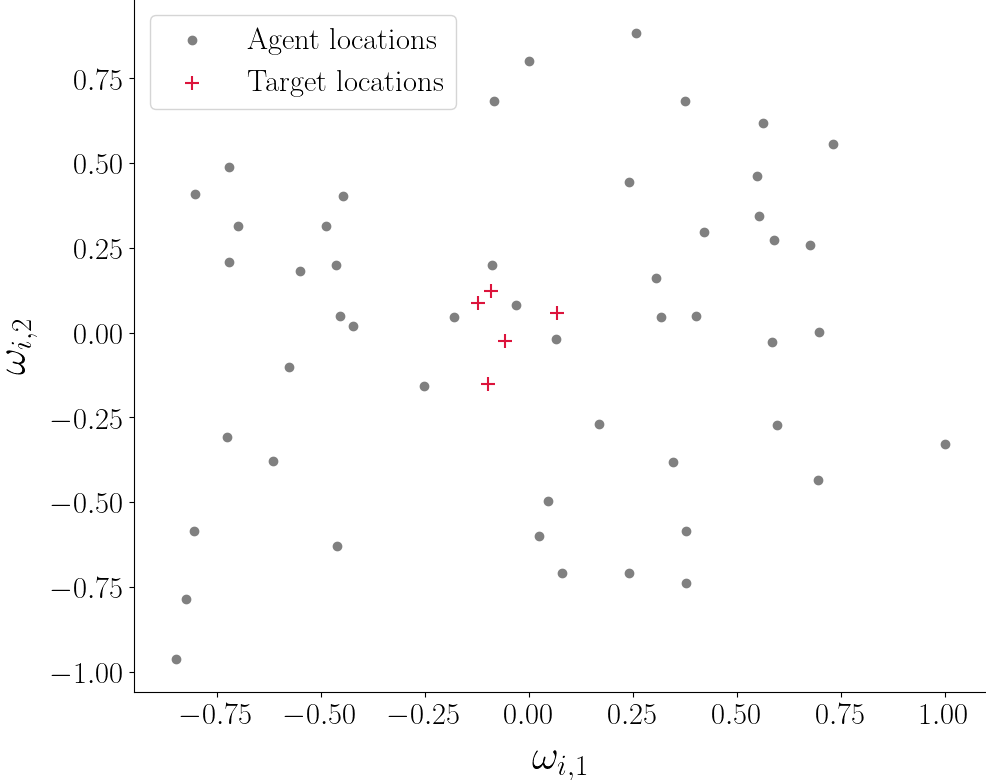} &
		
							\includegraphics[width=0.22\linewidth]{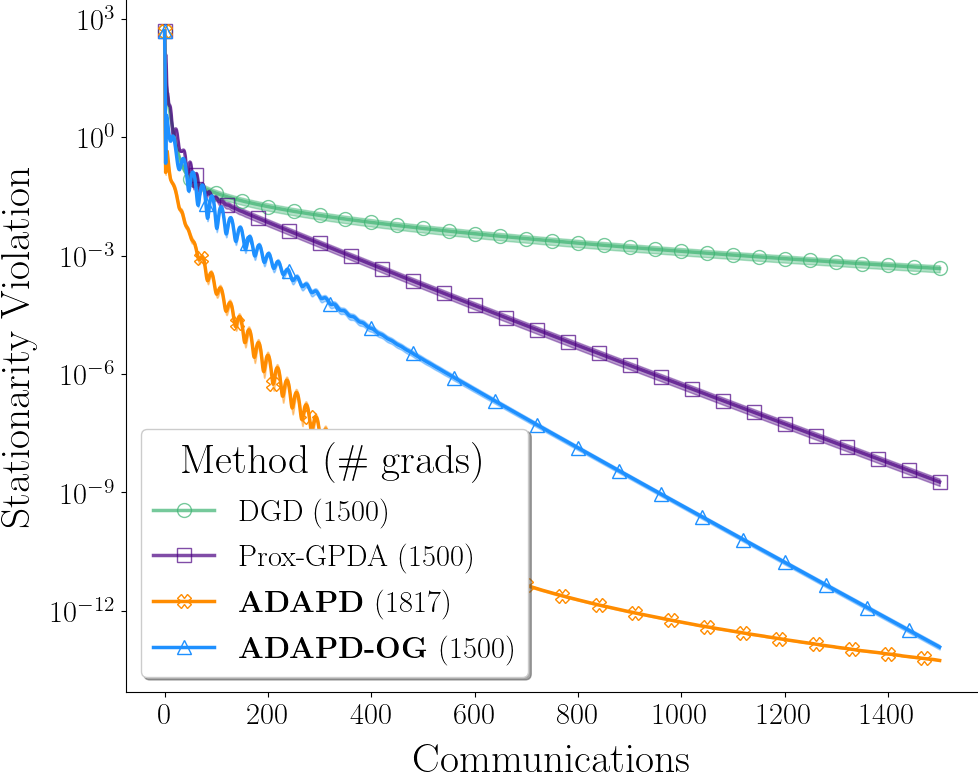} &
		
							\includegraphics[width=0.22\linewidth]{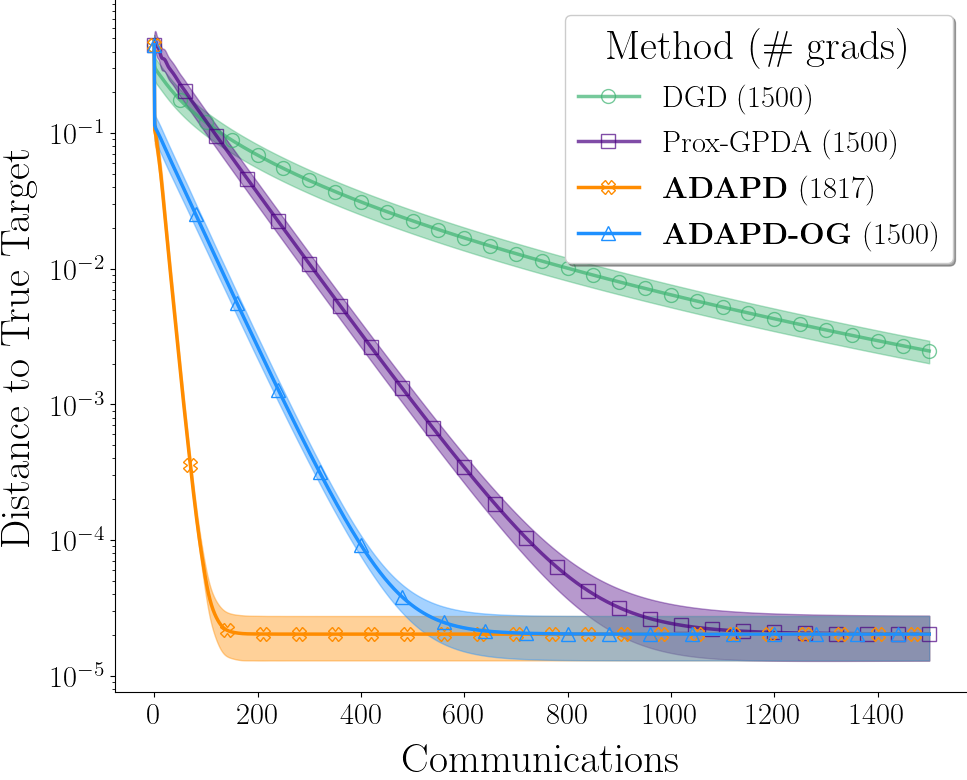} 
		
			\end{tabular}
				
			\caption{In order from left to right: agent locations and their connectivity (darker colors indicate more connections), example of target locations, stationarity violation, and distance to true targets for the multi-target cooperative localization problem.}
				
			\label{fig:coop_loc_results}
		
\end{figure*}

For this example, $\W$ is chosen to be the Laplacian-based constant edge weight matrix from (\ref{mixing_mat:laplacian}). We randomly generate $N_T=5$ targets and limit the communications to be 1${,}$500 for all algorithm runs, with all methods starting from the same initial point. Since the targets are randomly generated for each experiment, we perform 10 independent trials and plot the mean results, with an associated 95\% confidence interval.

Figure~\ref{fig:coop_loc_results} shows that in terms of stationarity violation, ADAPD is superior to DGD and Prox-GPDA. Using only 20\% more gradient computations on each agent, ADAPD is able to both solve the localization problem and find the true targets with fewer communications than the other methods. Additionally, ADAPD-OG utilizes the same number of gradient computations and neighbor communications as Prox-GPDA and DGD, but still performs better.

% ------------------------------------------------------------------------ %
% 											CNN										  	     %
% ------------------------------------------------------------------------ %
\subsection{Convolutional Neural Networks}

For these experiments, we fix $N=8$ agents and use a ring-structured graph where self-weighting and neighbor weighting is set to be $\textfrac{1}{3}$. We train the models on a cluster of 8 NVIDIA Tesla V100 GPUs, where each GPU represents an agent. PyTorch is used in the training of the models and OpenMPI is used to perform the neighbor communication of the neural network weights. All experiments are performed with 10 different initial starting points. We report the average results, as well as a 95\% confidence interval taken over the 10 trials.

\subsubsection{MNIST}

The first Convolutional Neural Network (CNN) experiment we perform is training LeNet~\cite{lecun98} on the MNIST dataset. We make the activation function for each layer the hyperbolic tangent function to ensure smoothness of the local objective functions. Since methods like DSGT~\cite{zhang20gradtrack} and D-GET~\cite{sun20} require multiple neighbor communications during each update, we instead fix the number of \emph{epochs} for this experiment to 50 and fix the mini-batch size to 64 for all methods. We randomly generate 10 sets of initial points for the agents and report the average of all relevant metrics, as well as a 95\% confidence interval. For ADAPD and ADAPD-OG, we simply replace the full gradient computation by a stochastic gradient during each local agent update. It is worth noting that neither ADAPD, nor Prox-PDA, have theoretical convergence guarantees in this experimental setting. Nonetheless, we see impressive results for this problem and thus include it. To see the effect of stochasticity here, we run the ADAPD in Alg.~\ref{algo:} by computing both one and two stochastic gradients during line 3. For Prox-PDA, we compute one stochastic gradient step. Similar to~\cite{lian17}, we report the stationarity violation for all methods, as well as the training loss and testing accuracy using the average of the local agent's weights\footnote{Similar results for both the MNIST and CIFAR-10 image classification problems are observed if the local weights are used to compute the training loss and testing accuracy.}. In practice, this is not feasible due to the decentralized communication pattern, however, an average model can be obtained after all local training has been done by performing many neighbor communication rounds~\cite{xiao04}. Note that the training loss reported here is not scaled by $\textfrac{1}{N}$ to facilitate a fair comparison with standard CNN training methods (i.e. centralized training).

% ----- MNIST RESULTS ----- %
\begin{figure*}[h]

	\centering
	\setlength\tabcolsep{1pt}
	\begin{tabular}{cccc}
	
					\includegraphics[width=0.22\linewidth]{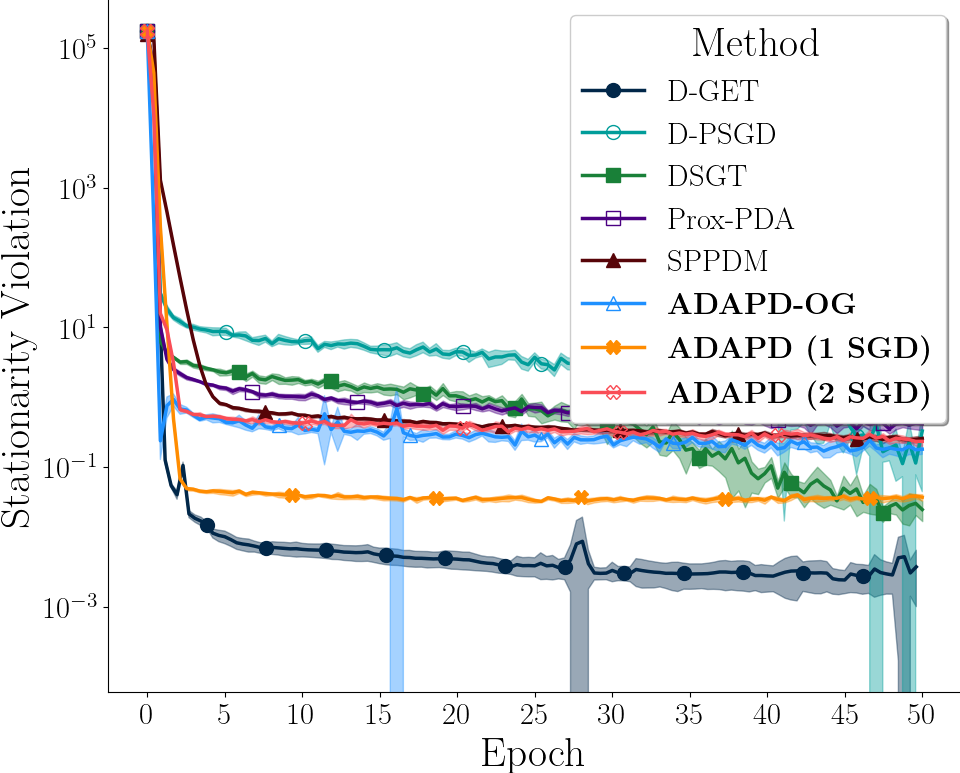} &

					\includegraphics[width=0.22\linewidth]{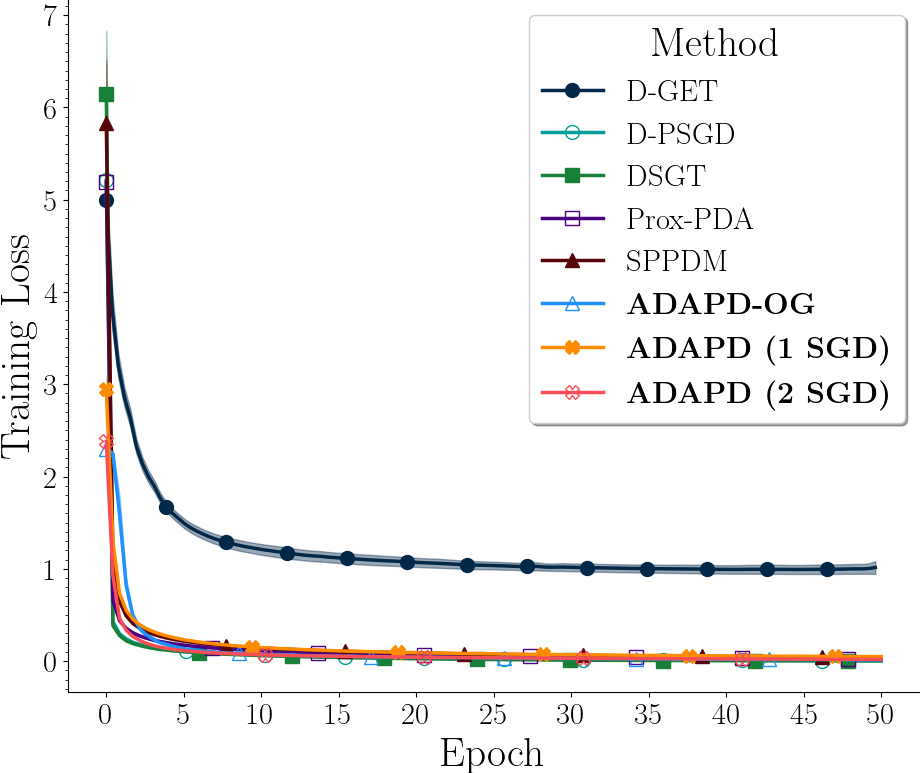} &

					\includegraphics[width=0.22\linewidth]{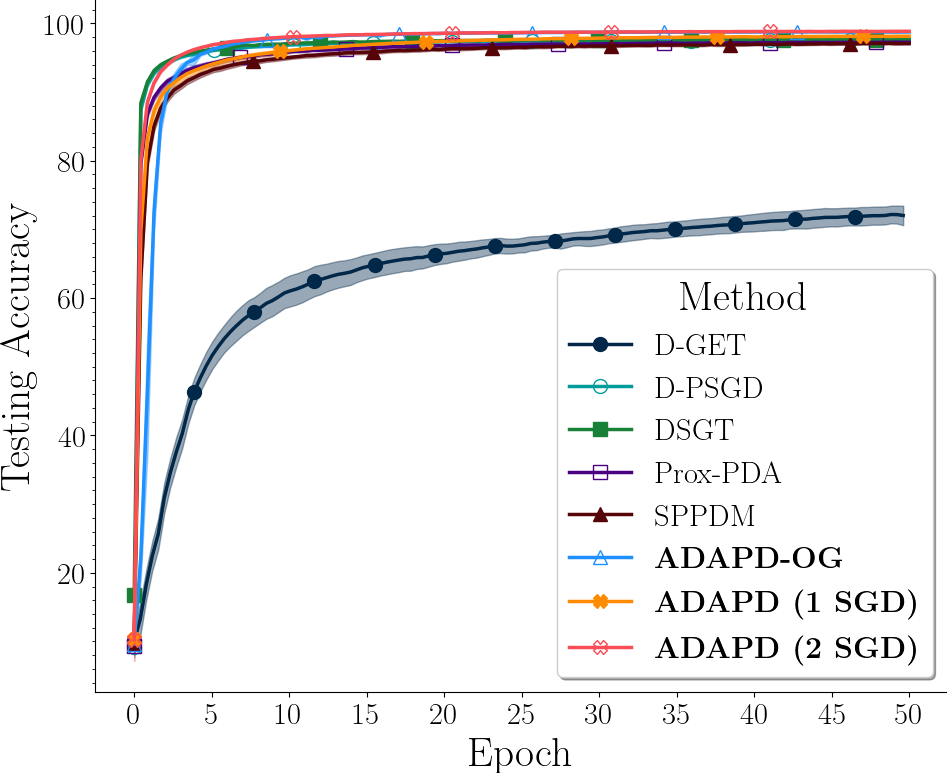} &

					\includegraphics[width=0.22\linewidth]{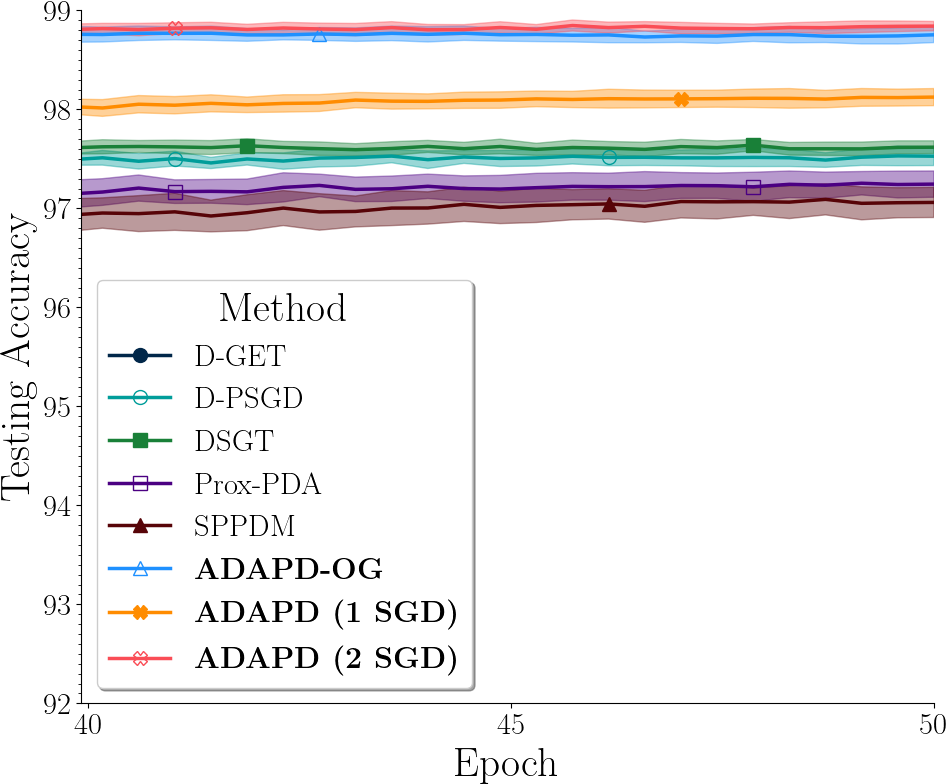} 

	\end{tabular}
		
		\caption{In order from left to right: stationarity violation, training loss, testing accuracy, and a zoomed version of the testing accuracy over the last ten epochs for the MNIST image classification problem.}
		
		\label{fig:mnist_results}
\end{figure*}

Additionally, we report the wall-clock time taken to reach and stay above 97\% testing accuracy for the MNIST image classification problem in Table~\ref{table:time_comparison_mnist}. This value comes from selecting the highest whole number of testing accuracy that most methods exceed. D-GET does not achieve this accuracy in the alloted amount of epochs. The ``Samples'' column indicates the amount of data visited by each agent to achieve the 97\% testing accuracy and the  ``Communications'' column indicates the corresponding number of communications performed by each agent (for D-GET, these values are simply the total numbers used during training). We also include each method's highest testing accuracy in the last column.

% ----- TABLE FOR TIME COMPARISON ---- %
\begin{table*}[h]
		\begin{center}
		{\small
			\begin{tabular}{lcccc}
			\toprule
			\multicolumn{1}{c}{\multirow{2}{*}{Method}} & \multicolumn{3}{c}{ To reach 97\% testing accuracy}  & \multirow{2}{*}{Highest accuracy (\%)}\\
			\cmidrule{2-4}
			& Time (s) & Samples & Communications &\\
			\midrule
			D-GET & \xmark & 376,524 & 4,096 & 72.17 \\
			D-PSGD & 31.56 & 92,800 & 1,450 & 97.53 \\
			DSGT & 26.78 & 73,600 & 2,300 & 97.64 \\
			Prox-PDA & 80.99 & 227,200 & 3,550 & 97.25 \\
			SPPDM & 116.73 & 326,400 & 5,100 & 97.09 \\
			ADAPD-OG & \textbf{16.35} & \textbf{48,000} & 750 & 98.77 \\
			ADAPD (1 SGD) & 74.2 & 121,600 & 1,900 & 98.12 \\
			ADAPD (2 SGD) & 47.5 & 83,200 & \textbf{650} & \textbf{98.85} \\
			\bottomrule
			\end{tabular}}
		\end{center}
		\caption{Time to reach 97\% testing accuracy on the MNIST image classification problem. Final column represents highest overall testing accuracy. Bold items indicate the best value.}
		\label{table:time_comparison_mnist}
\end{table*}

While D-GET is able to achieve the lowest stationarity violation, the training loss and testing accuracy indicate it does not converge to a solution that solves the classification problem well. Figure~\ref{fig:mnist_results} and Table~\ref{table:time_comparison_mnist} show that both ADAPD and ADAPD-OG outperform competitors in terms of testing accuracy, suggesting that ADAPD is able to find a solution that generalizes better than other methods. Additionally, Table~\ref{table:time_comparison_mnist} shows that ADAPD (with 2 SGD steps) and ADAPD-OG require far fewer communications to achieve a high testing accuracy. In a network setting where communication time dominates the computation time, ADAPD and its variants can outperform the competitors.

\subsubsection{CIFAR-10}

The second CNN experiment we perform is training the ALL-CNN model~\cite{springenberg15} on the CIFAR-10 dataset~\cite{krizhevsky09}. We add batch normalization after every ReLU activation function and perform no data augmentation prior to training. For these experiments, we fix the mini-batch size to 128 for all methods and limit the number of updates so that each method runs for 500 epochs. We only use the ADAPD algorithm with 1 stochastic gradient step for these experiments, but we tune the dual step-size in (\ref{y_update}) and (\ref{z_update}). In Figure~\ref{fig:cifar_results}, we report the same metrics as in the MNIST experiment.

% ----- CIFAR RESULTS ----- %
\begin{figure*}[h]

	\centering
	\setlength\tabcolsep{1pt}
	\begin{tabular}{cccc}

					\includegraphics[width=0.22\linewidth]{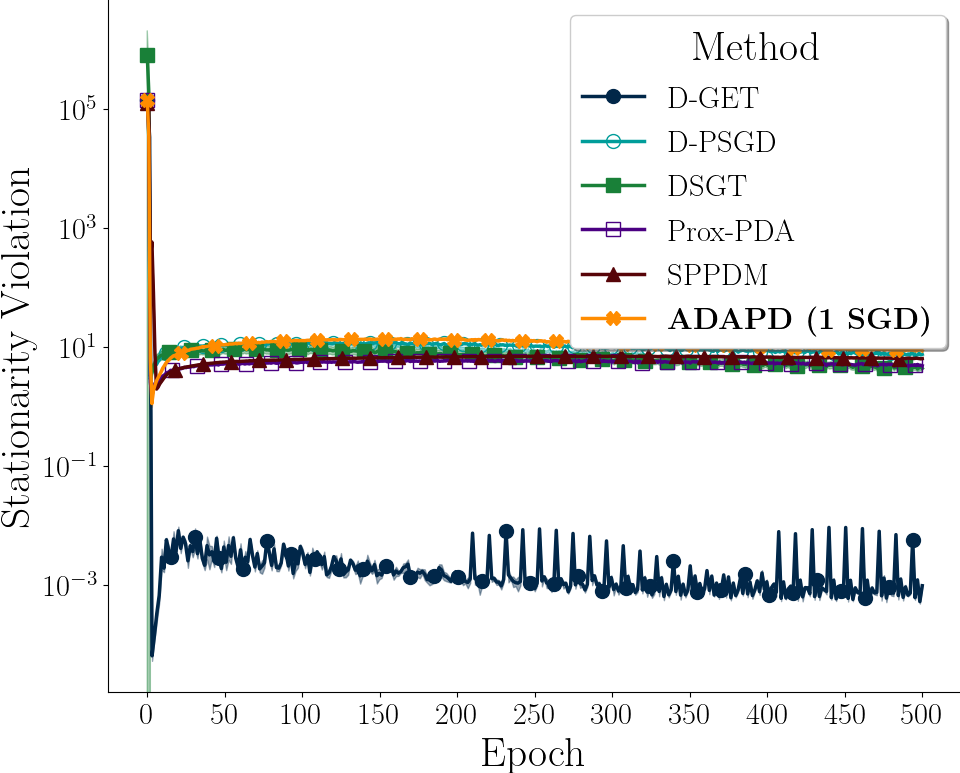} &

					\includegraphics[width=0.22\linewidth]{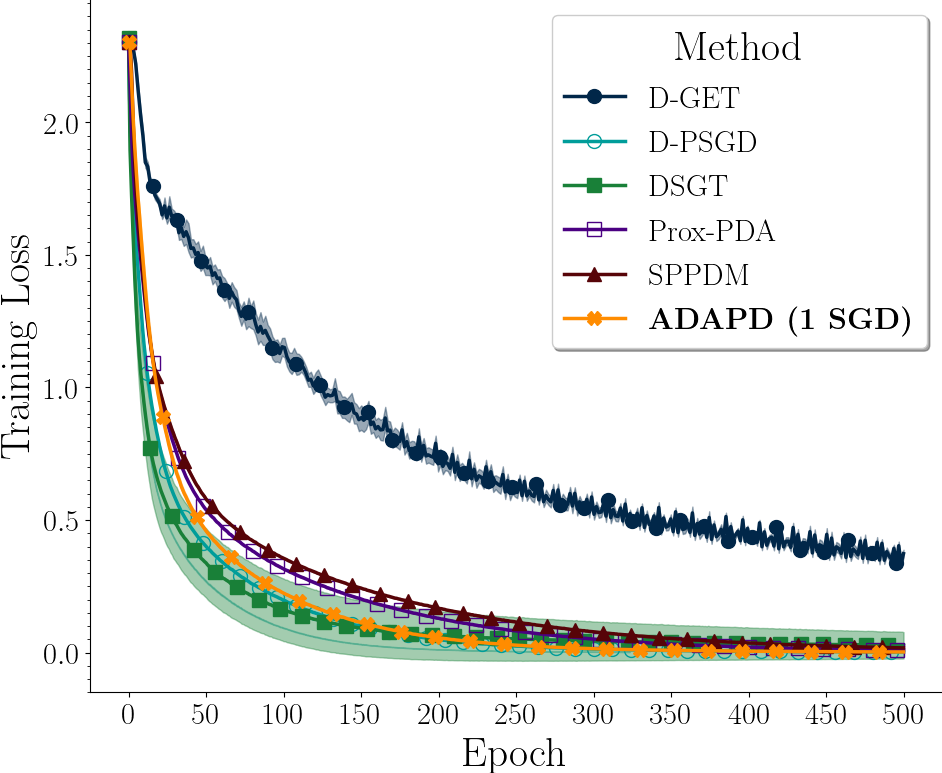} &

					\includegraphics[width=0.22\linewidth]{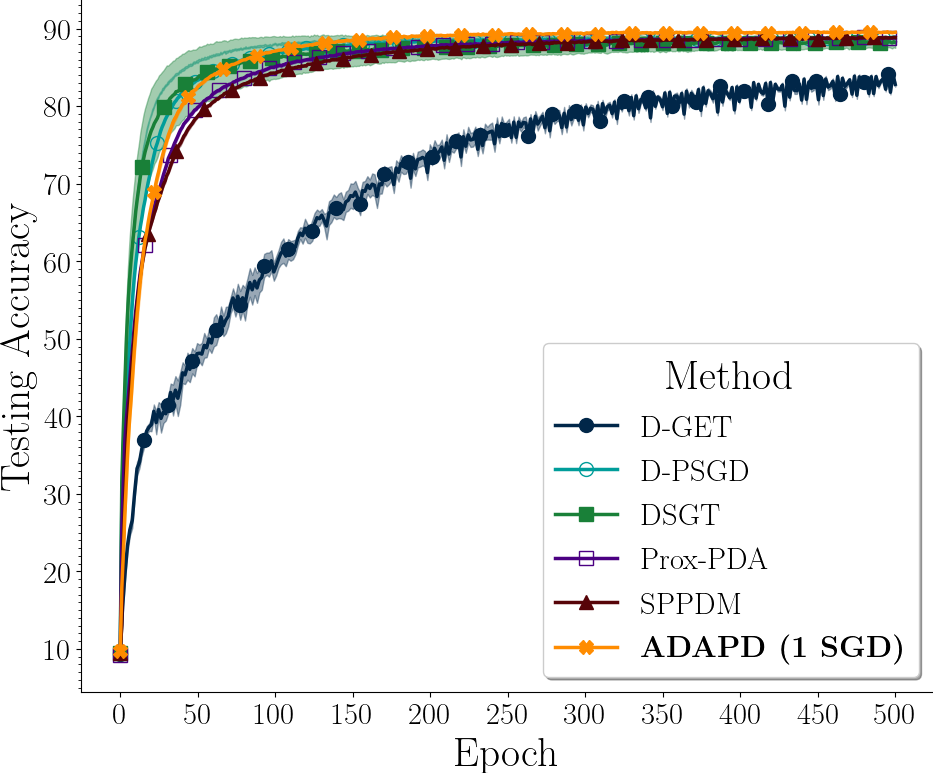} &

					\includegraphics[width=0.22\linewidth]{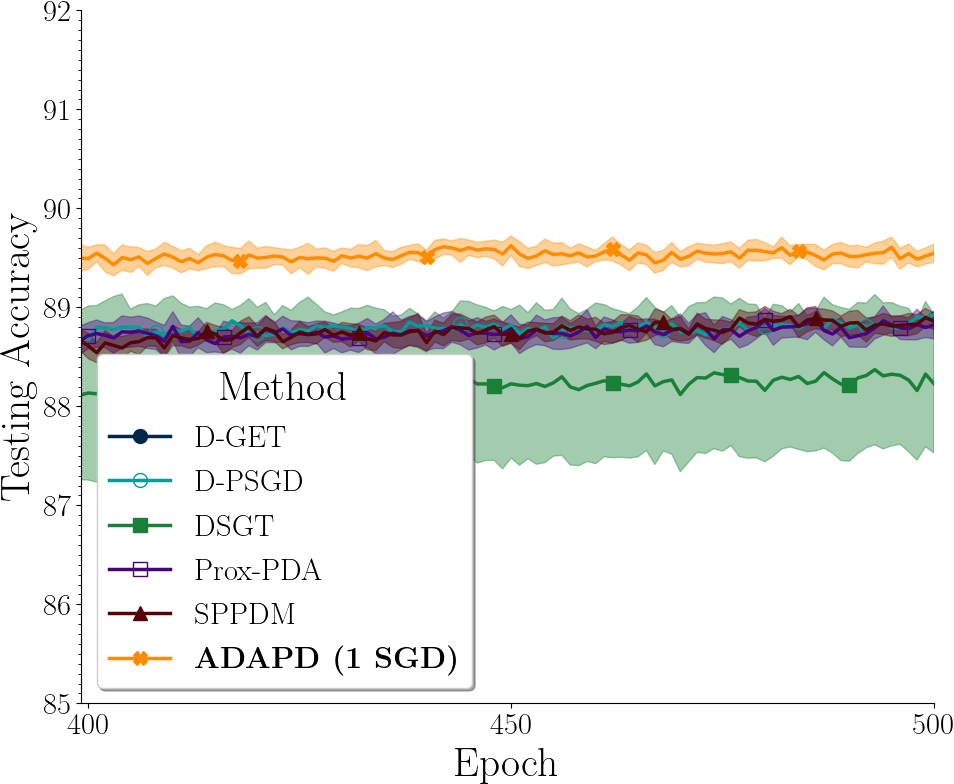}

	\end{tabular}
		
		\caption{In order from left to right: stationarity violation, training loss, testing accuracy, and a zoomed version of the testing accuracy over the last hundred epochs for the CIFAR-10 image classification problem.}
		
		\label{fig:cifar_results}
\end{figure*}

% ----- TABLE FOR TIME COMPARISON ---- %
\begin{table*}[h]
		\begin{center}
		{\small
			\begin{tabular}{lcccc}
			\toprule
			\multicolumn{1}{c}{\multirow{2}{*}{Method}} & \multicolumn{3}{c}{ To reach 88\% testing accuracy}  & \multirow{2}{*}{Highest accuracy (\%)}\\
			\cmidrule{2-4}
			& Time (s) & Samples & Communications &\\
			\midrule
			D-GET & \xmark & 3,125,006 & 35,530 & 84.16 \\
			D-PSGD & \textbf{651.88} & 1,011,200 & 7,900 & 88.92 \\
			DSGT & 1,900.23 & 2,348,800 & 36,700 & 88.37 \\
			Prox-PDA & 1,025.57 & 1,523,200 & 11,900 & 88.88 \\
			SPPDM & 1,395.84 & 1,708,800 & 13,350 & 88.91 \\
			ADAPD (1 SGD) & 870.11 & \textbf{806,400} & \textbf{6,300} & \textbf{89.62} \\
			\bottomrule
			\end{tabular}}
		\end{center}
		\caption{Time to reach 88\% testing accuracy on the CIFAR-10 image classification problem. Final column represents highest overall testing accuracy. Bold items indicate the best value.}
		\label{table:time_comparison_cifar}
\end{table*}

Similar to the MNIST image classification problem, we report the wall-clock time taken to reach and stay above 88\% testing accuracy in Table~\ref{table:time_comparison_cifar}. In terms of stationarity, all methods besides D-GET struggle. However, ADAPD performs better than the competitors in terms of testing accuracy (see Figure~\ref{fig:cifar_results} and Table~\ref{table:time_comparison_cifar}). Similar to the MNIST results, this suggests that ADAPD is able to find a solution to the image classification problem that generalizes better than the competitors.  Additionally, ADAPD greatly saves on the number of data samples and communications necessary to achieve a high testing accuracy.

% ------------------------------------------------------------------------ %
% 										Conclusions							  	     %
% ------------------------------------------------------------------------ %
\section{Conclusion}\label{sec:conclusion}

In this work, we presented ADAPD: \textbf{A} \textbf{D}ecentr\textbf{A}lized \textbf{P}rimal-\textbf{D}ual framework for solving non-convex and smooth consensus optimization problems over a network of agents. Two variants to ADAPD are presented, the ADAPD-OG (\textbf{O}ne \textbf{G}radient) and the ADAPD-MC (\textbf{M}ultiple \textbf{C}ommunications). We demonstrated that ADAPD and ADAPD-OG achieves $\bigO{L(1-\rho)^{-2}\varepsilon^{-1}}$ communication complexity to find an $\varepsilon$-stationary point and showed this can be reduced to $\bigO{L(1-\rho)^{-0.5}\varepsilon^{-1}}$ when the MC variant is used; this is optimal for the class of smooth, non-convex, decentralized consensus problems considered in this work. Finally, we presented four numerical experiments that validate our claim that ADAPD outperforms other state-of-the-art decentralized methods. Future research topics would be extending the theoretical guarantees of  ADAPD-OG to the stochastic case and demonstrating convergence in a time-varying/asynchronous setting of ADAPD and its variants.

% ------------------------------------------------------------------------ %
% 									Bibliography							  	     %
% ------------------------------------------------------------------------ %
{\small
\bibliographystyle{hieeetr}
\bibliography{sources.bib}
}

% ------------------------------------------------------------------------ %
% 										Appendices						  		    %
% ------------------------------------------------------------------------ %
\begin{appendices}

% ------------------------------------------------------------------------ %
% 									Chebyshev 								    %
% ------------------------------------------------------------------------ %
\section{Supporting Lemmas and proofs for the Chebyshev acceleration}\label{appendix:cheby}

\printProofs[cheby]

% ---------------------------
% Prox-PDA equivalence proof
% ---------------------------
\section{On the equivalence between Prox-PDA~\cite{hong17} and distributed ADMM~\cite{shi14}}\label{appendix:prox}

Here, we show that the distributed ADMM algorithm~\cite{shi14}, which uses edge-based constraints to enforce consensus (see (3) in~\cite{shi14}), can reduce to the Prox-PDA method in \cite{hong17}. Under the assumption of $\vec{A}^\top \vec{A}=\vec{L}^-$ (the signed Laplacian of $\graph$), Prox-PDA performs global updates:
    \begin{equation}\label{hong17:proxpda}
        \begin{split}
            \x^{(k+1)}=&\argmin_{\x}\left\{F(\x)-\ip{\boldsymbol{\mu}^{(k)},\vec{A}\x}+\frac{\beta}{2}H(\x,\x^{k};\vec{A},\vec{B})\right\},\\
            \boldsymbol{\mu}^{(k+1)}=&\boldsymbol{\mu}^{(k)}+\beta \vec{A}\x^{k+1},
        \end{split}
    \end{equation}
where $H(\x,\x^k;\vec{A},\vec{B})=\fronorm{\vec{A}\x}+\norm{\x-\x^k}_{\vec{B}^\top\vec{B}}^2$. Choosing $\vec{B}^\top\vec{B}=\vec{L}^+$ (the unsigned Laplacian of $\graph$), we have
    \begin{align}\label{proxpda_expansion}
        &H(\x,\x^k;\vec{A},\vec{B})\nonumber\\
        =&\frac{\beta}{2}\sum_{i=1}^N\left(2\abs{\mathcal{N}_i}\norm{\blx_i}_2^2-\ip{ \blx_i,\sum_{j\in\mathcal{N}_i}\blx_j}\right)+\frac{\beta}{2}\ip{\blx^{k},\vec{B}^\top \vec{B}\x^{k}}\nonumber\\
        &+\frac{\beta}{2}\sum_{i=1}^N\left(\ip{ \blx_i,\sum_{j\in\mathcal{N}_i}\blx_j}-2\abs{\mathcal{N}_i}\ip{\blx_i,\blx_i^{k}}-2\ip{\blx_i,\sum_{j\in\mathcal{N}_i}\blx_j^{k}}\right).\nonumber\\
    \end{align}
Multiplying both sides of the $\boldsymbol{\mu}$ update in~\eqref{hong17:proxpda} by $\vec{A}^\top$, letting $\boldsymbol{\alpha}^{k}\triangleq \vec{A}^\top \boldsymbol{\mu}^{k}\,\forall\,k$, and dropping the $\frac{\beta}{2}\ip{\blx^{k},\vec{B}^\top \vec{B}\x^{k}}$ term from~\eqref{proxpda_expansion} (as the $\argmin$ is about $\x$ in~\eqref{hong17:proxpda}) results in (10) from~\cite{shi14}. Hence, the two algorithms are equivalent. As a result, the distributed ADMM updates from~\cite{shi14} converge in the non-convex case by the convergence of Prox-PDA~\cite{hong17}.

% ------------------------------------------------------------------------ %
% 									Inexact Proofs								   %
% ------------------------------------------------------------------------ %
 \section{Supporting Lemmas and proofs for ADAPD}\label{appendix:algo}

 % ----- Print first proof ----- %
 \printProofs[dualbound]
 
 The next three Lemmas serve as building blocks to prove Lemma~\ref{lemma:al_decrease}.

\begin{lemma}\label{lemma:x_decrease}
If (\ref{x_optimality}) is satisfied and $\eta<\textfrac{1}{2L}$, then
{\small
	\begin{equation}\label{lemma:x_decrease:bound}
		\begin{split}
			&\L_\eta(\x^{k+1},\x_0^{k};\y^k,\z^k)-\L_\eta(\x^{k},\x_0^{k};\y^k,\z^k)\\
			&\le\textfrac{2L\eta-1}{2\eta}\fronorm{\x^{k+1}-\x^k}+\textfrac{\epsilon_{k+1}}{2L}, \forall\, k\ge0.
		\end{split}
	\end{equation}
}
\end{lemma}
\begin{proof} By (\ref{assumption:smoothness}), we have
{\small
	\begin{align*}
	&-F(\x^k)\\
	&\le-F(\x^{k+1})+\ip{-\nabla F(\x^{k+1}),\x^k-\x^{k+1}}+\textfrac{L}{2}\fronorm{\x^{k+1}-\x^k}.
	\end{align*}
	}
	Hence,
	{\small
	\begin{align*}
	&\L_\eta(\x^{k+1},\x_0^{k};\y^k,\z^k)-\L_\eta(\x^{k},\x_0^{k};\y^k,\z^k)\\
	&\le\ip{\nabla F(\x^{k+1}),\x^{k+1}-\x^k}-\textfrac{1}{2\eta}\fronorm{\x^k-\x_0^k}\\
		&\quad+\textfrac{L}{2}\fronorm{\x^{k+1}-\x^k}+\ip{\y^k,\x^{k+1}-\x^k}+\textfrac{1}{2\eta}\fronorm{\x^{k+1}-\x_0^k}\\
	&=\ip{\nabla F(\x^{k+1})+\y^k+\textfrac{1}{\eta}(\x^{k+1}-\x_0^k),\x^{k+1}-\x^k}\\
	&\quad+\textfrac{L}{2}\fronorm{\x^{k+1}-\x^k}-\textfrac{1}{2\eta}\fronorm{\x^{k+1}-\x^k}\\
	&\stack{(\ref{fact:peterpaulinequality})}{\le}\textfrac{1}{2L}\fronorm{\nabla F(\x^{k+1})+\y^k+\textfrac{1}{\eta}(\x^{k+1}-\x_0^k)}\\
		&\quad+\textfrac{2L\eta-1}{2\eta}\fronorm{\x^{k+1}-\x^k}\\
	&\le\textfrac{2L\eta-1}{2\eta}\fronorm{\x^{k+1}-\x^k}+\textfrac{\epsilon_{k+1}}{2L}
	\end{align*}
	}
	where in the last inequality, we have used $\fronorm{\nabla F(\x^{k+1})+\y^k+\textfrac{1}{\eta}(\x^{k+1}-\x_0^k)}\le\epsilon_{k+1}.$ 
\end{proof}

\begin{lemma}\label{lemma:x0_decrease}
	The partial gradient $\nabla_{\x_0}\L_\eta(\x,\x_0;\y,\z)$ is $\textfrac{3}{\eta}$-Lipschitz continuous about $\x_0$ for any $(\x,\y,\z)$. Further, for all $k\ge0,$ we have
	{\small
	\begin{equation}\label{lemma:x0_decrease:bound}
		\begin{split}
		&\L_\eta(\x^{k+1},\x_0^{k+1};\y^k,\z^k)-\L_\eta(\x^{k+1},\x_0^{k};\y^k,\z^k)\\
		&\le-\textfrac{1}{2\eta}\fronorm{\x_0^{k+1}-\x_0^k}.
		\end{split}
	\end{equation}
	}
\end{lemma}
\begin{proof} Compute
	{\small
	\begin{align*}
	&\norm{\nabla_{\x_0}\L_\eta(\x,\x_0;\y,\z)-\nabla_{\x_0}\L_\eta(\x,\x_0';\y,\z)}_F\\
	&=\norm{\textfrac{1}{\eta}(2\identity-\W)[\x_0-\x_0']}_F\le\textfrac{3}{\eta}\norm{\x_0-\x_0'}_F,
	\end{align*}
	}
	where we have used the compatibility of the 2-norm and the Frobenius norm and Assumption~\ref{assumption:mixing_matrix}(iv). Hence, it holds
	{\small
	\begin{align*}
	&\L_\eta(\x^{k+1},\x_0^{k+1};\y^k,\z^k)\\
	&\le\L_\eta(\x^{k+1},\x_0^{k};\y^k,\z^k)+\textfrac{3}{2\eta}\fronorm{\x_0^{k+1}-\x_0^k}\\
		&\quad+\ip{\nabla_{\x_0}\L_\eta(\x^{k+1},\x_0^{k};\y^k,\z^k),\x_0^{k+1}-\x_0^k}\\
	&\stackrel{(\ref{x0_update})}{=}\L_\eta(\x^{k+1},\x_0^{k};\y^k,\z^k)+(\textfrac{3}{2\eta}-\textfrac{2}{\eta})\fronorm{\x_0^{k+1}-\x_0^k}.
	\end{align*}
	}
	Rearranging terms gives the desired result.
\end{proof}

\begin{lemma}\label{lemma:dual_decrease}
	For all $k\ge0,$ the followings hold:
		\begin{align}
			\small&\L_\eta(\x^{k+1},\x_0^{k+1};\y^{k+1},\z^k)-\L_\eta(\x^{k+1},\x_0^{k+1};\y^k,\z^k)\nonumber\\
			&=\eta\fronorm{\y^{k+1}-\y^k},\label{lemma:dual_decrease:y_bound}\\
			&\L_\eta(\x^{k+1},\x_0^{k+1};\y^{k+1},\z^{k+1})-\L_\eta(\x^{k+1},\x_0^{k+1};\y^{k+1},\z^k)\nonumber\\
			&=\eta\fronorm{\z^{k+1}-\z^k}.\label{lemma:dual_decrease:z_bound}
		\end{align}
\end{lemma}
\begin{proof} By the $\y$ update (\ref{y_update}), we have
{\small
	\begin{align*}
		&\L_\eta(\x^{k+1},\x_0^{k+1};\y^{k+1},\z^k)-\L_\eta(\x^{k+1},\x_0^{k+1};\y^k,\z^k)\\
		&=\ip{\y^{k+1}-\y^k,\x^{k+1}-\x^{k+1}_0}\\
		&=\ip{\y^{k+1}-\y^k,\eta(\y^{k+1}-\y^k)}\\
		&=\eta\fronorm{\y^{k+1}-\y^k}
	\end{align*}
}
	for all $k\ge0.$ Hence, (\ref{lemma:dual_decrease:y_bound}) holds.
	Similarly, by the $\z$ update (\ref{z_update}), we have
	{\small
		\begin{align*}
			\small&\L_\eta(\x^{k+1},\x_0^{k+1};\y^{k+1},\z^{k+1})-\L_\eta(\x^{k+1},\x_0^{k+1};\y^{k+1},\z^k)\\
			&=\ip{\z^{k+1}-\z^k,\sqrt{\identity-\W}\x_0^{k+1}}\\
			&=\ip{\z^{k+1}-\z^k,\eta(\z^{k+1}-\z^k)}\\
			&=\eta\fronorm{\z^{k+1}-\z^k}
		\end{align*}
	}
		for all $k\ge0.$ Thus (\ref{lemma:dual_decrease:z_bound}) holds, and we complete the proof. 
\end{proof}

% ----- PRINT INEXACT PROOFS ----- %
\printProofs[buildLyapunov]

% ----- PRINT INEXACT PROOFS ----- %
\printProofs[inexactproofs]

% ------------------------------------------------------------------------ %
% 								Single Step Proofs							    %
% ------------------------------------------------------------------------ %
\section{Supporting Lemmas and proofs for ADAPD-OG}\label{appendix:single_step}
% ----- PRINT OG PROOFS ----- %
We begin the convergence analysis with showing the change in the augmented Lagrangian function value between two consecutive iterations.

% REDO X UPDATE
\begin{lemma}\label{lemma:ss:x_update}
		Provided that $\eta<\textfrac{1}{L}$, we have
		{\small
		\begin{equation}\label{lemma:ss:x_update:bound}
			\begin{split}
				&\L_\eta(\x^{k+1},\x_0^{k};\y^k,\z^k)-\L_\eta(\x^{k},\x_0^{k};\y^k,\z^k)\\
				&\le\textfrac{L\eta-1}{2\eta}\fronorm{\x^{k+1}-\x^k}
			\end{split}
		\end{equation}
		}for all $k\ge0.$
\end{lemma}

\begin{proof}
	By (\ref{assumption:smoothness}), we have
	{\small
		\begin{align*}
			&F(\x^{k+1})\\
			&\le F(\x^k)+\ip{\nabla F(\x^k),\x^{k+1}-\x^k}+\textfrac{L}{2}\fronorm{\x^{k+1}-\x^k}.
		\end{align*}
	}Thus,
	{\small
		\begin{align*}
			&\L_\eta(\x^{k+1},\x_0^{k};\y^k,\z^k)-\L_\eta(\x^{k},\x_0^{k};\y^k,\z^k)\\
			&\le\ip{\nabla F(\x^{k}),\x^{k+1}-\x^k}+\textfrac{L}{2}\fronorm{\x^{k+1}-\x^k}+\ip{\y^k,\x^{k+1}-\x^k}\\
				&+\textfrac{1}{2\eta}\fronorm{\x^{k+1}-\x_0^k}-\textfrac{1}{2\eta}\fronorm{\x^k-\x_0^k}\\
			&=\ip{\nabla F(\x^k)+\y^k+\textfrac{1}{\eta}(\x^{k+1}-\x_0^k),\x^{k+1}-\x^k}\\
				&+\textfrac{L\eta-1}{2\eta}\fronorm{\x^{k+1}-\x^k}\\
			&\stack{(\ref{x_single_step_update})}{=}\textfrac{L\eta-1}{2\eta}\fronorm{\x^{k+1}-\x^k}.
		\end{align*}
}\end{proof}

% AUGMENTED LAGRANGIAN DESCENT NOW
\begin{theoremEnd}[normal]{lemma}\label{lemma:ss:al_decrease}
	Let $\{(\x^{k},\x_0^{k};\y^{k},\z^{k})\}$ be obtained from Alg.~\ref{algo:single_step} or equivalently by updates \eqref{x_single_step_update} and \eqref{x0_update}-\eqref{z_update}. If $\eta<\textfrac{1}{L}$, then it holds for all $k\ge0$,
	{\small
		\begin{equation}\label{lemma:ss:al_decrease:bound}
			\begin{split}
				&\L_\eta(\x^{k+1},\x_0^{k+1};\y^{k+1},\z^{k+1})-\L_\eta(\x^{k},\x_0^{k};\y^{k},\z^{k})\\&\le\textfrac{L\eta-1}{2\eta}\fronorm{\x^{k+1}-\x^k}-\textfrac{1}{2\eta}\fronorm{\x_0^{k+1}-\x_0^k}\\
				&+\eta\fronorm{\y^{k+1}-\y^k}+\eta\fronorm{\z^{k+1}-\z^k}.
			\end{split}
		\end{equation}}
\end{theoremEnd}

\begin{proofEnd}\enskip
	The proof follows from rewriting $\L_\eta(\x^{k+1},\x_0^{k+1};\y^{k+1},\z^{k+1})-\L_\eta(\x^{k},\x_0^{k};\y^{k},\z^{k})$ as
	{\small
		\begin{align*}
			&\L_\eta(\x^{k+1},\x_0^{k+1};\y^{k+1},\z^{k+1})-\L_\eta(\x^{k},\x_0^{k};\y^{k},\z^{k})\\
			&=\L_{\eta}(\x^{k+1},\x_0^{k+1};\y^{k+1},\z^{k+1})-\L_{\eta}(\x^{k+1},\x_0^{k+1};\y^{k+1},\z^k)\\
				&+\L_{\eta}(\x^{k+1},\x_0^{k+1};\y^{k+1},\z^k)-\L_{\eta}(\x^{k+1},\x_0^{k+1};\y^{k},\z^k)\\
				&+\L_{\eta}(\x^{k+1},\x_0^{k+1};\y^{k},\z^k)-\L_{\eta}(\x^{k+1},\x_0^{k};\y^{k},\z^k)\\
				&+\L_{\eta}(\x^{k+1},\x_0^k;\y^k,\z^k)-\L_{\eta}(\x^k,\x_0^k;\y^k,\z^k)
		\end{align*}
	}and using (\ref{lemma:ss:x_update:bound}), (\ref{lemma:x0_decrease:bound}), (\ref{lemma:dual_decrease:y_bound}), and (\ref{lemma:dual_decrease:z_bound}) to bound each of the terms on the right hand side of the equality.
\end{proofEnd}

% SINGLE STEP DUAL VARIABLE BOUND
\begin{theoremEnd}[normal]{lemma}\label{lemma:ss:dual_var_bound}
	Under the assumptions of Lemma~\ref{lemma:ss:al_decrease}, it holds that for all $k\ge0$,
		{\small
		\begin{align}\label{lemma:ss:dual_var_bound:y_bound}
			\eta\fronorm{\y^{k+1}-\y^k}&\le2L^2\eta\fronorm{\x^k-\x^{k-1}}+\textfrac{2}{\eta}\fronorm{\vec{V}_0^k}, \\
		\label{lemma:ss:dual_var_bound:z_bound}
			\eta\fronorm{\z^{k+1}-\z^k}&\le\textfrac{4L^2\eta}{(1-\rho)}\fronorm{\x^k-\x^{k-1}}+\textfrac{6}{(1-\rho)\eta}\fronorm{\vec{V}_0^k},
		\end{align}
		}where $\vec{V}_0^k$ is defined in (\ref{v_vector}).
\end{theoremEnd}

\begin{proofEnd}\enskip
	To prove (\ref{lemma:ss:dual_var_bound:y_bound}), we have from (\ref{ss:dual_variable_relation:y_bound}) that
	{\small
		\begin{align*}
			\eta\fronorm{\y^{k+1}-\y^k}&=\eta\fronorm{-\nabla F(\x^k)+\nabla F(\x^{k-1})-\textfrac{1}{\eta}\vec{V}_0^k}\\
			&\stack{(\ref{fact:youngsinequality}),(\ref{assumption:smoothness})}{\le}2L^2\eta\fronorm{\x^k-\x^{k-1}}+\textfrac{2}{\eta}\fronorm{\vec{V}_0^k}.
		\end{align*}
	}To prove (\ref{lemma:ss:dual_var_bound:z_bound}), we start from \eqref{eq:bd-range-Z-mid} to have
	{\small
		\begin{align}\label{eq:bd-range-Z-OG}
						&\eta\fronorm{\sqrt{\identity-\W}\big(\z^{k+1}-\z^k\big)}\cr
						&\le2\eta\fronorm{\y^{k+1}-\y^k}+\textfrac{2}{\eta}\fronorm{\W\vec{V}_0^k}\cr
						&\stack{(\ref{lemma:ss:dual_var_bound:y_bound})}{\le}4L^2\eta\fronorm{\x^k-\x^{k-1}}+\textfrac{4}{\eta}\fronorm{\vec{V}_0^k}+\textfrac{2}{\eta}\fronorm{\W\vec{V}_0^k}\cr
						&\le4L^2\eta\fronorm{\x^k-\x^{k-1}}+\textfrac{6}{\eta}\fronorm{\vec{V}_0^k},
					\end{align}
	}where the last inequality uses Assumption~\ref{assumption:mixing_matrix}(iv). 	
	Now we notice that (\ref{lemma:dual_variable_relation:z_bound}) still holds for ADAPD-OG. Hence by choosing $\z^0\in\mathrm{range}\big(\sqrt{\identity-\W}\big),$ we have	 $\z^k\in\mathrm{range}\big(\sqrt{\identity-\W}\big)$ for all $k\ge0$ from \eqref{z_update_classic}. Thus \eqref{eq:bd-range-Z} still holds, and it together with \eqref{eq:bd-range-Z-OG} implies (\ref{lemma:ss:dual_var_bound:z_bound}).			
\end{proofEnd}

% SUPPORTING INEQUALITY LEMMA
\begin{theoremEnd}[normal]{lemma}\label{lemma:ss:supporing_inequality}
	For all $k\ge0,$ the following relation holds
	{\small
		\begin{equation}\label{lemma:ss:supporting_inequality:bound}
			\begin{split}
			&\textfrac{1}{2\eta}\left(\fronorm{\sqrt{\identity-\W}\x_0^{k+1}}+\fronorm{\sqrt{\identity-\W}(\x_0^{k+1}-\x_0^k)}-\fronorm{\sqrt{\identity-\W}\x_0^{k}}\right)\\
				&+\textfrac{1}{2\eta}\left(\norm{\vec{V}_0^k}_\W^2+\norm{\x_0^{k+1}-\x_0^k}_\W^2-\norm{\x_0^k-\x_0^{k-1}}_\W^2\right)\\
			&\le\textfrac{L\eta-1}{2\eta}\fronorm{\x_0^{k+1}-\x_0^k}+\textfrac{L}{2}\fronorm{\x^{k}-\x^{k-1}}+\textfrac{1}{2\eta}\fronorm{\x_0^k-\x_0^{k-1}}\\
			&-\textfrac{1}{2\eta}\fronorm{\vec{V}_0^k}
			\end{split}
		\end{equation}
	}where $\vec{V}_0^k$ is defined in (\ref{v_vector}).
\end{theoremEnd}

\begin{proofEnd}\enskip 
	By (\ref{ss:dual_variable_relation:y_bound}), we have
		{\small
		\begin{equation}\label{lemma:ss:supporting_inequality:eqn1}
		\begin{split}
			&\ip{\y^{k+1}-\y^k,\x_0^{k+1}-\x_0^{k}}\\
			&=\ip{-\nabla F(\x^{k})+\nabla F(\x^{k-1})-\textfrac{1}{\eta}\vec{V}_0^k,\x_0^{k+1}-\x_0^{k}}.
		\end{split}
		\end{equation}
		}We handle the left and right hand side separately. For the left hand side, we have
	{\small
		\begin{align*}
		&\ip{\y^{k+1}-\y^k,\x_0^{k+1}-\x_0^{k}}\\
		&\stack{(\ref{lemma:supporting_inequality:lhs_bound})}{=}\textfrac{1}{2\eta}\left(\fronorm{\sqrt{\identity-\W}\x_0^{k+1}}+\fronorm{\sqrt{\identity-\W}(\x_0^{k+1}-\x_0^k)}\right)\\
			&-\textfrac{1}{2\eta}\fronorm{\sqrt{\identity-\W}\x_0^{k}}\\
			&+\textfrac{1}{2\eta}\left(\norm{\vec{V}_0^k}_\W^2+\norm{\x_0^{k+1}-\x_0^k}_\W^2-\norm{\x_0^k-\x_0^{k-1}}_\W^2\right)
		\end{align*}
	}For the right hand side, we have
	{\small
		\begin{align*}
			&\ip{-\nabla F(\x^{k})+\nabla F(\x^{k-1})-\textfrac{1}{\eta}\vec{V}_0^k,\x_0^{k+1}-\x_0^{k}}\\
			&\stack{(\ref{fact:peterpaulinequality}),(\ref{assumption:smoothness})}{\le}\enskip\textfrac{L}{2}\left(\fronorm{\x^{k}-\x^{k-1}}+\fronorm{\x_0^{k+1}-\x_0^k}\right)-\textfrac{1}{\eta}\ip{\vec{V}_0^k,\x_0^{k+1}-\x_0^{k}}\\
			&=\textfrac{L}{2}\left(\fronorm{\x^{k}-\x^{k-1}}+\fronorm{\x_0^{k+1}-\x_0^k}\right)\\
				&-\textfrac{1}{2\eta}\left(\fronorm{\vec{V}_0^k}+\fronorm{\x_0^{k+1}-\x_0^k}-\fronorm{\x_0^k-\x_0^{k-1}}\right).
		\end{align*}
	}Combining like terms results in the right hand side of (\ref{lemma:ss:supporting_inequality:bound}); further using the equality established in (\ref{lemma:ss:supporting_inequality:eqn1}) completes the proof.
\end{proofEnd}

% SECOND AL DECREASE!!
\begin{theoremEnd}[normal]{lemma}\label{lemma:ss:second_al_decrease}
	Let $\{(\x^{k},\x_0^{k};\y^{k},\z^{k})\}$ be obtained from Alg.~\ref{algo:single_step} or equivalently by updates \eqref{x_single_step_update}, \eqref{x0_update}-\eqref{z_update}. If $\eta<\textfrac{1}{L}$, then
	{\small
		\begin{equation}\label{lemma:ss:second_al_decrease:bound}
			\begin{split}
				&\L_\eta(\x^{k+1},\x_0^{k+1};\y^{k+1},\z^{k+1})+\textfrac{\hat{C}}{2\eta}\fronorm{\sqrt{\identity-\W}\x_0^{k+1}}\\
				&+\textfrac{\hat{C}}{2\eta}\fronorm{\x_0^{k+1}-\x_0^k}\\
				&\le\L_\eta(\x^{k},\x_0^{k};\y^{k},\z^{k})+\textfrac{\hat{C}}{2\eta}\fronorm{\sqrt{\identity-\W}\x_0^{k}}+\textfrac{\hat{C}}{\eta}\fronorm{\x_0^k-\x_0^{k-1}}\\
				&+\left(\textfrac{L\eta-1}{2\eta}\right)\fronorm{\x^{k+1}-\x^k}+\left(\textfrac{\hat{C}L\eta-\hat{C}-1}{2\eta}\right)\fronorm{\x_0^{k+1}-\x_0^k}\\
				&+\textfrac{4L^2(1-\rho)\eta+8L^2\eta+\hat{C}L(1-\rho)}{2(1-\rho)}\fronorm{\x^k-\x^{k-1}}
			\end{split}
		\end{equation}
	}for all $k\ge0,$ where $\hat{C}$ is a fixed constant that satisfies $\hat{C}\ge\textfrac{12+4(1-\rho)}{(1-\rho)^2}.$
\end{theoremEnd}

\begin{proofEnd}\enskip
	By Lemmas \ref{lemma:ss:dual_var_bound} and (\ref{lemma:ss:al_decrease:bound}), we have
	{\small
		\begin{align*}
		&\L_\eta(\x^{k+1},\x_0^{k+1};\y^{k+1},\z^{k+1})-\L_\eta(\x^{k},\x_0^{k};\y^{k},\z^{k})\\
		&\le\textfrac{L\eta-1}{2\eta}\fronorm{\x^{k+1}-\x^k}-\textfrac{1}{2\eta}\fronorm{\x_0^{k+1}-\x_0^k}\\
			&+\textfrac{(2(1-\rho)+4)L^2\eta}{(1-\rho)}\fronorm{\x^k-\x^{k-1}}+\textfrac{2(1-\rho)+6}{(1-\rho)\eta}\fronorm{\vec{V}_0^k}.
		\end{align*}
	}Multiplying $\hat{C}>0$ to both sides of \eqref{lemma:ss:supporting_inequality:bound} and adding it to the above inequality gives
		{\small
		\begin{align*}
			&\L_\eta(\x^{k+1},\x_0^{k+1};\y^{k+1},\z^{k+1})+\textfrac{\hat{C}}{2\eta}\fronorm{\sqrt{\identity-\W}\x_0^{k+1}}\\
				&+\textfrac{\hat{C}}{2\eta}\left(\fronorm{\sqrt{\identity-\W}(\x_0^{k+1}-\x_0^k)}-\fronorm{\sqrt{\identity-\W}\x_0^{k}}\right)\\
					&+\textfrac{\hat{C}}{2\eta}\left(\norm{\vec{V}_0^k}_\W^2+\norm{\x_0^{k+1}-\x_0^k}_\W^2-\norm{\x_0^k-\x_0^{k-1}}_\W^2\right)\\
			&\le\L_\eta(\x^{k},\x_0^{k};\y^{k},\z^{k})+\textfrac{L\eta-1}{2\eta}\fronorm{\x^{k+1}-\x^k}-\textfrac{1}{2\eta}\fronorm{\x_0^{k+1}-\x_0^k}\\
				&+\textfrac{4L^2(1-\rho)\eta+8L^2\eta+\hat{C}L(1-\rho)}{2(1-\rho)}\fronorm{\x^k-\x^{k-1}}+\textfrac{\hat{C}}{2\eta}\fronorm{\x_0^k-\x_0^{k-1}}\\
				&+\textfrac{\hat{C}L\eta-\hat{C}}{2\eta}\fronorm{\x_0^{k+1}-\x_0^k}+\textfrac{12+4(1-\rho)-\hat{C}(1-\rho)}{2(1-\rho)\eta}\fronorm{\vec{V}_0^k}.
		\end{align*}
	}Since the minimum eigenvalue of $\identity+\W$ is $\rho_N>0$ in (\ref{spectral_gap}), it holds $\textfrac{12+4(1-\rho)}{(1-\rho)}\identity \preccurlyeq \hat C\left(\identity+\W\right)$ when $\hat C\ge\textfrac{12+4(1-\rho)}{(1-\rho)\rho_N}$. Furthermore, since $\frac{1}{1-\rho}\ge\frac{1}{\rho_N}$ for $\hat C\ge\textfrac{12+4(1-\rho)}{(1-\rho)^2}$, we have $0\le\textfrac{\hat C(1-\rho)-12-4(1-\rho)}{2(1-\rho)\eta}\fronorm{\vec{V}_0^k}+\textfrac{\hat C}{2\eta}\norm{\vec{V}_0^k}_\W^2$ by noticing
		{\small
			\begin{align}\label{lemma:c_bound-eq}
				&\textfrac{\hat C(1-\rho)-12-4(1-\rho)}{2(1-\rho)\eta}\fronorm{\vec{V}_0^k}+\textfrac{\hat C}{2\eta}\norm{\vec{V}_0^k}_\W^2\\
				&=\norm{\vec{V}_0^k}_{\textfrac{\hat C(1-\rho)-12-4(1-\rho)}{2(1-\rho)\eta}\identity+\textfrac{\hat C}{2\eta}\W}^2 \ge0.
			\end{align}
		}Thus,
	{\small
		\begin{align*}
				&\L_\eta(\x^{k+1},\x_0^{k+1};\y^{k+1},\z^{k+1})+\textfrac{\hat{C}}{2\eta}\fronorm{\sqrt{\identity-\W}\x_0^{k+1}}\\
				&+\textfrac{\hat{C}}{2\eta}\fronorm{\sqrt{\identity-\W}(\x_0^{k+1}-\x_0^k)}+\textfrac{\hat{C}}{2\eta}\norm{\x_0^{k+1}-\x_0^k}_\W^2\\
				&\le\L_\eta(\x^{k},\x_0^{k};\y^{k},\z^{k})+\textfrac{\hat{C}}{2\eta}\fronorm{\sqrt{\identity-\W}\x_0^{k}}+\textfrac{L\eta-1}{2\eta}\fronorm{\x^{k+1}-\x^k}\\
					&-\textfrac{1}{2\eta}\fronorm{\x_0^{k+1}-\x_0^k}+\textfrac{\hat{C}}{2\eta}\norm{\x_0^k-\x_0^{k-1}}_\W^2\\
					&+\textfrac{4L^2(1-\rho)\eta+8L^2\eta+\hat{C}L(1-\rho)}{2(1-\rho)}\fronorm{\x^k-\x^{k-1}}\\
					&+\textfrac{\hat{C}L\eta-\hat{C}}{2\eta}\fronorm{\x_0^{k+1}-\x_0^k}+\textfrac{\hat{C}}{2\eta}\fronorm{\x_0^k-\x_0^{k-1}}.
		\end{align*}
	}Combining like terms, simplifying
	{\small
		\begin{align*}
			&\textfrac{\hat{C}}{2\eta}\fronorm{\sqrt{\identity-\W}(\x_0^{k+1}-\x_0^k)}+\textfrac{\hat{C}}{2\eta}\norm{\x_0^{k+1}-\x_0^k}^2_{\W}\\
			&=\textfrac{\hat{C}}{2\eta}\norm{\x_0^{k+1}-\x_0^k}^2_{\identity-\W}+\textfrac{\hat{C}}{2\eta}\norm{\x_0^{k+1}-\x_0^k}^2_{\W}\\
			&=\textfrac{\hat{C}}{2\eta}\fronorm{\x_0^{k+1}-\x_0^k}
		\end{align*}
	}and noting that $\textfrac{\hat{C}}{2\eta}\norm{\x_0^{k}-\x_0^{k-1}}^2_{\W}\le\textfrac{\hat{C}}{2\eta}\fronorm{\x_0^{k}-\x_0^{k-1}}$ yields the desired result.
\end{proofEnd}

% ----- NEW LYAPUNOV ----- %
We now define a new Lyapunov function based on the results from Lemma~\ref{lemma:ss:second_al_decrease}. Fix $\hat{C}\triangleq\textfrac{16}{(1-\rho)^2}$ used in (\ref{lemma:ss:second_al_decrease:bound}) and define
{\small
	\begin{equation}\label{lyapunov_single_step}
	\begin{split} 	\hat{\Phi}^k&\triangleq\L_\eta(\x^{k},\x_0^{k};\y^{k},\z^{k})+\textfrac{\hat{C}}{2\eta}\fronorm{\sqrt{\identity-\W}\x_0^{k}}+\textfrac{\hat{C}}{\eta}\fronorm{\x_0^k-\x_0^{k-1}}\\
	 	&+\textfrac{4L^2(1-\rho)\eta+8L^2\eta+\hat{C}L(1-\rho)}{2(1-\rho)}\fronorm{\x^k-\x^{k-1}}.
	\end{split}
	\end{equation}
	}Before showing that this Lyapunov function is lower bounded, we first demonstrate the relation between the Lyapunov function at two consecutive iterations. Using Lemma~\ref{lemma:ss:second_al_decrease}, for all $k\ge0$, we have
{\small
	\begin{equation}\label{ss:lyapunov:relation}
	\begin{split}
			&\hat{\Phi}^{k+1}+\left(\textfrac{(1-\rho)-(\hat{C}+1)L(1-\rho)\eta-((1-\rho)+2)4L^2\eta^2}{2(1-\rho)\eta}\right)\fronorm{\x^{k+1}-\x^k}\\
			&+\left(\textfrac{1-\hat{C}L\eta}{2\eta}\right)\fronorm{\x_0^{k+1}-\x_0^k}\le\hat{\Phi}^k
	\end{split}
	\end{equation}
}which comes directly from adding and subtracting $\textfrac{4L^2(1-\rho)\eta+8L^2\eta+\hat{C}L(1-\rho)}{2(1-\rho)}\fronorm{\x^{k+1}-\x^{k}}$ to the left hand side of (\ref{lemma:ss:second_al_decrease:bound}), combining like terms, and using (\ref{lyapunov_single_step}). We now show that $\hat{\Phi}^k$ has a finite lower bound via the following proposition.

\begin{theoremEnd}[normal]{proposition}\label{prop:ss:lyapunov_lower_bound}
	Under Assumptions \ref{assumption:mixing_matrix} and \ref{assumption:objective_function}, let $\{(\x^k,\x_0^k;\y^k,\z^k)\}$ be obtained from Alg.~\ref{algo:single_step} or equivalently by \eqref{x_single_step_update} and \eqref{x0_update}-\eqref{z_update}. Choose $\hat{C}$ and $\eta$ such that
	{\small
		\begin{equation}\label{prop:ss:constants_bound}
		\small\hat{C}\triangleq\textfrac{16}{(1-\rho)^2}\text{ and }
		\eta<\textfrac{1}{2\hat{C}L}.
		\end{equation}
	}Then the Lyapunov function \eqref{lyapunov_single_step} is uniformly lower bounded. More specifically, for all $k\ge0,$
		\begin{equation}\label{prop:ss:lower_bound}
			\small\hat{\Phi}^k\ge\underline{f}-1>-\infty,
		\end{equation}
	where $\underline{f}$ is defined in Assumption~\ref{assumption:objective_function}.
\end{theoremEnd}

\begin{proofEnd}\enskip
	First, we have $\L_\eta(\x^{k},\x_0^{k};\y^{k},\z^{k})\le\hat{\Phi}^k$ for all $k,$ by the definition of $\hat{\Phi}^k$ in \eqref{lyapunov_single_step}. Second, by the definition of $\underline{f}$ in \eqref{assumption:lower_bound}, we have for any integer number $K\ge1$,
	{\small
		\begin{align}\label{prop:ss:eqn1}
			&\sum_{k=0}^{K-1}\left(\hat{\Phi}^{k+1}-\underline{f}\right)\ge\sum_{k=0}^{K-1}\left(\L_\eta(\x^{k+1},\x_0^{k+1};\y^{k+1},\z^{k+1})-\underline{f}\right)\cr
			&\stack{\eqref{eq:lower-bd-phik-sum}}{\ge}-\textfrac{\eta}{2}\fronorm{\y^0}-\textfrac{\eta}{2}\fronorm{\z^0}. 
		\end{align}
	}
	Thirdly, by \eqref{ss:lyapunov:relation} and the choice of $\hat{C}$ and $\eta$, it holds that
		\begin{equation}\label{prop:ss:eqn2}
			\small\hat{\Phi}^{k+1}\le\hat{\Phi}^k.
		\end{equation}
	Now assume that there exists a $k_0\ge0$ such that $\hat{\Phi}^{k_0}-\underline{f}<-1.$ Then \eqref{prop:ss:eqn2} gives $\hat{\Phi}^k-\underline{f}\le\hat{\Phi}^{k_0}-\underline{f}<-1$ for all $k\ge k_0.$ Hence, $\sum_{k=k_0+1}^{\infty}\left(\hat{\Phi}^k-\underline{f}\right)=-\infty$ which contradicts \eqref{prop:ss:eqn1}. Therefore, we conclude that $\hat{\Phi}^{k}-\underline{f}\ge-1$ for all $k\ge0$ and complete the proof.
\end{proofEnd}

We are now in position to prove the convergence rate results of Alg.~\ref{algo:single_step}.

\begin{theoremEnd}[normal]{thm}\label{therom:ss:convergence}
	Under the same conditions assumed in Proposition~\ref{prop:ss:lyapunov_lower_bound}, it holds that
		\begin{equation}\label{theorem:ss:convergence:eqn1}
				\small\textfrac{\hat{C}_1}{K}\sum_{k=0}^{K-1}\left(\fronorm{\x^{k+1}-\x^k}+\fronorm{\x_0^{k+1}-\x_0^k}\right)\le\textfrac{\Delta_{\hat{\Phi}}}{K}
			\end{equation}
	where $\Delta_{\hat{\Phi}}\triangleq\hat{\Phi}^0-\underline{f}+1$ and
		\begin{equation}\label{theorem:ss:convergence:c_bound}
		\small\hat{C}_1\triangleq\textfrac{L}{(1-\rho)^2}\le\textfrac{(1-\rho)-(\hat{C}+1)L(1-\rho)\eta-((1-\rho)+2)4L^2\eta^2}{2(1-\rho)\eta}
		\end{equation}
\end{theoremEnd}

\begin{proofEnd}\enskip
	Summing up \eqref{ss:lyapunov:relation} from $k=0$ to $K-1$ and dividing by $K$ results in
	{\small
		\begin{align*}
			&\left(\textfrac{(1-\rho)-(\hat{C}+1)L(1-\rho)\eta-((1-\rho)+2)4L^2\eta^2}{2(1-\rho)\eta}\right)\textfrac{1}{K}\sum_{k=0}^{K-1}\fronorm{\x^{k+1}-\x^k}\\
				&+\left(\textfrac{1-\hat{C}L\eta}{2\eta}\right)\textfrac{1}{K}\sum_{k=0}^{K-1}\fronorm{\x_0^{k+1}-\x_0^k}\\
			&\stack{(\ref{prop:ss:lower_bound})}{\le}\quad\textfrac{\hat{\Phi}^0-\underline{f}+1}{K}.
		\end{align*}
	}By the choice of $\hat{C}$ and $\eta$, the $\hat{C}_1$ defined in \eqref{theorem:ss:convergence:c_bound} is positive, and the above inequality implies the desired result.
\end{proofEnd}

% Print the theorem proof
\printProofs[singlestep]

% ------------------------------------------------------------------------ %
% 									Inexact Proofs								   %
% ------------------------------------------------------------------------ %
 \section{Proofs of the complexity results}\label{appendix:complexity}
 
 \printProofs[complexity]

\end{appendices}

% End of document
\end{document}

%% file: macros_ieee.tex
\usepackage[framemethod=tikz]{mdframed}
\usepackage[procnames]{listings}
\usepackage{epstopdf}
\usepackage{setspace}
\usepackage{matlab-prettifier}
\usepackage{url}
\usepackage{cite}
\usepackage{epsfig,epsf,fancybox}
\usepackage{amsmath}
\usepackage{mathtools}
\usepackage{mathrsfs}
\usepackage{amsfonts}
\usepackage{graphicx}
\usepackage{color}
\usepackage{multicol}
\usepackage[linktocpage,pagebackref,colorlinks,linkcolor=black,anchorcolor=black,citecolor=black,urlcolor=black,hypertexnames=false]{hyperref}
\usepackage{boxedminipage}
\usepackage{stmaryrd}
\usepackage{multirow}
\usepackage{booktabs}
\usepackage{cite}
\usepackage{float}
\usepackage[ruled,vlined,linesnumbered]{algorithm2e}
\usepackage{algpseudocode}
\usepackage{comment}
\usepackage{bm}
\usepackage{enumerate}
\usepackage{pifont}% http://ctan.org/pkg/pifont
\usepackage{tablefootnote}

\newcommand{\R}{\mathbb{R}}

\newcommand{\Exp}{\mathbb{E}}
\newcommand{\basis}{\mathtt{\mathbf{e}}}

\newcommand{\blx}{\mathtt{\mathbf{x}}}
\newcommand{\x}{\mathtt{\mathbf{X}}}
\newcommand{\y}{\mathtt{\mathbf{Y}}}
\newcommand{\z}{\mathtt{\mathbf{Z}}}
\newcommand{\w}{\mathtt{\mathbf{W}}}

\newcommand{\identity}{\mathtt{\mathbf{I}}}
\newcommand{\vecx}{\mathtt{\mathbf{x}}}

\newcommand{\0}{\mathtt{\mathbf{0}}}
\renewcommand{\L}{\mathcal{L}}

\newcommand{\W}{\mathtt{\mathbf{W}}}

\newcommand{\xmark}{\ding{55}}
\newcommand{\graph}{\mathcal{G}}
\newcommand{\edges}{\mathcal{E}}
\newcommand{\vertices}{\mathcal{V}}

\newcommand{\mymatrix}[1]{\begin{bmatrix}#1\end{bmatrix}}
\newcommand{\norm}[1]{\left\Vert #1 \right\Vert}
\newcommand{\abs}[1]{\left\lvert #1\right\rvert}
\newcommand{\ip}[1]{\left\langle #1\right\rangle}

\DeclareMathOperator*{\argmin}{argmin}

\newcommand{\bigO}[1]{\mathcal{O}\left(#1\right)}
\newcommand{\fronorm}[1]{\left\Vert#1\right\Vert_F^2}
\renewcommand{\vec}[1]{\mathtt{\mathbf{#1}}}
\newcommand{\stack}[2]{\stackrel{\mathclap{#1}}{#2}}

\newcommand{\defined}{\triangleq}

\definecolor{keywords}{RGB}{255,0,90}
\definecolor{comments}{RGB}{0,0,113}
\definecolor{red}{RGB}{200,0,0}
\definecolor{green}{RGB}{0,100,0}
\definecolor{light-gray}{gray}{0.95}
\lstdefinestyle{mypython}{
	language=Python, 
	basicstyle=\mlttfamily\small,
	backgroundcolor=\color{light-gray},
	numbers=left,
	numberstyle=\tiny,
	keywordstyle=\color{red},
	commentstyle=\color{green},
	stringstyle=\color{green},
	showstringspaces=false,
	identifierstyle=\color{black},
	procnamekeys={def,class},
	frame=single,
	breaklines=true,
	columns=fullflexible,
	postbreak=\mbox{\textcolor{red}{$\hookrightarrow$}\space}
}

\lstdefinestyle{mystyle}{
	breaklines=true,
	numbers=left,
	numberstyle=\small,
	numbersep=8pt,
	style=Matlab-editor,
	basicstyle=\mlttfamily\small,
	emph=[1]{for,end,break},emphstyle=[1]\color{red},
	xleftmargin=15pt,
	framexleftmargin=20pt, % Increase me if the lines of code are greater than 100
	aboveskip=10pt,
	belowskip=10pt,
	frame=single,
	breaklines=true,
	columns=fullflexible,
	postbreak=\mbox{\textcolor{red}{$\hookrightarrow$}\space}
}